\documentclass[10pt]{article}
\usepackage{amsmath,amssymb,enumerate,mathrsfs}

\usepackage[latin1]{inputenc}
  \usepackage{color,eucal}
\usepackage{graphicx}
\usepackage{psfrag, paralist}
\numberwithin{equation}{section} \numberwithin{figure}{section}

\definecolor{ddmagenta}{rgb}{0.7,0,1.0}
\definecolor{ddcyan}{rgb}{0,0.1,1.0}
\definecolor{dred}{rgb}{.8,0,0}
\definecolor{ddgreen}{rgb}{0,0.4,0.4}

\newcommand{\bele}{\begin{lemm}}
\newcommand{\enle}{\end{lemm}}
\newcommand{\bedef}{\begin{defi}}
\newcommand{\bete}{\begin{teor}}
\newcommand{\eddef}{\end{defi}}
\newcommand{\ente}{\end{teor}}
\newcommand{\beos}{\begin{remark}}
\newcommand{\eddos}{\end{remark}}
\newcommand{\bepr}{\begin{prop}}
\newcommand{\empr}{\end{prop}}
\newcommand{\bepro}{\begin{prob}}
\newcommand{\empro}{\end{prob}}
\newcommand{\bede}{\begin{defin}}
\newcommand{\edde}{\end{defin}}
\newcommand{\beco}{\begin{coro}}
\newcommand{\enco}{\end{coro}}

\newcommand{\beeq}[1]{\begin{equation}
 \label{#1}}
\newcommand{\eddeq}{\end{equation}}

\newcommand{\beeqa}[1]{\begin{eqnarray}
  \label{#1}}
\newcommand{\eddeqa}{\end{eqnarray}}

\newcommand{\beal}[1]{\begin{align}
 \label{#1}}
\newcommand{\eddal}{\end{align}}

\newcommand{\bespl}[1]{\begin{split}
 \label{#1}}
\newcommand{\edspl}{\end{split}}

\newcommand{\bega}[1]{\begin{gather}
 \label{#1}}
\newcommand{\edga}{\end{gather}}

\newcommand{\beeqax}{\begin{eqnarray*}}
\newcommand{\eddeqax}{\end{eqnarray*}}

\newcommand{\no}{\nonumber}

\newcommand{\tensore}{\varepsilon({\bf u})}
\newcommand{\tensoret}{\varepsilon(\partial_t \uu)}

\numberwithin{equation}{section}

\newcommand{\teta}{\vartheta}

\newcommand{\dt}{\partial_t}

\newcommand{\uu}{\mathbf{u}}

\newcommand{\ww}{\mathbf{w}}
\newcommand{\vv}{\mathbf{v}}

\newcommand{\eeta}{{\mbox{\boldmath$\eta$}}}

\newcommand{\mmu}{{\mbox{\boldmath$\mu$}}}

\newcommand{\eps}{\varepsilon}
\newcommand{\epsi}{\varepsilon}

\newcommand{\weak}{\rightharpoonup}
\newcommand{\weakstar}{\mathop{\rightharpoonup}^{*}}

 \DeclareMathOperator{\dive}{div}

\let\TeXchi\chi
\def\chi{{\setbox0 \hbox{\mathsurround0pt
$\TeXchi$}\hbox{\raise\dp0 \copy0 }}}

\newtheorem{maintheorem}{Theorem}
\newtheorem{theorem}{Theorem}[section]

\newtheorem{lemma}{Lemma}[section]
\newtheorem{proposition}[lemma]{Proposition}
\newtheorem{remark}[lemma]{Remark}%
\newtheorem{problem}[lemma]{Problem}
\newtheorem{notation}[lemma]{Notation}%
\newtheorem{example}[lemma]{Example}

\begin{document}

\newcommand{\Q}{\Omega \times (0,T)}
\newcommand{\Qt}{\Omega \times (0,t)}
\newcommand{\Qtau}{\Omega \times (0,\tau)}
\newcommand{\til}{\widetilde}
\renewcommand{\part}{\partial_t}
\renewcommand{\teta}{\vartheta}
\newcommand{\accaunoz}{H^1_0(\Omega)}
\newcommand{\accadue}{H^2(\Omega)}
\newcommand{\nnu}{\nonumber}
\newcommand{\irre}{\rho}
\newcommand{\vinc}{\beta}
\newcommand{\weaksto}{{\rightharpoonup^*}}
\newcommand{\weakto}{\rightharpoonup}
\newcommand{\debole}{\,\weak\,}
\newcommand{\debolestar}{\,\weakstar\,}

\newcommand{\pairing}[4]{ \sideset{_{ #1 }}{_{ #2 }}  {\mathop{\langle #3 , #4
\rangle}}}

\newcommand{\comments}[1]{\marginpar{\tiny\textit{#1}}}

 \def\fin{\hfill
         \trait .3 5 0
         \trait 5 .3 0
         \kern-5pt
         \trait 5 5 -4.7
         \trait 0.3 5 0
 \medskip}
 \def\trait #1 #2 #3 {\vrule width #1pt height #2pt depth #3pt}
\newcommand{\forae}{\text{for a.e.}}
\newcommand{\aein}{\text{a.e.\ in}}

\newcommand{\down}{\downarrow}
\newcommand{\up}{\to}

\newcommand{\R}{\Bbb{R}}
\newcommand{\N}{\Bbb{N}}
\newcommand{\Z}{\Bbb{Z}}
\newcommand{\C}{\Bbb{C}}
\newcommand{\M}{\Bbb{M}}
\newcommand{\F}{\Bbb{F}}

\newcommand{\piecewiseConstant}[2]{\overline{#1}_{\kern-1pt#2}^\eps}
\newcommand{\pwC}{\piecewiseConstant}
\newcommand{\underlinepiecewiseConstant}[2]{\underline{#1}_{\kern-1pt#2}^\eps}
\newcommand{\upwC}{\underlinepiecewiseConstant}

\newcommand{\piecewiseLinear}[2]{{#1}_{\kern-1pt#2}^\eps}
\newcommand{\pwL}{\piecewiseLinear}
\newcommand{\pwM}[2]{\widetilde{#1}_{\kern-1pt#2}}
 \def\trait #1 #2 #3 {\vrule width #1pt height #2pt depth #3pt}
\newcommand{\pwN}[2]{#1_{\kern-1pt#2}}
 \def\trait #1 #2 #3 {\vrule width #1pt height #2pt depth #3pt}

\newcommand{\uk}{\pwN {\uu^k}{\tau}}
\newcommand{\un}{\pwN {\uu^n}{\tau}}
\newcommand{\Fn}{\pwN {J^n}{\tau}}
\newcommand{\uku}{\pwN {\uu^{k-1}}{\tau}}
\newcommand{\unu}{\pwN {\uu^{n-1}}{\tau}}
\newcommand{\chin}{\pwN {\chi^n}{\tau}}
\newcommand{\Fun}{\pwN {\mathbf{F}^n}{\tau}}
\newcommand{\gun}{\pwN {\mathbf{g}^n}{\tau}}
\newcommand{\chinuno}{\pwN {\chi^{n+1}}{\tau}}
\newcommand{\dom}{\text{dom}}

\newcommand{\uuh}{\widehat{\uu}}
\newcommand{\chih}{\widehat{\chi}}
\newcommand{\eetah}{\widehat{\eeta}}
\newcommand{\xih}{\widehat{\xi}}
\newcommand{\zetah}{\widehat{\zeta}}

\newcommand{\FG}[1]{\mathbf{#1}}
\def\bsA{{\FG A}} \def\bsB{{\FG B}} \def\bsC{{\FG C}}
\def\bsD{{\FG D}} \def\bsE{{\FG E}} \def\bsF{{\FG F}}
\def\bsG{{\FG G}} \def\bsH{{\FG H}} \def\bsI{{\FG I}}
\def\bsJ{{\FG J}} \def\bsK{{\FG K}} \def\bsL{{\FG L}}
\def\bsM{{\FG M}} \def\bsN{{\FG N}} \def\bsO{{\FG O}}
\def\bsP{{\FG P}} \def\bsQ{{\FG Q}} \def\bsR{{\FG R}}
\def\bsS{{\FG S}} \def\bsT{{\FG T}} \def\bsU{{\FG U}}
\def\bsV{{\FG V}} \def\bsW{{\FG W}} \def\bsX{{\FG X}}
\def\bsY{{\FG Y}} \def\bsZ{{\FG Z}}


\newcommand{\uun}{u_{\mathrm{N}}}
\newcommand{\whuun}{\widehat{u}_{\mathrm{N}}}
\newcommand{\uuni}[1]{(u_{#1})_{\mathrm{N}}}
\newcommand{\uut}{\uu_{\mathrm{T}}}
\newcommand{\wwt}{\ww_{\mathrm{T}}}
\newcommand{\dotu}{\partial_t{\uu}}
\newcommand{\dotui}[1]{\partial_t{\uu}_{#1}}
\newcommand{\dotut}{\partial_t{\uu}_{\mathrm{T}}}
\newcommand{\vvn}{v_{\mathrm{N}}}
\newcommand{\vvt}{\vv_{\mathrm{T}}}
\newcommand{\nn}{\mathbf{n}}
\newcommand{\zz}{\mathbf{z}}
\newcommand{\dotv}{\partial_t{\vv}}
\newcommand{\dotvt}{\partial_t{\vv}_{\mathrm{T}}}
\newcommand{\ssigma}{{\mbox{\boldmath$\sigma$}}}
\newcommand{\sigman}{\sigma_{\mathrm{N}}}
\newcommand{\sigmat}{\ssigma_{\mathrm{T}}}
\newcommand{{\gc}}{{\Gamma_{\mathrm{c}}}}
\newcommand{\cn}{c_{\mathrm{N}}}
\newcommand{\ct}{c_{\mathrm{T}}}
\newcommand{\Cweak}{\mathrm{C}^0_{\mathrm{weak}}}
\newcommand{\Reg}{\mathcal{R}}
\newcommand{\dd}{\,\mathrm{d}}
\newcommand{\foraa}{\text{for a.a.}}
\newcommand{\funeta}[1]{\mathcal{J}_{#1}}
\newcommand{\plfuneta}[1]{\partial\mathcal{J}_{#1}}

\newcommand{\uue}{\uu_{\eps}}
\newcommand{\ud}[1]{\uu_{#1}^{\eps}}
\newcommand{\chid}[1]{\chi_{#1}^{\eps}}
\newcommand{\etad}[1]{\eeta_{#1}^{\eps}}
\newcommand{\zud}[1]{\zz_{#1}^{\eps}}
\newcommand{\mud}[1]{\mmu_{#1}^{\eps}}
\newcommand{\xid}[1]{\xi_{#1}^{\eps}}
\newcommand{\udno}[1]{(\uu_{#1}^{\eps})_{\mathrm{N}}}
\newcommand{\uuen}{(u_{\eps})_{\mathrm{N}}}
\newcommand{\uune}{u_{\mathrm{N}}^{\eps}}
\newcommand{\uute}{(\uue)_{\mathrm{T}}}
\newcommand{\dotue}{\partial_t{\uu}_{\eps}}
\newcommand{\dotute}{(\partial_t{\uue})_{\mathrm{T}}}
\newcommand{\chie}{\chi_{\eps}}
\newcommand{\zetae}{\zeta_\eps}
\newcommand{\zze}{\mathbf{z}_\eps}
\newcommand{\He}{\mathcal{H}}
\newcommand{\matrid}{\mathbf{1}}
\newcommand{\fc}{\mathfrak{c}}
\newcommand{\tetazerose}{\vartheta_s^{0,\eps}}
\newcommand{\tetazeroe}{\vartheta_0^{\eps}}
\newcommand{\tetazeros}{\vartheta_s^{0}}
\newcommand{\tetazero}{\vartheta_0}

\newcommand{\Hc}{H_{\Gamma_c}}
\newcommand{\V}{V}
\newcommand{\Vc}{V_{\Gamma_c}}
\newcommand{\Yc}{Y_{\Gamma_c}}
\newcommand{\bvarphi}{{\mbox{\boldmath$\varphi$}}}
\newcommand{\hunoc}{H^1 (\Gamma_c)}
\newcommand{\hc}{\Hc}
\newcommand{\bfW}{\mathbf{W}}
\newcommand{\bfw}{\mathbf{W}}
\newcommand{\sig}{\sigma}
\newcommand{\HH}{\mathrm{H}}
\newcommand{\BV}{\mathrm{BV}}
\newcommand{\applog}{{\cal L}_\varepsilon}
\newcommand{\tetas}{\teta_s}

\newenvironment{proof}[1][Proof]{\textbf{#1.} }{\ \rule{0.5em}{0.5em}}


\newenvironment{rcomm}{\color{dred}}{\color{black}}
\newcommand{\ber}{\begin{rcomm}}
\newcommand{\edr}{\end{rcomm}}
\newenvironment{controlla}{\color{dred}}{\color{black}}
\newenvironment{new}{\color{ddcyan}}{\color{black}}
\newcommand{\ben}{\begin{new}}
\newcommand{\enn}{\end{new}}
\newenvironment{rnew}{\color{ddcyan}}{\color{black}}
\newenvironment{enew}{\color{ddgreen}}{\color{black}}
\newcommand{\been}{\begin{enew}}
\newcommand{\eenn}{\end{enew}}

                                %
\title{\Large  Analysis of a temperature-dependent model for adhesive contact with friction}

\author{
Elena Bonetti\footnote{\emph{Dipartimento di Matematica ``F.\
Casorati'', Universit\`a di Pavia.
 Via Ferrata, 1 -- 27100 Pavia, Italy}
email: {\tt elena.bonetti\,@\,unipv.it}}, Giovanna
Bonfanti\footnote{\emph{Dipartimento di Matematica, Universit\`a di
 Brescia, via Valotti 9, I--25133 Brescia, Italy,}
 e-mail: {\tt bonfanti\,@\,ing.unibs.it}}, and
Riccarda Rossi\footnote{\emph{Dipartimento di Matematica,
Universit\`a di
 Brescia, via Valotti 9, I--25133 Brescia, Italy,}
 e-mail: {\tt riccarda.rossi\,@\,ing.unibs.it}} }

\date{19.07.2012}
 \maketitle

\begin{abstract}
\noindent
We  propose  a model for  (unilateral) contact with adhesion between a viscoelastic body and
a rigid support,   encompassing thermal \emph{and} frictional effects.
Following \textsc{Fr\'emond}'s  approach, adhesion is described in terms of a surface damage parameter $\chi$. The related  equations are
  the momentum balance for the vector of small  displacements,   and  parabolic-type  evolution equations for $\chi$
and for the absolute temperatures of the body and of the adhesive  substance  on the contact surface.
 All of  the constraints on the internal variables, as well as the contact and the friction conditions, are rendered by means
of subdifferential operators.
  Furthermore, the temperature equations,
 derived from an \emph{entropy} balance law,   feature   singular functions.
  Therefore, the resulting  PDE system  has a highly nonlinear character.

 The main result of the paper states the
 existence of global-in-time solutions to the associated Cauchy problem. It is proved by
 passing to the limit in a carefully tailored approximate problem, via variational techniques.
\end{abstract}
\vskip3mm



\section{Introduction}
\noindent In this paper we propose and analyze a PDE system
modelling  thermal effects in adhesive contact with friction. More
specifically,  we focus on the phenomenon of contact 
between a thermoviscoelastic body and a \emph{rigid} support. We suppose that the body adheres to the
support on a part 
 of its boundary.

 Following  \textsc{M. Fr\'emond}'s approach
\cite{fremond-nedjar,fre},  we describe adhesion in terms of a surface damage
parameter $\chi$ (and its gradient $\nabla \chi$), related to the state of the bonds responsible for
 the adherence of the body to the support.  Further, we consider small displacements $\uu$ and possibly different temperatures
in the body and on the contact surface. Internal constraints, such  as unilateral conditions, are ensured by
the presence of non-smooth monotone operators, also generalizing the Signorini conditions
for unilateral contact to the case when adhesion is active.   Frictional effects are
 encompassed
in the model
through a regularization of the well-known Coulomb law, here generalized to the case of adhesive contact and assuming
thermal dependence of the friction coefficient.

The mathematics of (unilateral) contact problems, possibly with friction,
has received notable attention lately, as
attested, among others,  by the recent monographs on the topic
\cite{eck05}  
and
 \cite{sofonea-han-shillor}.
The analysis of \textsc{Fr\'emond}'s model for  contact with adhesion   has been developed
over the last years in a  series of papers: in \cite{bbr1}--\cite{bbr2} we have focused on
existence, uniqueness, and long-time behavior results on  the isothermal system for (frictionless) adhesive contact.
 Thermal effects both on the contact surface, and in the bulk domain, have been
included in the later papers  \cite{bbr3}--\cite{bbr4}, respectively tackling the well-posedness and
large-time behavior analysis. In \cite{bbr5} we have dealt with a (isothermal) unilateral contact problem taking into account both
friction and adhesion. As far as we know, the present contribution is the first one
addressing a model which combines  adhesion,  frictional  contact,  and thermal effects.

\paragraph{The PDE system.}
Before  stating the PDE system under investigation,  let us   fix the notation for the  normal and tangential component of vectors and tensors
that  will
 be used in what follows.
\begin{notation}
\upshape
Given a vector   $\vv\in \R^3$,  we denote by $\vvn$ and $\vvt$ its normal component and
its tangential part, defined  on the contact surface
by
\begin{equation}
\label{notation-vettore} \vvn:= \vv \cdot \mathbf{n}, \quad \vvt:=
\vv - \vvn \mathbf{n},
\end{equation}
 where $\mathbf{n}$ denotes the outward unit normal vector to the boundary.
Analogously, the normal component and the tangential part of the stress tensor
$\ssigma$ are denoted by $\sigman$ and $\sigmat$, and defined by
\begin{equation}
\label{notation-tensore} \sigman:= \ssigma \mathbf{n} \cdot
\mathbf{n}, \quad \sigmat:=  \ssigma \mathbf{n}  - \sigman
\mathbf{n}.
\end{equation}
\end{notation}
\noindent
 We address the PDE system
\begin{align}
\label{e1} &\partial_t (\ln(\vartheta)) - \dive (\partial_t\mathbf{u}) -\Delta
\vartheta= h \qquad \text{in $\Omega \times (0,T)$,}
\\
\label{condteta} &\partial_{\mathbf{n}} \vartheta= \begin{cases} 0 & \text{in
$(\partial \Omega \setminus \Gamma_c) \times (0,T)$,}
\\
-k(\chi) (\vartheta-\vartheta_s)-\fc'(\teta-\teta_s)  {\partial
I_{(-\infty,0]}( \uun )} | \dotut|  & \text{in $ \Gamma_c \times
(0,T)$,}
\end{cases}
\\
\label{eqtetas} &\partial_t (\ln(\vartheta_s)) - \partial_t
(\lambda(\chi))  -\Delta \vartheta_s= k(\chi)
(\vartheta-\vartheta_s)+ \fc'(\teta-\teta_s)  {\partial
I_{(-\infty,0]}( \uun )} | \dotut|  \quad  \text{ in $\Gamma_c \times
(0,T)$,}
\\
\label{condtetas} &\partial_{\nn_s} \vartheta_s =0 \qquad \text{in
$\partial \Gamma_c \times (0,T)$,}
\\
&-  \hbox{div}\ssigma={\bf f}  \ \text{ with } \ \ssigma=K_e\tensore+ K_v\tensoret+\vartheta \matrid \quad \hbox{in }\Omega\times (0,T),\label{eqI}\\
&{\bf u}={\bf 0}\quad\hbox{in }\Gamma_{1}\times (0,T),
\quad  \ssigma{\bf n}={\bf g} \quad \hbox{in }\Gamma_{2}\times (0,T),\label{condIi}\\
&\sigman\in -\chi \uun - {\partial I_{(-\infty,0]}(\uun)}
\quad\hbox{in }{\Gamma_c} \times (0,T),\label{condIii}\\
& \sigmat\in -\chi \uut -\fc(\teta-\teta_s){\partial
I_{(-\infty,0]}( \uun )} {\bf d} (\dotu)
\quad\hbox{in }{\Gamma_c} \times (0,T),\label{condIiii}\\
&\partial_t{\chi}+\delta  \partial I_{(-\infty,0]}(\partial_t{\chi})
-\Delta\chi+\partial I_{[0,1]}(\chi) + \sigma'(\chi) \ni
-\lambda'(\chi) (\vartheta_s-\vartheta_{\mathrm{eq}})- \frac 1 2
|\uu|^2\quad\hbox{in }{\Gamma_c} \times (0,T),\label{eqII}
\\
&\partial_{\nn_s} \chi=0 \quad \text{in $\partial {\Gamma_c} \times
(0,T)$} \label{bord-chi}
\end{align}
with $\delta \geq 0$, where
\begin{equation}
 \label{formu-d}
\mathbf{d}(\vv) = \begin{cases}
\frac{\vv_{\mathrm{T}}}{|\vv_{\mathrm{T}} |} & \text{if
$\vv_{\mathrm{T}} \neq \mathbf{0}$}
\\
\{ \ww_{\mathrm{T}}\, : \ \ww \in \overline{B}_1 \} & \text{if
$\vv_{\mathrm{T}} = \mathbf{0}$,}
\end{cases}
 \end{equation}
 and $\overline{B}_1$ is
 the closed unit ball in $\R^3$. Note that $\mathbf{d}: \R^3 \rightrightarrows \R^3$ is the subdifferential
 of the function   $j: \R^3 \to [0,+\infty)$ defined by
 $j(\vv)= |\vv_{\mathrm{T}} |$,   cf.\ \eqref{def-j} later on.
In what follows, $\Omega$ is a (sufficiently smooth) bounded domain in $\R^3$,
in which the body is located,
with $\partial\Omega= \overline{\Gamma}_1 \cup \overline{\Gamma}_2 \cup \overline{\Gamma}_c$ and
 ${\Gamma_c}$
 the contact surface between  the body and  the  rigid support.
 From now on, we will suppose that $\gc$
  is a smooth
bounded domain of $\R^2$
 (one may think of a flat surface),  and we denote by $\mathbf{n}_s$ the outward unit
 normal vector to $\partial \gc$.  Let us point out that  $\uu$ is the vector of small
 displacements,
  $\ssigma$ the stress tensor,
  whereas $\chi$ is the so-called \emph{adhesion
 parameter}, which denotes the fraction of active microscopic bonds
 on the contact surface.
As for the thermal variables,
  $\teta$ is  the absolute temperature of the body $\Omega$ whereas $\teta_s$
is the absolute temperature of the adhesive  substance  on $\gc$.
Here and in what follows, we shall write $v$, in place of $v|_{\Gamma_c}$, for  the trace
on $\Gamma_c$ of a function $v$ defined in $\Omega$.

While postponing to Section \ref{s:2} the rigorous derivation of the PDE
system (\ref{e1}--\ref{bord-chi}) based on the laws of Thermomechanics,
let us briefly comment on the nonlinearities therein involved.
First, we observe that the singular terms $\ln(\teta)$ and $\ln(\tetas)$ in \eqref{e1} and \eqref{eqtetas} force the temperatures $\teta$ and $\tetas$
to remain strictly positive, which is necessary to get thermodynamical consistency.  The
multivalued
 operator $\partial I_C:
\R \rightrightarrows \R$
(with $C$ the interval $[0,1]$ or the half-line $(-\infty,0])$) is the subdifferential (in the sense of convex analysis) of the indicator function of the convex set $C$.
 We recall that   $\partial I_C(y) \neq \emptyset $
 if and only if $y \in C$, with
 $\partial I_C(y)=\{0\}$ if $y$ is in  the   interior of $C$, while $\partial I_C(y) $ is   the normal cone to the
boundary of $C$ if $y\in\partial C$. By means of indicator functions, we enforce the  unilateral condition $\uun\leq 0$, as well as  the constraint   $\chi\in[0,1]$ and, if $\delta >0$, the
irreversibility of the damage process, viz.\  $\partial_t\chi\leq0$. Then,  the subdifferential terms
occurring in \eqref{condteta},   \eqref{eqtetas},  \eqref{condIii},   \eqref{condIiii},   and \eqref{eqII}
  represent
internal forces which  activate
 to prevent $\uun$, $\chi$, and $\partial_t\chi$  from taking values outside the physically admissible range.
 Finally,  $K_e$ and $K_v$ are positive-definite tensors,
$k$, $\lambda$,  $\sigma$, and   $\fc$ are sufficiently smooth functions, whereas
 $\teta_\mathrm{eq}$ is a critical temperature  and $h$, $\mathbf{f}$, and
 $\mathbf{g}$ are given data:  we refer to  Section 2 for all details on their physical
 meaning, and to Section  \ref{ss:2.2} for the precise assumptions on them.

As for  conditions \eqref{condIii}--\eqref{condIiii}, let us observe that
(\ref{condIii}) can be rephrased as
\begin{align}
&\uun\leq 0,\quad \sigman+  \chi \uun \leq 0,\quad \uun (\sigman+
\chi \uun)=0 \quad\hbox{in }{\Gamma_c} \times (0,T),\label{Signorini}
\end{align}
which, in the case $\chi=0$, reduce to the classical Signorini
conditions for unilateral contact. Conversely, when $0 < \chi \leq 1$, $\sigman$ can be
positive, namely the  action  of the adhesive substance on $\gc$  prevents
separation when a tension is applied.
 Moreover, in view of   \eqref{condIii}, \eqref{condIiii} can
be expressed by
\begin{equation}
\label{Coulomb}
\begin{aligned}
&|\sigmat +\chi \uut|\leq  \fc(\teta-\teta_s)|\sigman + \chi \uun |,\\
&|\sigmat +\chi \uut|< \fc(\teta-\teta_s)|\sigman + \chi \uun |\Longrightarrow \dotut = \mathbf{0},\\
&|\sigmat +\chi \uut|= \fc(\teta-\teta_s) |\sigman + \chi \uun |\Longrightarrow
\exists\,  \nu \geq 0: \ \dotut=-\nu (\sigmat +\chi \uut),
\end{aligned}
\end{equation}
which generalize the {\it dry friction} Coulomb law, to the case
when adhesion effects are taken into account. Note  that the \emph{positive} function $\fc$ in \eqref{Coulomb}   is the
 \emph{friction coefficient}.
\paragraph{Related literature.}
We refer to~\cite{eck05} and \cite{sofonea-han-shillor},
for a general  survey of the
literature on models of (unilateral) contact, as well as to the references in  \cite{bbr1,bbr2}
for \emph{isothermal} models  of contact with adhesion, possibly with friction (see also \cite{bbr5}).
Here we will just focus on (a partial review of) temperature-dependent models for frictional contact: as previously said,
to the best of our knowledge the PDE system (\ref{e1}--\ref{bord-chi})
 is the first one modelling
 \emph{adhesive contact} with  frictional
\emph{and} thermal effects.

The major difficulty in the analysis of unilateral  contact problems with friction is the presence of constraints on both the
(normal)
displacement
$\uun$ and the tangential velocity $\dotut$.  It can be overcome only by resorting  to suitable simplifications in the related
PDE systems.
Over the years, several options have bee explored in this direction, without affecting the
 physical consistency of the underlying models.

 Since \textsc{Duvaut}'s pioneering work \cite{duvaut},
a commonly accepted
approximation of the dry friction Coulomb law \eqref{Coulomb}, which we will also adopt,
 involves the usage of a nonlocal regularizing
operator $\Reg $, cf.\ \eqref{replacement} below.
An alternative regularization of the classical Coulomb law is the so-called
\emph{SJK-Coulomb} law of friction, which is for example considered  in \cite{rochdi-shillor2000} and in \cite{akrs2002},
dealing with \emph{quasistatic} models for thermoviscoelastic contact with friction.

An alternative possibility is to replace the Signorini contact conditions \eqref{Signorini} with a \emph{normal compliance}
condition, which allows for the interpenetration of the surface asperities and thus for dispensing with the unilateral constraint on
$\uun$.
 Analytically, the normal compliance law corresponds to a penalization of the subdifferential operator $\partial I_{(-\infty,0]}$ in
\eqref{condIii}--\eqref{condIiii}. In this connection, we refer e.g.\ to \cite{and-kut-shil97}, analyzing a
\emph{dynamic} model for frictional contact of a thermoviscoelastic body with a rigid obstacle, with the power law normal compliance condition for unilateral contact, and the corresponding
 generalization of
Coulomb's law of dry friction. Unilateral contact is  modelled
by a  normal compliance condition in
\cite{Andrews} as well, where a dynamic contact problem for a thermoviscoelastic body, with frictional and wear effects on the contact surface, is investigated.  A wide class of dynamic frictional contact problems in thermoelasticity and thermoviscoelasticity is also tackled in \cite{Figue-Trabucho}, with contact rendered by means
of  a normal compliance law.

In an extensive series of papers (cf.\ the references in the monograph \cite{eck05}),
  \textsc{C. Eck \& J.
Jaru\v{s}ek}  developed  a different approach, which   enabled   them
to prove existence results for  dynamic contact
problems, coupling \emph{dry} friction and Signorini contact,
\emph{without} recurring to any regularizing operator. However, they
 used   a different form of Signorini conditions, expressed not in terms
of $\uu$ but of $\dotu$. They observed   that this different law can be
interpreted as a
first-order approximation with respect
to the time variable, realistic for a short time interval and for a vanishing
initial gap between the body and the obstacle, and hence  it is physically interesting.
Within this modelling approach, existence results for
contact problems with friction  and thermal effects  were   for instance obtained in
\cite{Eck-Jarusek} and in \cite{Eck2002}.

Let us stress that, in all of these contributions the existence of solutions is proved for a weak
formulation of the related PDE systems, which involves   a variational inequality, and not
an evolutionary differential inclusion,
 for the displacement.

\paragraph{Analytical difficulties.}
From now on, we will take $\delta=0$ and thus confine ourselves to the investigation of   the \emph{reversible} case.
We postpone the analysis of the \emph{irreversible} case
$\delta>0$ to the forthcoming paper \cite{bbr7}, where we will address also
slightly different equations for $\teta$ and $\teta_s$
obtained by a different choice of the bulk and surface entropy fluxes
(here given by $-\nabla\teta$ and $-\nabla\tetas$, respectively). For more details, we refer to the
  upcoming  Remark \ref{remflusso}.

Still,
the analysis of system (\ref{e1}--\ref{bord-chi}) is fraught with
obstacles:
\begin{compactenum}
\item[1)]
 First of all, the highly nonlinear character of the
 equations, due to the presence of several singular and  multivalued operators.
 In particular, 
since Coulomb  friction is included in the model, a multivalued operator
occurs even in the coupling terms between the equations for $\teta$,
$\teta_s$, and~$\uu$, and \eqref{condIiii} features the product of two subdifferentials.
\item[2)]
A second analytical difficulty is given by the coupling of \emph{bulk}
and \emph{(contact) surface} equations, which requires sufficient
 regularity of the bulk variables $\teta$ and $\uu$ for their
traces on $\gc$ to make sense.
\item[3)] In turn, the mixed boundary conditions on $\uu$ do not
allow for elliptic regularity estimates which could significantly
enhance the spatial regularity of $\uu$ and $\partial_t \uu$. In particular,
we are not in the position to obtain the   $H^2 (\Omega;\R^3)$-regularity for
$\uu$ and $\partial_t{\uu}$. The same problem arises
 for $\teta$, due to the
third type boundary condition \eqref{condteta}.
\item[4)] Additionally, the temperature  equations \eqref{e1} and \eqref{eqtetas}
need to be carefully handled because of the singular character of
the terms $\partial_t \ln(\teta)$ and $\partial_t \ln(\teta_s)$
therein occurring. In particular, they do not allow but for poor
time-regularity of $\teta$ and $\teta_s$.
\end{compactenum}
Let us stress that  the presence of several multivalued operators in
system (\ref{e1}--\ref{bord-chi}) is due to the fact that,
 all the constraints on the internal variables, as well as the unilateral contact
conditions and the friction law, are rendered by means of subdifferential operators. Therefore,
in our opinion
 the resulting formulation of the problem provides more complete information than those formulations
   based on
variational inequalities. Indeed, it
   enables us to clearly identify the internal forces and reactions  which derive  by
 the enforcement of the
 physical constraints.

In particular,  let us dwell on  the
coupling of unilateral contact and the
dry friction Coulomb law: as already mentioned, it  introduces severe mathematical
difficulties,  unresolved even in the
case without adhesion. These problems are  mainly related to the fact that we cannot control
$\sigman$ and $\sigmat$ in \eqref{condIii} and \eqref{condIiii} pointwise, due to the presence of
 two different non-smooth operators.  For this reason,
  as  done in  \cite{bbr5} and following \textsc{Duvaut}'s work \cite{duvaut},
 we are led to regularize \eqref{condIiii}   by resorting to  a {\it nonlocal version} of the Coulomb law. The
idea is to replace the nonlinearity in \eqref{condIiii} involving friction  by
the  term
\begin{equation}
\label{replacement} \text{
 $ \fc(\teta-\teta_s)|\Reg( \partial I_{(-\infty,0]}(\uun)| \mathbf{d}(\partial_t\uu)$.}
\end{equation}
Accordingly, the coupling term
$\fc'(\teta-\teta_s)  \partial
I_{(-\infty,0]}( \uun ) | \dotut| $ in \eqref{condteta} and \eqref{eqtetas} is replaced by
\begin{equation}
\label{replacement2}
\fc'(\teta-\teta_s)   | \Reg(\partial
I_{(-\infty,0]}( \uun )) |  | \dotut|.
\end{equation}{
  In \eqref{replacement} and \eqref{replacement2}, $\Reg$ is a
 regularization operator, taking into account
 nonlocal interactions on the contact surface.
 We refer to \cite{bbr5} for further comments and references
 to several items in the literature on frictional contact problems, where this regularization is adopted.
  In the forthcoming Remark \ref{regolarizzo}, we hint at
 the modeling  derivation of the  PDE system (\ref{e1}--\ref{bord-chi}) with
  the terms \eqref{replacement} and \eqref{replacement2}, while
  in Example \ref{ex-reg} below we give  the construction of an operator $\Reg$ complying with the
  conditions we will need to impose for our existence result.

 The difficulties attached to the
 coupling between thermal and frictional effects are apparent in the
 dependence of the friction coefficient $\fc$ on the thermal gap $\teta-\teta_s$.
 In order to deal with this, and to tackle the (passage to the limit in the approximation of the) terms
 \eqref{replacement} and \eqref{replacement2}, it is crucial to prove strong
 compactness for (the sequences approximating) $\teta$ and $\teta_s$ in suitable spaces.
 A key step in this direction is the derivation of an estimate for $\teta$ in
 $\BV (0,T;W^{1,3+\epsilon}(\Omega)')$ and for $\teta_s $ in  $ \BV (0,T;W^{1,2+\epsilon}(\gc)')$ for
 some $\epsilon>0$. Combining this information with a suitable version of the Lions-Aubin compactness
 theorem generalized to $\BV$ spaces (see, e.g., \cite{roub-book}), we will deduce the desired compactness
 for (the sequences approximating) $\teta$ and $\teta_s$.

\paragraph{Our results.}
The main result of this paper  (cf.\ Theorem \ref{mainth:1} in Section \ref{ss:2.3})
 states  the existence of (global-in-time) solutions
to  the Cauchy problem for system (\ref{e1}--\ref{bord-chi}).

In order to prove it, we introduce a carefully tailored approximate problem for
system (\ref{e1}--\ref{bord-chi}),
in which we replace some of the nonlinearities therein involved by suitable  Moreau-Yosida-type regularizations.
The existence of solutions for the approximate problem is obtained via
a Schauder fixed point argument,  yielding a local-in-time
existence result. The latter is   combined with a series of \emph{global} a priori estimates,
holding with constants independent of the regularization parameter, 
 which
 enable us to extend the local solution to a global one.
These very same estimates are then exploited in the passage to the limit
in the approximate problem,
 together with techniques from maximal monotone operator theory and
refined compactness results. In this way, we carry out  the proof of Thm.\ \ref{mainth:1}.
\paragraph{Plan of the paper.}
In Section \ref{s:2} we  outline  the rigorous  derivation of system (\ref{e1}--\ref{bord-chi})
and we comment on its thermodynamical consistency. Hence, in Section
\ref{s:setup}, after enlisting all of the assumptions on the nonlinearities featured in (\ref{e1}--\ref{bord-chi}) and on the problem data, we set up the variational formulation of the problem and state
our main existence result (cf.\ Theorem\ \ref{mainth:1} below). Section \ref{s:3} is devoted to the
proof of Thm.\ \ref{th:exist-approx-rev}, ensuring the
 existence of
 (global-in-time)
 solutions to  the
approximate problem. 
 Finally, in Section \ref{s:3.3} we pass to the limit with the approximation parameter
and conclude the proof of Thm.\ \ref{mainth:1}.

\section{The model}
\label{s:2}
\noindent
In this section we
 sketch the modeling approach underlying the derivation of our  phase-transition type system
 for a contact problem, in which adhesion and friction are taken into account in a non-isothermal framework.
 The equations, written in the bulk  domain and on the contact surface, are recovered by
  the
 general laws of
 Thermomechanics with
 energies and dissipation potentials defined in $\Omega$ and on $\Gamma_c$.
 They can be considered as balance equations,  based on
 a generalization
of the principle of virtual powers.   Such a generalization accounts   for  micro-movements and micro-forces, which are responsible for the breaking
 of the adhesive bonds on the contact surface.
  We will not   derive the PDE system (\ref{e1}--\ref{bord-chi}) in  full detail, referring the reader to the discussions  in
   \cite{bbr3} on  the modeling of thermal effects in contact with adhesion,     and in
    \cite{bbr5} for
a (isothermal) model of adhesive contact with friction. Here, we will rather briefly
outline the main ingredients of the derivation,  and  just focus  on
the modeling novelty of this paper.

The phenomenon of adhesive contact is described  by state
and dissipative variables, defining the equilibrium and the evolution of the system, respectively.
The state variables are  the symmetric strain
$\tensore$, the trace of  the  vector of small displacements  $\uu$ on the contact surface,
 the surface  adhesion  parameter $\chi$, its gradient $\nabla\chi$, and
 the absolute  temperatures $\teta$ and $\teta_s$ of the bulk domain and of
the  contact surface, respectively.
The evolution is derived in terms of a pseudo-potential of dissipation, depending on the dissipative
variables  $\partial_t\uu$, $\partial_t\chi$, $\nabla\teta$, $\nabla\teta_s$, and  the thermal gap on the contact surface
$\teta-\teta_s$.
\paragraph{The free energy and the dissipation potential.} The free energy of our system is written as follows
$$
\Psi=\Psi_\Omega+\Psi_\Gamma,
$$
 $\Psi_\Omega$ being defined in $\Omega$ and $\Psi_\Gamma$ on $\Gamma_c$.
The bulk  contribution  $\Psi_\Omega$ is given by
\begin{equation}\label{freeO}
\Psi_\Omega=\teta(1-\ln(\teta))+\teta\mathrm{tr}(\tensore)+\frac 1 2 \tensore K_e\tensore,
\end{equation}
with  $K_e$ the elasticity tensor  and  $\mathrm{tr}(\tensore)$ the trace of $\tensore$.
 Note that, for the sake of simplicity, we have  taken both the specific heat and the thermal
expansion coefficient
 equal to $1$. The free energy
  on the contact surface is defined by
\begin{align}\label{freeG}
\Psi_{\Gamma}= & \teta_s(1-\ln(\teta_s))+\lambda(\chi)(\teta_s-\teta_{\mathrm{eq}})+I_{[0,1]}(\chi)+\sigma(\chi)\\\no
&\quad +\frac {\cn}{2} \chi (\uun)^2+ \frac {\ct}{2}
\chi \vert{\uut}\vert^2+
I_{(-\infty,0]}(\uun)+\frac{ \kappa_s}{2}\vert\nabla\chi\vert^2,
\end{align}
 where
 $\cn, \, \ct, \, \kappa_s$ are  positive constants.
 Note that $\cn$ and
$\ct$ (which are the adhesive coefficients for the normal and tangential
components, respectively) a priori  may
be different, due to possible anisotropy in the response of the material
to stresses.   However,
 for the sake of simplicity in what follows  we let $\cn=\ct=\kappa_s=1$.
In \eqref{freeG},
 $\sigma$ is a sufficiently smooth function accounting for some internal properties of the
  adhesive substance on $\gc$,
such as cohesion:  the simplest  form for the cohesive   contribution to  the energy is $\sigma(\chi)=w_s(1-\chi)$ for some $w_s>0$.
Let us now briefly comment on the internal constraints.
 The energy is defined for any value of the state variables, but it
 is set equal to $+\infty$ if the variables assume values which are not physically consistent.
 Indeed, the indicator function
$I_{(-\infty,0]}$ enforces the internal constraint  $\uun\leq 0$,
 i.e.\ it  renders the impenetrability condition between the body
and the support. Finally, the term $I_{[0,1]}(\chi)$ forces $\chi$ to
 take values in the interval $[0,1]$.
 The function $\lambda$ provides the latent heat $\lambda'$,
while $\teta_{\mathrm{eq}}$
 is a  critical temperature,  which governs the evolution of the cohesion of the adhesive substance with respect to the temperature.

 As far as dissipation is concerned,  we consider in particular dissipative effects on the boundary due to friction. The pseudo-potential
is
$$
\Phi=\Phi_\Omega+\Phi_\Gamma
$$
 $\Phi_\Omega$ being defined in $\Omega$ and $\Phi_\Gamma$ on $\Gamma_c$.
For the bulk  contribution $ \Phi_\Omega$ we have
\begin{align}
\label{pseudovol} \Phi_\Omega:= \frac 1 2 \tensoret K_v\tensoret+\frac {\alpha(\teta)} {2}|\nabla\teta|^2,
\end{align}
while the contact surface contribution $\Phi_{{\Gamma}}$ reads
\begin{align}
\label{pseudobordo}
\begin{aligned}
\Phi_{{\Gamma}}: & = \fc(\teta-\teta_s) \left|-R_N +  \uun \chi
\right| j(\dotu)+ \frac {1}
2\vert\partial_t{\chi}\vert^2\\ & \quad +\delta I_{(-\infty,0]}(\partial_t{\chi})+\frac {\alpha_s(\teta_s)} 2|\nabla\teta_s|^2+\frac
1 2k(\chi)(\teta-\teta_s)^2,
\end{aligned}
\end{align}
where the positive function
$\fc$ has the meaning of a  friction coefficient,   $R_N$ will be specified later (see \eqref{reazione}),
the function $j$ is
\begin{equation}
\label{def-j}
 j(\vv) = | \vv_{\mathrm{T}} | \quad \text{for all
$\vv \in \R^3.$}
\end{equation}
$\delta\geq0$, and
$\alpha$ and $\alpha_s$ are the  thermal diffusion coefficients in the bulk domain and on the contact surface, respectively.
The  positive (and sufficiently smooth) function $k$  is also a  contact surface thermal diffusion coefficient,
 accounting for the heat
exchange between the body and the  adhesive substance on $\gc$.
The assumptions on all of  these functions have to guarantee that the pseudo-potential of dissipation
is a non-negative function, convex w.r.t.\  the dissipative variables (note that $\teta$ and $\teta_s$ are not dissipative variables),
and  that dissipation is zero once the dissipative variables are equal to zero  (see also the
 upcoming discussion on the \emph{thermodynamical consistency} of the model).
Observe that, if $\delta>0$,   the model  encompasses  an {\it irreversible} evolution
for the damage-type variable $\chi$, as  it enforces    $\partial_t\chi\leq0$.
 Furthermore, the $1$-homogeneity of
 the function $j$ in \eqref{def-j}
reflects  the rate-independent character of frictional
dissipation.
\begin{remark}
\upshape
Note that the functions $\sigma $  and $\lambda$ in the free energy \eqref{freeG}  are related to the cohesion of the adhesive substance, as it results from \eqref{eqII}. Indeed, the term $-\sigma'(\chi)-\lambda'(\chi)(\teta_s-\teta_{\mathrm{eq}})$ is a (generalized)  cohesion of the material, depending on the temperature.  It represents a threshold for $\frac 1 2|{\bf u}|^2$ to produce damage.
\end{remark}

The main novelty of this paper in comparison with \cite{bbr5} is the fact that the friction coefficient $\fc$
depends on the  thermal gap  $\teta-\teta_s$ (cf. \cite[p.\ 12]{sofonea-han-shillor}).
This  relies on  the mo\-de\-ling ansatz that \emph{friction generates heat}
(cf.\ \cite{Eck-Jarusek}), so that  ultimately  we have the contribution
 $\fc'(\teta-\teta_s)  {\partial I_{(-\infty,0]}( \uun )} |
\dotut|$  as a source of heat on the contact surface $\gc$, both in  the boundary condition
\eqref{condteta} for $\teta$ and in equation \eqref{eqtetas} for $\teta_s$.
\paragraph{The balance equations.}
The equations for the  evolution of  the temperature variables are recovered from
entropy balance equations in $\Omega$ and $\Gamma_c$ (cf. \cite{BFR} for a detailed motivation
 of this approach and \cite{bbr3}
for the application to contact problems), i.e.
\begin{equation}\label{entrpI}
\partial_t s+\hbox{div }{\bf Q}=h\hbox{ in }\Omega \times (0,T),
\quad{\bf Q}\cdot{\bf n}=F\hbox{ on }  \partial\Omega \times (0,T)
\end{equation}
with
$h$ an external volume source, $F$ the  entropy flux through the boundary, $s$ the entropy, ${\bf Q}$ the entropy flux
 in the bulk domain,
and
\begin{equation}\label{entropII}
\partial_t s_s+\hbox{div }{\bf Q}_s=F\hbox{ in }\Gamma_c  \times (0,T), \quad{\bf Q}_s\cdot{\bf n}_s=0\hbox{ on }\partial\Gamma_c  \times (0,T),
\end{equation}
where we have denoted by $s_s$ and ${\bf Q}_s$ the entropy and the entropy flux on the contact surface, respectively. Note that,
on $\Gamma_c$ the term $F$ (involved   in the  boundary condition for \eqref{entrpI})  is the
 entropy provided by the adhesive substance to the body.
 We couple \eqref{entrpI} and \eqref{entropII}  with the
generalized momentum equations for macroscopic motions in $\Omega$ and micro-movements in
$\Gamma_c$, viz.
\begin{align}\label{momI}
&-\hbox{div }{\ssigma}={\bf f}\hbox{ in }\Omega \times (0,T),\\
&{\ssigma}{\bf n}=-{\bf R}\hbox{ on }\Gamma_c\times (0,T),\quad{\bf u}={\bf 0}\hbox{ on }\Gamma_1\times (0,T),
\quad{\mathbf\ssigma}{\bf n}={\bf g}\hbox{ on }\Gamma_2\times (0,T)
\end{align}
where  recall that  $\ssigma$ is the macroscopic stress tensor,
${\bf f}$   is  a volume applied force, ${\bf g}$ is  a known traction,   $-{\bf R}$ is the action of the obstacle on the solid,  and
\begin{equation}\label{momII}
B-\hbox{div }{\bf H}=0\hbox{ in }\Gamma_c\times (0,T),\quad{\bf H}\cdot{\bf n}_s=0\hbox{ on }\partial\Gamma_c\times (0,T),
\end{equation}
with ${\bf H}$ and $B$  microscopic internal stresses,  responsible for the damage of the
adhesive bonds between the body and the support.
\paragraph{The constitutive equations.}
To recover the PDE system (\ref{e1}--\ref{bord-chi}), we
have to combine  \eqref{entrpI}--\eqref{entropII}   with suitable constitutive relations, for the involved physical quantities, in terms of
$\Psi$ and $\Phi$.
We have
\begin{equation}\label{entropie}
s=-\frac{\partial\Psi_\Omega}{\partial\teta},\quad s_s=-\frac{\partial\Psi_\Gamma}{\partial\teta_s}
\end{equation}
and
\begin{equation}\label{flussi}
{\bf Q}=-\frac{\partial\Phi_\Omega}{\partial\nabla\teta},\qquad{\bf Q}_s=-\frac{\partial\Phi_\Gamma}{\partial\nabla\teta_s}.
\end{equation}
As for as the flux through the boundary $F$, we  impose
\begin{equation}\label{flussbordo}
 F=\frac{\partial\Phi_\Gamma}{\partial(\teta-\teta_s)}\hbox{ on }\Gamma_c,\quad F=0\hbox{ on the remaning part}.
\end{equation}
Then, we prescribe  for  (the dissipative and non-dissipative contributions to) the stress tensor
\begin{equation}\label{stress}
 {\ssigma}= {\ssigma}^{nd}+ {\ssigma}^d=\frac{\partial\Psi_\Omega}{\partial\tensore}+\frac{\partial\Phi_\Omega}{\partial\tensoret},
\end{equation}
and the reaction ${\bf R}=R_N{\bf n}+{\bf R}_T$   reads
\begin{align}\label{reazione}
 &R_N=R_N^{nd}=\frac{\partial\Psi_\Gamma}{\partial{\uun}}\\
 \label{reazione-tangential}
 &{\bf R}_T={\bf R}_T^{nd}+{\bf R}_T^d=\frac{\partial\Psi_\Gamma}{\partial{\uut}}+\frac{\partial\Phi_\Gamma}{ \partial(\partial_t\uut)}.
\end{align}
 Finally, $B$ and $\mathbf{H}$ are given by
\begin{align}
&
\label{bimaiuscolo}
B= B^{nd}+B^d=\frac{\partial\Psi_\Gamma}{\partial\chi}+\frac{\partial\Phi_\Gamma}{ \partial(\partial_t \chi)}\,,
\\
&
\label{accamaiuscolo}
\mathbf{H}= {\bf H}^{nd}+{\bf H}^d=\frac{\partial\Psi_\Gamma}{\partial \nabla\chi}+ \frac{\partial\Phi_\Gamma}{ \partial(\nabla (\partial_t \chi))}\,.
\end{align}
\begin{remark}\label{remflusso}
\upshape
 Let us briefly comment on the entropy flux  laws for ${\bf Q}$ and ${\bf Q}_s$:
   as shown by \eqref{pseudovol} and \eqref{pseudobordo}, they
  depend on the choice of  the  thermal diffusion coefficients  $\alpha$ and $\alpha_s$.
 In general    the latter are  functions  of  the temperature variables $\teta$ and
  $\teta_s$
  (cf.\ e.g.\ \cite{eck05}).
 If $\alpha(\teta)=\alpha>0$ we get ${\bf Q}=-\alpha\nabla\teta$. In this case,
 for  the heat flux ${\bf q}=\teta{\bf Q}$
 we  obtain   ${\bf q}=-\alpha\teta\nabla\teta=-\frac\alpha 2\nabla\teta^2$.
 Analogously, if $\alpha(\teta)=\frac 1\teta$
 we get ${\bf Q}=-\nabla\ln(\teta)$ and for ${\bf q}$ the standard Fourier law ${\bf q}=-\nabla\teta$. Hence, in the case, e.g.,
 $\alpha(\teta)=\teta$ (which is admissible for thermodynamics) we have ${\bf Q}=-\teta\nabla\teta=-\frac 1 2\nabla\teta^2$. Analogous considerations  hold for ${\bf Q}_s$.
\end{remark}

 In what follows,
we   fix as heat flux laws
 \begin{equation}\label{heat-fluxes}
 \alpha(\teta) \equiv \alpha=1,\quad\alpha_s(\teta_s) \equiv \alpha_s=1.
 \end{equation}
 Furthermore,
 as already mentioned,   we will confine our  analysis  to the  case of a  \emph{reversible} evolution
 for the damage-type variable $\chi$,  taking
 \[
 \text{$\delta=0$ in $\Phi_\Gamma$.}
 \]
 With these choices,
 combining  \eqref{entropie}--\eqref{accamaiuscolo}  with \eqref{entrpI}--\eqref{momII}
 we derive the PDE system (\ref{e1}--\ref{bord-chi}).
\begin{remark}\label{regolarizzo}
\upshape
 Let us point out that, the PDE system tackled in this paper, featuring the
 non-local regularization for the Coulomb law $\Reg$, can be derived
 following  the very same  procedure described  above. Indeed,  it is sufficient to replace $\Phi_\Gamma$ by
 \begin{align}
\label{pseudobordoreg}
\begin{aligned}
 \Phi_{{\Gamma}}: & = \fc(\teta-\teta_s)     \left| \Reg(-R_N +  \uun \chi) \right| j(\dotu)+ \frac {1}
2\vert\partial_t{\chi}\vert^2
\\ & \quad
+\delta I_{(-\infty,0]}(\partial_t{\chi})+\frac {\alpha_s(\teta_s)} 2|\nabla\teta_s|^2+\frac
1 2k(\chi)(\teta-\teta_s)^2.
\end{aligned}
\end{align}
\end{remark}

\paragraph{\bf Thermodynamical consistency.}
Let us now  briefly comment on the derivation of the model in relation to its thermodynamical consistency.
In particular we discuss the validity of the Clausius-Duhem inequality  for the whole thermomechanical system we are dealing with. The latter consists of the body $\Omega$ and the contact adhesive surface $\Gamma_c$, encompassing 
 the dissipative heat exchange between   the body  and  the adhesive substance  on $\Gamma_c$, due to the difference of the  temperatures.
 Recall that we have derived the
 evolution equations for the temperature variables from an
 entropy balance in the domain $\Omega$ and on $\Gamma_c$. To prove thermodynamical consistency (which corresponds to  showing  that dissipation is positive), we have to discuss the entropy balance for the whole system, i.e.\ also including the  interactions between the bulk domain and the adhesive substance.

As far as the domain $\Omega$
 is concerned,  this property  is  fairly standard
(the reader may refer to \cite{fre}): it is ensured once $\Phi_\Omega$
has the characteristics of a  pseudo-potential of dissipation. In particular, this
requires that, in \eqref{pseudovol} the tensor  $K_v$ is positive definite, and  the
thermal diffusion coefficient  $\alpha$ is non-negative. This  leads 
 to  positivity of the dissipation, viz.\
$$
\varepsilon(\partial_t{\bf u})K_v\varepsilon(\partial_t{\bf u })+\alpha(\teta)|\nabla\teta|^2\geq 0.
$$
Then, the Clausius-Duhem inequality is obtained by writing the energy balance and exploiting the already specified constitutive relations.

Second, the internal energy balance on $\Gamma_c$ is recovered by the first principle,
in which the effective power of the internal forces responsible for the damage process is included.
Again using the constitutive relations, it yields
\begin{equation}\label{Iprincol}
\theta_s\partial_t s_s+\theta_s\hbox{div }{\bf Q}_s=\theta_s Q_s+B^d\partial_t\chi
+{\bf H}^d\cdot\nabla\partial_t\chi-{\bf Q}_s\cdot\nabla\theta_s,
\end{equation}
where we denote by ${Q}_s$ the entropy received by the adhesive substance from the solid and $B^d,{\bf H}^d,{\bf Q}_s$ are defined by \eqref{bimaiuscolo},
 \eqref{accamaiuscolo}, and \eqref{flussi}.
Note that $Q_s$ actually represents an entropy source. In our model we have prescribed that  (see \eqref{entrpI}--\eqref{entropII})
\begin{equation}\label{ipotesiimp}
-Q=Q_s=F,
\end{equation}
$Q$ denoting  the entropy received by the solid from the adhesive substance. Note that by \eqref{entrpI} we have
\begin{equation}\label{Iprincont}
{\bf Q}\cdot{\bf n}+Q=0.
\end{equation}

To deduce the  conditions  on $F$ which ensure the thermodynamical consistency, we write
 an entropy equality
  on the contact surface, accounting for the interaction between the body and the adhesive substance, i.e.
\begin{equation}\label{Iprinins}
\bar\teta\partial_t s_{int}=-Q\teta-Q_s\teta_s,
\end{equation}
where $s_{int}$ is the entropy exchange,
defined in terms of the variable $\bar\teta=\frac 12 (\teta+\teta_s)$
(for any further details see \cite{fre}).
Now,  the Clausius-Duhem inequality on the contact surface (combining the adhesive substance and its interaction with the body)
reads
\begin{equation}\label{CDcomb}
 \partial_t s_s+\hbox{div }{\bf Q}_s+\partial_t s_{int}\geq {\bf Q}\cdot{\bf n}.
\end{equation}
Thus, exploiting
\eqref{Iprincol}, \eqref{ipotesiimp},   \eqref{Iprincont},  and \eqref{Iprinins},  inequality \eqref{CDcomb} follows once it is ensured that
(recall that the absolute temperatures are strictly positive)
$$
F(\theta-\theta_s)+B^d\partial_t\chi
+{\bf H}^d\cdot\nabla\partial_t\chi-{\bf Q}_s\cdot\nabla\theta_s\geq 0.
$$
Observe that this corresponds to prescribing in \eqref{pseudobordo} that
$$
\alpha_s\geq0,\quad \fc\geq 0, \quad k\geq 0,\quad \fc'(\teta-\teta_s)(\teta-\teta_s)\geq 0.
$$
 Indeed,  all of the above conditions are ensured by \eqref{heat-fluxes} and by the forthcoming Hypotheses (I) and (V) on $\fc$ and $k$. 
\begin{remark}\upshape
 Note that, in the case no heat is exchanged between the body and the adhesive substance,  the Clausius-Duhem inequality on the contact surface
 is granted upon  requiring that
 $$
 B^d\partial_t\chi
+{\bf H}^d\cdot\nabla\partial_t\chi-{\bf Q}_s\cdot\nabla\theta_s\geq 0.
$$
This follows from \eqref{Iprincol} taking $Q_s=0$.
\end{remark}

\section{Main results}
\label{s:setup}
\subsection{Setup}
\label{ss:2.1} \noindent Throughout the paper we
shall
 assume that
 \begin{equation}
 \label{assumpt-domain}
 \begin{gathered}
 \text{
 $\Omega$ is a
bounded   Lipschitz  domain in $\R^3$, with }
\\
\partial\Omega= \overline{\Gamma}_1 \cup \overline{\Gamma}_2 \cup \overline{\Gamma}_c, \ \text{ $\Gamma_i$, $i=1,2,c$,
open disjoint subsets in the relative topology of $\partial\Omega$, such that  }
\\
\mathcal{H}^{2}(\Gamma_1), \,  \mathcal{H}^{2}(\Gamma_c)>0, \qquad
\text{and ${\Gamma_c} \subset \R^2$ a  sufficiently smooth  \emph{flat} domain.}
\end{gathered}
\end{equation}
 More precisely,
by \emph{flat} we mean that $\Gamma_c$ is a subset of a hyperplane of $\R^2$ and $\mathcal{H}^2(\Gamma_c)=\mathcal{L}^2(\Gamma_c)$,
$\mathcal{L}^d$ and $\mathcal{H}^d$ denoting the $d$-dimensional Lebesgue and  Hausdorff measures.
As for smoothness,
we require that $\Gamma_c$ has a $\mathrm{C}^2$-boundary: thanks to, e.g., \cite[Thm.\ 8.12]{gilbarg-trudinger},
 this will justify
the elliptic regularity estimates we will perform on the solution component $\chi$.
\begin{notation}
\label{notation-1} \upshape
 Given a Banach space $X$, we denote by
 $\pairing{}{X}{\cdot}{\cdot}$ the duality pairing
between its dual space  $X'$ and $X$ itself and by $\Vert\cdot\Vert_X$
 the norm in $X$. In particular, we shall
use the following   short-hand notation for  function spaces
\[
\begin{gathered}
H:= L^2(\Omega), \quad V:=H^1(\Omega),\quad
 \bsH:= L^2 (\Omega;\R^3), \quad \bsV :=H^1 (\Omega;\R^3),
 \\  \Hc: = L^2
({\Gamma_c}), \quad
\Vc: = H^1 ({\Gamma_c}), \quad \Yc:= H^{1/2}_{00,\Gamma_1}({\Gamma_c}),
 \\
   \bsW:=\{{\bf v}\in \bsV\, : \ {\bf v}={\bf 0}\hbox{ a.e. on
}\Gamma_1\},\quad
\bsH_{{\Gamma_c}} := L^2 ({\Gamma_c};\R^3), \quad \bsY_{{\Gamma_c}}
:= H^{1/2}_{00,\Gamma_1}(\Gamma_c;\R^3),  
\\
\end{gathered}
\]
where we recall that
\[
 H^{1/2}_{00,\Gamma_1}(\Gamma_c)=
\Big\{ w\in H^{1/2}(\Gamma_c)\, :  \ \exists\,
\tilde{w}\in H^{1/2}(\Gamma) \text{ with }
\tilde{w}=w  \text{ in $\Gamma_c$,} \
\tilde{w}=0 \text{ in $\Gamma_1$} \Big\}
\]
 and $H^{1/2}_{00,\Gamma_1}(\Gamma_c;\R^3)$ is analogously defined. We will also use the space
 $H^{1/2}_{00,\Gamma_1}(\Gamma_2;\R^3).$
The space $\bsW$  is endowed with the natural norm induced by
$\bsV$. We will make use of the operator
\begin{equation}
A:\Vc \to \Vc' \qquad \pairing{}{\Vc}{A\chi}{w}:= \int_\gc \nabla
\chi \,\nabla w \dd x \  \text{ for all }\chi, \, w \in \Vc
\end{equation}
and of the notation 
\begin{equation}
\label{mean-value}
m(w):= \frac1{\mathcal{L}^d(A)} \int_A w \dd x \quad \text{for } w \in L^1(A).
\end{equation}
\end{notation}
\paragraph{Useful inequalities.}
 We  are going to exploit the following
trace results and continuous embeddings
\begin{equation}
\label{cont-embe}
 V \subset L^4(\gc), \qquad  \bsW \subset
L^{4}(\gc;\R^3), \qquad   \Vc \subset L^{q}(\gc) \text{ for all $1\leq
q<\infty$,} 
\end{equation}
and  the fact that, by Poincar\'e's inequality, for every Lipschitz domain $A \subset \R^d$
\begin{equation}
\label{equiv-norm} \exists\, C>0 \ \forall  v \in  W^{1,2}(A) \, :
\qquad \| v\|_{ W^{1,2}(A)} \leq C(\| v\|_{L^1(A)}+\| \nabla
v\|_{L^2(A)}),
\end{equation}
where $\| v\|_{L^1(A)}$ can be replaced by $|m(v)|$.
\paragraph{Linear viscoelasticity.}
We recall the definition of  the standard bilinear forms of linear
viscoelasticity, which are involved in the variational formulation
of equation~\eqref{eqI}. Dealing with an anisotropic and
inhomogeneous material,
 we assume that the fourth-order tensors $K_e=(a_{ijkh})$ and  $K_v=(b_{ijkh})$,
denoting the elasticity and the viscosity tensor, respectively,
satisfy the classical symmetry and ellipticity conditions
$$
\begin{aligned}
& a_{ijkh}=a_{jikh}=a_{khij}\,,\quad  \quad
b_{ijkh}=b_{jikh}=b_{khij} \quad \text{for }   i,j,k,h=1,2,3
\\
&  \exists \, \alpha_0>0 \,:  \qquad a_{ijkh} \xi_{ij}\xi_{kh}\geq
\alpha_0\xi_{ij}\xi_{ij} \quad   \forall\, \xi_{ij}\colon \xi_{ij}=
\xi_{ji}\quad  \text{for } i,j=1,2,3\,,
\\
& \exists \, \beta_0>0 \,:  \qquad b_{ijkh} \xi_{ij}\xi_{kh}\geq
\beta_0\xi_{ij}\xi_{ij} \quad   \forall\, \xi_{ij}\colon \xi_{ij}=
\xi_{ji}\quad  \text{for }  i,j=1,2,3\,,
\end{aligned}
$$
where the usual summation convention is used. Moreover, we require
$$
a_{ijkh}, b_{ijkh} \in L^{\infty}(\Omega)\,, \quad  i,j,k,h=1,2,3.
$$
By the previous assumptions on the elasticity and viscosity
coefficients, the following bilinear forms $a, b : \bsW \times \bsW
\to \R$,   defined~by
$$
\begin{aligned}
a({\bf u},{\bf v}):=\int_{\Omega} a_{ijkh} \varepsilon_{kh}({\bf u})
\varepsilon_{ij}({\bf v}) \dd x \quad  \text{for all } \uu, \vv \in \bsW,
\\
b({\bf u},{\bf v}):=\int_{\Omega} b_{ijkh} \varepsilon_{kh}({\bf u})
\varepsilon_{ij}({\bf v})  \dd x \quad  \text{for all }  \uu, \vv \in \bsW,
\end{aligned}
$$
turn out to be continuous and symmetric. In particular, we have
\begin{equation}
\label{continuity} \exists \, \bar{C} >0: \ |a(\uu, \vv)| +   |b(\uu,
\vv)| \leq  \bar{C}\| \uu\|_{\bsW} \| \vv\|_{\bsW} \quad \text{for all }
\uu, \vv \in \bsW.
\end{equation}
Moreover, since $\Gamma_1$ has positive measure,
 by Korn's inequality we deduce that $a(\cdot,\cdot)$ and
$b(\cdot,\cdot)$ are $\bsW$-elliptic, i.e., there exist $C_{a},
C_{b}>0 $ such that
\begin{align}
& a({\bf u},{\bf u})\geq C_a\Vert{\bf u}\Vert^2_{\bsW} \qquad
\text{for all }\uu\in \bsW, \,\label{korn_a}
\\
& b({\bf u},{\bf u})\geq C_b\Vert{\bf u}\Vert^2_{\bsW} \qquad
\text{for all }\uu\in \bsW. \label{korn_b}
\end{align}

\subsection{Assumptions}
\label{ss:2.2}
\noindent We specify all of the assumptions on the nonlinearities in
system (\ref{e1}--\ref{bord-chi}). \smallskip

\noindent \textbf{Hypothesis (I).} As for the \emph{friction
coefficient} $\fc: \R \to  (0,+\infty)$, we require that
\begin{equation}
\label{hyp-fc}
\begin{gathered}
 \fc \in \mathrm{C}^1(\R), \qquad \exists\, c_1,\, c_2
>0
 \ \forall\, x \in \R\, : \ \fc(x) \geq c_1, \ \
 |\fc'(x)| \leq c_2, \qquad
 \fc'(x) x \geq 0.
\end{gathered}
\end{equation}
For instance,
 the function $\fc(x) = \int_0^x \arctan(t)\dd t + c_1 =
x\arctan(x)-\frac12 \ln(1+x^2) +  c_1 $ complies with~\eqref{hyp-fc}.

\smallskip

 \noindent
 \textbf{Hypothesis (II).} We generalize the operator $\partial
I_{(-\infty,0]}$ 
by replacing $I_{(-\infty,0]}$  in \eqref{freeG}  with a function
\begin{equation}
\label{hyp:phi}
\begin{gathered}
 \phi : \R  \rightarrow [0, + \infty] \, \text{ proper, convex and  lower semicontinuous, with $\phi(0)=0$}
\end{gathered}
\end{equation}
and
effective domain $\mathrm{dom}(\phi)$;  let us emphasize that the \emph{physical case}, in which the constraint
$\uun \leq 0$ on $\gc$ is enforced, occurs when $\mathrm{dom}(\phi) \subset (-\infty,0]$  (and it is included in our analysis).
 Then, we define
\begin{equation}
\label{def-phi-scalare} \varphi:  \Yc \to [0,+\infty] \ \text{ by }
\ \varphi(v): = \begin{cases} \int_{\Gamma_c} \phi(v) \dd x & \text{if
$\phi(v) \in L^1 ({\Gamma_c})$,}
\\
+\infty
 & \text{otherwise.}
 \end{cases}
\end{equation}
Hence, we introduce
\begin{equation}
\label{funct-phi} \bvarphi: \bsY_{{\Gamma_c}}  \to [0,+\infty], \
\text{ defined by } \ \bvarphi(\uu):= \varphi(\uun) \quad \text{for
all $\uu  \in \bsY_{{\Gamma_c}}.$}
\end{equation}
Since $\bvarphi: \bsY_{{\Gamma_c}} \to [0,+\infty]$ is a proper,
convex and lower semicontinuous functional on $\bsY_{{\Gamma_c}}$,
its subdifferential $\partial \bvarphi: \bsY_{{\Gamma_c}}
\rightrightarrows \bsY_{{\Gamma_c}}'$ is a maximal monotone
operator.
\smallskip

\noindent \textbf{Hypothesis (III).} Concerning the  regularizing
operator $\Reg$, in the lines of \cite{bbr5} we require that there exists $\nu>0$ such that
\begin{align}
 \label{hyp-r-1}
&
\begin{aligned}
 & \Reg: L^2 (0,T;\bsY_{{\Gamma_c}}') \to L^\infty
(0,T;L^{2+\nu}(\gc;\R^3))  \text{  is weakly-strongly continuous, viz.}
\\
  &
  \eeta_n \weakto \eeta  \ \text{ in
$L^2 (0,T;\bsY_{{\Gamma_c}}')$} \ \  \Rightarrow \ \
 \Reg(\eeta_n) \to
\Reg(\eeta) \ \text{ in $L^\infty (0,T;L^{2+\nu}(\gc;\R^3)$}
\end{aligned}
\end{align}
for all  $(\eeta_n),\,\eeta \in L^2
(0,T;\bsY_{{\Gamma_c}}')$.  It is not difficult to check that
\eqref{hyp-r-1} implies that
$ \Reg: L^2 (0,T;\bsY_{{\Gamma_c}}') \to L^\infty
(0,T;L^{2+\nu}(\gc;\R^3)) $   is bounded.
In Example \ref{ex-reg} at the end of this section  we are going to exhibit an operator $\Reg$ complying with
\eqref{hyp-r-1}.
\smallskip

\noindent
 \textbf{Hypothesis (IV).} We generalize $\partial I_{[0,1]}$ in
\eqref{eqII} by considering  the  multivalued operator $\partial\widehat{\beta} : [0,+\infty)
\rightrightarrows  \R$, with
\begin{equation}
\label{A6} \tag{\textsc{H4}}
\begin{gathered}
 \widehat{\beta}: \R \rightarrow  (-\infty,+\infty] \, \text{ proper, convex and  lower semicontinuous,}
 \\
 \text{
  such that }
  \ \mathrm{dom}(\widehat \beta) \subset [ 0,+\infty). 
\end{gathered}
\end{equation}
In what follows, we use the notation $\beta:= \partial\widehat{\beta}$.
\smallskip

\noindent \textbf{Hypothesis (V).} We assume that the nonlinearities
$k$ in \eqref{condteta}--\eqref{eqtetas}, $\lambda$ in
\eqref{eqtetas} and \eqref{eqII}, and $\sigma$ in \eqref{eqII}
fulfill
\begin{align}
&  \label{hyp-k}    k \, : \R \to [0,+\infty)\ \  \text{is Lipschitz
continuous,}  
\\
&  \label{hyp-lambda} \lambda \in \mathrm{C}^{1} (\R), \ \text{with
} \lambda': \R \to \R \text{ Lipschitz continuous,}
\\
 &
 \label{hyp-sig} \sigma \in \mathrm{C}^{1} (\R), \ \text{with
} \sigma': \R \to \R \text{ Lipschitz continuous.}
\end{align}
\paragraph{Assumptions on the problem  and on the initial data.}
We  suppose that
\begin{align}
\label{hypo-h} &  h \in L^2 (0,T;V')\cap L^1 (0,T;H) \,,
\\
&  \label{hypo-f} \mathbf{f} \in L^2 (0,T;\mathbf{W}')\,,
\\
&  \label{hypo-g}
 \mathbf{g} \in L^2 (0,T;
H_{00,\Gamma_1}^{1/2}(\Gamma_2;\R^3)').
\end{align}
 It follows
from~\eqref{hypo-f}--\eqref{hypo-g} that,  the function $\mathbf{F}:(0,T)
\to \bfw'$  defined by
$$
 \pairing{}{\bsW}{\mathbf{F}(t)}{\vv}:=\pairing{}{\bsW}{\mathbf{f}(t)}{\vv}
 +\pairing{}{H_{00,\Gamma_1}^{1/2}(\Gamma_2;\R^3)}{\mathbf{g}(t)}{\vv}
\quad \text{for all } {\vv} \in  \bsW  \text{ and almost all } t \in
(0,T),
$$
 satisfies
\begin{equation}
\label{effegrande} \mathbf{F} \in L^2(0,T;\bfw') \,.
\end{equation}
Finally, we require that the initial data fulfill
\begin{align}
& \label{cond-teta-zero} \vartheta_0 \in L^{1}(\Omega)  \
 \ \ \text{and} \ \ \
\ln(\vartheta_0) \in H\,,
\\
& \label{cond-teta-esse-zero} \vartheta_s^0  \in L^{1}(\Gamma_c)
 \
 \ \ \text{and} \ \ \ \ln(\vartheta_s^0) \in \Hc\,,
\\
& \label{cond-uu-zero} {\bf u}_0 \in \bsW \ \text{and} \ \uu_{0} \in
\dom (\bvarphi)\,,
\\
& \label{cond-chi-zero} \chi_0 \in \Vc,  \ \
\widehat{\beta}(\chi_0)\in
 L^1(\Gamma_c)\,.
\end{align}
We conclude with the example, partially
mutuated from \cite[Ex.\ 2.4]{bbr5}, of an operator $\Reg$ complying with Hypothesis (III).
\begin{example}
 \label{ex-reg}
\upshape
Fix $\ell : \Gamma_c \times \Gamma_c \to \R^3$ such that $\ell \in L^{2+\nu} (\Gamma_c; \bsY_{{\Gamma_c}}) $
 for some $\nu>0$,
and for all $\eeta \in L^2
(0,T;\bsY_{{\Gamma_c}}')$  set
\[
\Reg(\eeta) (x,t) := \left( \int_0^t
\langle {\eeta(\cdot,s)},{\ell(x,\cdot)} \rangle_{\bsY_{{\Gamma_c}}} \dd
s\right) \mathbf{w} \quad \foraa\, (x,t) \in {\Gamma_c} \times
(0,T),
\]
where $\mathbf{w}$ is a fixed vector, e.g. with unit norm, in
$\R^3$.
 Then, for almost all $(x,t) \in \Gamma_c \times (0,T)$ there
holds
\begin{equation}
\label{interesting-bis}
|\Reg(\eeta) (x,t) | \leq \int_0^t |\langle{\eeta(\cdot,s)},{\ell(x,\cdot)}\rangle_{\bsY_{{\Gamma_c}}}|  \dd s
\leq t^{1/2} \| \ell(x,\cdot)\|_{\bsY_{{\Gamma_c}}} \| \eeta \|_{L^2 (0,t;\bsY_{{\Gamma_c}}') }.
\end{equation}
Integrating \eqref{interesting-bis} over $\gc$,
we easily conclude that  for all $t \in (0,T)$
\[
\|\Reg(\eeta) (\cdot,t) \|_{L^{2+\nu}(\gc;\R^3)} \leq  T^{1/2} \| \eeta \|_{L^2 (0,T;\bsY_{{\Gamma_c}}') } \| \ell \|_{L^{2+\nu} (\Gamma_c; \bsY_{{\Gamma_c}})},
\]
hence  $\Reg: L^2 (0,T;\bsY_{{\Gamma_c}}') \to L^\infty
(0,T;L^{2+\nu}(\gc;\R^3))$ is a linear and bounded operator.
Furthermore, it
fulfills
\eqref{hyp-r-1}. Indeed, let $\eeta_n \weakto \eeta $  in
$ L^2 (0,T;\bsY_{{\Gamma_c}}')$: for almost all $(x,t) \in \Gamma_c \times (0,T)$ we
have
\[
\begin{aligned}
\Reg(\eeta_n)(x,t)   =\int_0^T \langle {\eeta_n
(\cdot,s)},{\mathbf{1}_{(0,t)} \ell(x,\cdot)} \rangle_{\bsY_{{\Gamma_c}}} \, \mathbf{w} \dd s   \to
\int_0^T \langle{\eeta
(\cdot,s)},{\mathbf{1}_{(0,t)} \ell(x,\cdot)}\rangle_{\bsY_{{\Gamma_c}}} \, \mathbf{w}\dd s =
\Reg(\eeta)(x,t)
\end{aligned}
\]
as $n \to \infty$. Then,
 estimate
\eqref{interesting-bis} and the dominated convergence theorem
 yield $\Reg(\eeta_n) \to
\Reg(\eeta)$  in $L^q (0,T; L^{2+\nu}(\gc;\R^3))$ for all $1\leq q <\infty$. In order to conclude that
 $\Reg(\eeta_n) \to
\Reg(\eeta)$  in $L^\infty(0,T; L^{2+\nu}(\gc;\R^3))$, it is sufficient to observe that the sequence  $(\Reg(\eeta_n) )_n$ is in fact
compact in $\mathrm{C}^0 ([0,T]; L^{2+\nu}(\gc;\R^3))$.  This follows from the Ascoli-Arzel\`a theorem, since,
arguing as for \eqref{interesting-bis}, it is not difficult to see that,
  if $(\eeta_n)_n \subset L^2 (0,T;\bsY_{{\Gamma_c}}')$ is bounded, then
 $(\Reg(\eeta_n) )_n$ fulfills the equicontinuity condition
for all $0 \leq s \leq t \leq T$
\[
\|\Reg(\eeta_n) (\cdot,t) - \Reg(\eeta_n) (\cdot,s) \|_{L^{2+\nu}(\gc;\R^3)} \leq  (t-s)^{1/2} \| \eeta_n \|_{L^2 (s,t;\bsY_{{\Gamma_c}}') } \| \ell\|_{L^{2+\nu} (\Gamma_c; \bsY_{{\Gamma_c}})} \leq C (t-s)^{1/2}.
\]
\end{example}

\subsection{Statement of the main result}
\label{ss:2.3}
 \noindent
We now specify the variational formulation of system
(\ref{e1}--\ref{bord-chi}).
 \begin{problem}
 \label{prob:rev}
 \upshape
 Given a quadruple of initial
data
 $(\vartheta_0, \vartheta_s^0 , \uu_0, \chi_0)$
 fulfilling~\eqref{cond-teta-zero}--\eqref{cond-chi-zero}, find a
 seventuple
$(\vartheta,  \vartheta_s,   \uu,\chi,\eeta,\mmu, \xi)$, with
\begin{align}
&
\label{reg-teta}
 \vartheta \in L^2 (0,T; V) \cap L^\infty (0,
T;L^1 (\Omega))
\,,
\\
\label{reg-log-teta}
 & \ln(\vartheta) \in L^\infty (0,T;H) \cap H^1
(0,T;V')\,,
 \\
&
\label{reg-teta-s}
\vartheta_s \in  L^2 (0,T; \Vc) \cap L^\infty (0, T;L^1
(\Gamma_c))\,, \\ &
\label{reg-log-teta-s}
\ln(\vartheta_s) \in L^\infty (0,T;\Hc) \cap H^1
(0,T;\Vc')\,,
\\
& \uu \in H^1(0,T;\bsW) \,,\label{reguI}\\
&\chi \in L^{2}(0,T;H^2 (\Gamma_c))  \cap L^{\infty}(0,T;\Vc) \cap
H^1 (0,T;\Hc)\,, \label{regchiI}
\\
 &
\eeta\in L^2(0,T; \bsY'_{{\Gamma_c}})\,, \label{etareg}
\\ & \mmu \in L^2(0,T;\bsH_{{\Gamma_c}})\,,\label{mureg}
\\
& \xi\in L^2(0,T; \Hc)\,, \label{xireg}
\end{align}
 satisfying  the initial
 conditions
\begin{align}
& \label{iniw} \vartheta(0)=\vartheta_0 \quad \aein \ \Omega\,,
\\
& \label{iniz} \vartheta_s(0)=\vartheta_s^0 \quad \aein \
\Gamma_c\,,
\\
& \label{iniu} \uu(0)=\uu_0 \quad {\aein \ \Omega}\,,
\\
& \label{inichi} \chi(0)=\chi_0 \quad {\aein \ \Gamma_c}\,,
\end{align}
  and
 \begin{align}
&
 \label{teta-weak}
\begin{aligned}
\pairing{}{V}{\partial_t \ln(\vartheta)}{v}
 &  -\int_{\Omega} \dive(\partial_t\mathbf{u}) \, v \dd x
+\int_{\Omega}   \nabla \vartheta \, \nabla v  \dd x +
\int_{\Gamma_c} k(\chi)  (\vartheta-\vartheta_s) v  \dd x
\\ & +\int_{\Gamma_c} \fc'(\teta-\teta_s) |\Reg (\eeta)| |\dotut| v  \dd x
 = \pairing{}{V}{h}{v} \quad
\forall\, v \in V \ \hbox{ a.e. in }\, (0,T)\,,
\end{aligned}
\\
& \label{teta-s-weak}
\begin{aligned}
 & \pairing{}{\Vc}{\partial_t \ln(\vartheta_s)}{v}
-\int_{\Gamma_c}
\partial_t \lambda(\chi) \, v  \dd x    +\int_{\Gamma_c} \nabla \vartheta_s  \, \nabla
v \dd x
\\ &
\qquad  = \int_{\Gamma_c} k(\chi) (\vartheta-\vartheta_s) v  \dd x
 +\int_{\Gamma_c} \fc'(\teta-\teta_s) |\Reg (\eeta)| |\dotut| v  \dd x
  \quad
\forall\, v \in \Vc  \ \hbox{ a.e. in }\, (0,T)\,,
\end{aligned}
\\
 &
\begin{aligned}
 &  b(\dotu,\vv)  +a(\uu,\vv)+ \int_{\Omega} \vartheta \dive (\vv)  \dd x
 +\int_{\Gamma_c} \chi \uu \cdot \vv \dd x  \\
&
\qquad  + \langle {\mbox{\boldmath$\eta$}} , \vv
\rangle_{
\bsY_{{\Gamma_c}}}
   +
\int_{{{\Gamma_c}}}\fc(\teta-\vartheta_s)   {\mmu}\cdot {\vv} \dd x =
\pairing{}{\bsW}{\mathbf{F}}{\vv}
 \quad \text{for all } \vv\in \bsW \ \hbox{ a.e. in }\, (0,T)\,,
 \end{aligned}
\label{eqIa}
\\
&\eeta\in
\partial \bvarphi(\uu)  \text{ in $\bsY'_{{\Gamma_c}}$, }
\hbox { a.e. in }\,
(0,T), \label{incl1}
\\
& \mmu =   |\Reg(\eeta)|    \zz  \ \text{ with } \  \zz \in{\bf d}(\dotu)
\hbox { a.e. in }\,
{\Gamma_c}\times (0,T),  \label{incl1-bis}
\\
&\partial_t{\chi}+A\chi+\xi+\sigma'(\chi)= -\lambda'(\chi)
\vartheta_s- \frac12 |\uu|^2 \quad\hbox{a.e. in } {\Gamma_c}\times
(0,T)\,,\label{eqIIa}\\
&\xi\in \beta(\chi)\hbox { a.e. in }\, {\Gamma_c} \times (0,T).
\label{inclvincolo}
\end{align}
\end{problem}
\noindent    Note that, to simplify notation we have
incorporated the contribution $-\lambda'(\chi)\teta_{\mathrm{eq}}$
occurring in \eqref{eqII} into the term $\sigma'(\chi)$ in \eqref{eqIIa}.
\begin{remark}
\upshape
\label{rmk:meaning}
Observe that the subdifferentials  in \eqref{incl1-bis} and \eqref{inclvincolo} are
multivalued operators  on $\R^3$ with values in $\R^3$, and on $\R$ with values in $\R$, respectively.
 Hence the related subdifferential inclusions for
$\zz$ and $\xi$ hold a.e.\ in $\gc \times (0,T)$. Instead,
\eqref{incl1} features the (abstract) operator $\partial  \bvarphi :\bsY_{{\Gamma_c}}
\rightrightarrows \bsY'_{{\Gamma_c}}$, thus \eqref{incl1} holds in $\bsY'_{{\Gamma_c}}$.
Note that, in the case we further  assume (for physical consistency) that dom$\phi\subset(-\infty,0]$,
 it still  follows
from the definition \eqref{funct-phi} of $\bvarphi$ that, if  $\uu \in \dom(\partial  \bvarphi)$,
then $\uu$ complies with the constraint $\uun \leq 0$ a.e.\ in $\gc \times (0,T)$.
\end{remark}
 We are now in the position to state our existence theorem for
Problem \ref{prob:rev}. Observe that we obtain enhanced regularity in time for
$\teta$ and $\teta_s$ thanks to some refined $\BV$-estimates (cf.\  Remark \ref{rmk:toBV}
later on).  Relying on \eqref{hyp-r-1} in Hypothesis (III), we also find \eqref{mureg-enhanced} for $\mmu$.
\begin{maintheorem}
\label{mainth:1}
$ $\\
 In the framework of \eqref{assumpt-domain},
 under Hypotheses (I)--(V) and conditions
\eqref{hypo-h}--\eqref{hypo-g} on the data $h$, $\mathbf{f}$,  $\mathbf{g}$, and
\eqref{cond-teta-zero}--\eqref{cond-chi-zero} on $\teta_0,$
$\teta_s^0,$ $\uu_0$, and $\chi_0$, Problem
\ref{prob:rev} admits at least a solution $(\vartheta,  \vartheta_s,
\uu,\chi,\eeta,\mmu, \xi)$, which in addition satisfies
\begin{align}
&
\label{bv-reg-teta} \teta \in \mathrm{BV}(0,T;W^{1,q}(\Omega)') \text{
for every $q>3$,}
\\
&
\label{bv-reg-teta-s}
 \teta_s \in
\mathrm{BV}(0,T;W^{1,\sigma}(\gc)')  \text{ for every $\sigma>2$},
\\
&
\label{mureg-enhanced}
\mmu \in L^\infty(0,T; L^{2+\nu}(\Gamma_c;\R^3)) \qquad \text{with $\nu>0$ from \eqref{hyp-r-1}.}
\end{align}
\end{maintheorem}
The \emph{proof} will be developed  throughout Sections \ref{s:3}--\ref{s:3.3}. First, we will
analyze a suitable approximation  of Problem \ref{prob:rev},
 for which we will obtain the existence of solutions in
 Thm.\ \ref{th:exist-approx-rev}.
In Section \ref{s:3.3} we will then pass to the limit and show
that (up to a subsequence) the approximate solutions converge to a
solution of Problem  \ref{prob:rev}.

\begin{notation}
\upshape In what follows, we will  denote most of the positive constants
occurring in the calculations  by the symbols $c,\, c', \, C,\,C'$,
whose meaning may vary even within the same line.
Furthermore, the symbols $I_i,$ $i=0,1,\ldots,$
 will be used as place-holders for  several integral terms popping in the various
 estimates: we warn the reader that we will not be self-consistent with the numbering, so that, for instance,
 the symbol
$I_1$ will appear several times with different meanings.
\end{notation}

\section{Approximation}
\label{s:3}
Here we focus on the approximation of the PDE system  (\ref{teta-weak}--\ref{inclvincolo})
through
Yosida-type regularization of some of the (maximal monotone) nonlinearities therein involved.
 For the related  definitions and results we refer
 to the classical monographs \cite{barbu76, brezis73}.

In Sec.\ \ref{s:3.1} we justify
the regularizations we will perform, giving raise to the approximate Problem $(P_\eps)$. The existence of
solutions is proved in two steps: first,
 in Sec.\ \ref{s:3.2} we
show that Problem $(P_\eps)$ admits local-in-time solutions
by means of a Schauder fixed point argument. We then extend them to
solutions on the whole $[0,T]$ relying on the  \emph{global} a priori estimates which
we derive in Sec.\ \ref{ss:3.3}. In fact, we will obtain
 estimates \emph{independent} of the approximation parameter $\eps$:
they will be the starting point for the passage to the limit
as $\eps \down 0 $ developed in Section \ref{s:3.3}.
\subsection{The approximate problem}
\label{s:3.1}
In order to motivate the way we are going to
approximate Problem \ref{prob:rev}, let us discuss in advance some of the
\emph{global} a priori estimates
to be performed on system (\ref{teta-weak}--\ref{inclvincolo}).
Clearly,  these estimates
correspond to (some of)
the summability and regularity properties \eqref{reg-teta}--\eqref{xireg} required for solutions
to Problem \ref{prob:rev}. As we will see, the  related  calculations can be developed on
system (\ref{teta-weak}--\ref{inclvincolo}) only on a \emph{formal} level: they  can be made rigorous
by means of the Yosida-type regularizations we will consider (cf.\ the global a priori estimates on
the approximate solutions performed in Section \ref{ss:3.3}).
\paragraph{Heuristics for the approximate problem.} The
basic \emph{energy estimate} for system (\ref{teta-weak}--\ref{inclvincolo})
consists  in
testing \eqref{teta-weak} by
$\vartheta$, \eqref{teta-s-weak} by $\vartheta_s$, \eqref{eqIa} by
$\partial_t\uu$,  \eqref{eqIIa} by $\partial_t\chi$, adding the resulting
relations, and integrating in time. Taking into account the \emph{formal} identities
\begin{equation}
\label{e:formal1}
\begin{aligned}
& \int_0^t \pairing{}{V}{\partial_t \ln(\vartheta)}{ \vartheta } \dd
r = \| \vartheta(t) \|_{L^1 (\Omega)} - \| \vartheta_0 \|_{L^1
(\Omega)}\,,
\\
& \int_0^t  \pairing{}{\Vc}{\partial_t \ln(\vartheta_s)}{
\vartheta_s } \dd r= \| \vartheta_s(t) \|_{L^1 (\Gamma_c)} - \|
\vartheta_s^0 \|_{L^1 (\Gamma_c)}\,, \end{aligned}
\end{equation}
this estimate leads, among others, to a bound for $\teta$ in $L^\infty (0,T;L^1(\Omega))$
and, correspondingly,
 for $\teta_s$ in $L^\infty (0,T;L^1(\Gamma_c))$ (cf.\ \eqref{reg-teta} and \eqref{reg-teta-s}).
As a first step towards making \eqref{e:formal1} rigorous, following
\cite{bcfg1, bbr4}
we will
\begin{enumerate}
\item replace the logarithm $\ln$ in equations \eqref{teta-weak} and \eqref{teta-s-weak} by its approximation $\applog: \R \to \R$
\begin{equation}
\label{def-applog}
\applog(r):=\varepsilon r+\ln_\varepsilon (r),
\end{equation}
where for $\eps>0$
$\ln_\varepsilon$ denotes the Yosida regularization of the logarithm $\ln$.
 Therefore,
 $\applog$  is differentiable, strictly increasing  and Lipschitz continuous, see also Lemma \ref{l:new-lemma2} below.
\end{enumerate}
As pointed out in Remark \ref{rem:formal} below, this
procedure
 is not sufficient to  justify
\eqref{e:formal1} completely. In order to do so,
following \cite{bbr3, bbr4}
we should also add a viscosity term
both in \eqref{teta-weak} and \eqref{teta-s-weak}, modulated by a second parameter $\nu$. 
We choose to overlook this point
for the sake of simplicity. Anyhow, let us stress that
approximating the logarithm $\ln$ by $\applog$ makes   the
 test of  \eqref{teta-weak} by $\ln(\teta)$ rigorous
(and, respectively, it justifies  the test  of \eqref{teta-s-weak} by $\ln(\teta_s)$).
 This gives raise to an estimate for $\ln(\teta)$
in $L^\infty (0,T;H)$ (for $\ln(\teta_s)$ in $L^\infty (0,T;H_{\Gamma_c})$, respectively).

A consequence of the aforementioned \emph{energy estimate} and of a comparison argument
in  the momentum equation \eqref{eqIa} is the following estimate
\begin{equation}
\label{comparison-1} \| \fc(\teta-\teta_s)\mmu + \eeta \|_{L^2
(0,T;\bsY_{{\Gamma_c}}')} \leq C.
\end{equation}
Using the \emph{formal} observation that $\mmu$ and $\eeta$ are
\emph{orthogonal},  from
\eqref{comparison-1} one concludes the crucial information
\begin{equation}
\label{second-aprio}
 \| \fc(\teta-\teta_s)\mmu\|_{L^2
(0,T;\bsY_{{\Gamma_c}}')} + \| \eeta \|_{L^2
(0,T;\bsY_{{\Gamma_c}}')} \leq C. \end{equation}
In order to justify this argument,
 we need to
 suitably approximate the maximal monotone operator
  $\partial\bvarphi: \bsY_{{\Gamma_c}} \rightrightarrows
  \bsY_{{\Gamma_c}}'$
  in such a way as to replace $\eeta \in
  \bsY_{{\Gamma_c}}'$ in \eqref{eqIa} with a term $\eeta_\eps$
  having \emph{null tangential component}, cf.\
  \eqref{null-component} below.
 Following \cite{bbr5}, we will thus
\begin{enumerate}
 \setcounter{enumi}{1}
\item   replace the function $\phi$, which enters in
the definition of the functional $\bvarphi$ through
\eqref{def-phi-scalare} and \eqref{funct-phi}, by its Yosida
approximation $\phi_\eps: \R \to [0,+\infty)$.
 We recall that $\phi_\eps $ is convex, differentiable, and  such that $\phi_\eps'$  is
 the Yosida regularization of the subdifferential $\partial\phi: \R \rightrightarrows \R$.
\end{enumerate}
 Therefore, in this way  we will  replace the constraint
\eqref{incl1} by its regularized version
\begin{equation}
\label{null-component} \eeta_\eps : =\phi_\eps'(\uun){\bf n}
\quad\hbox{a.e. in }\Gamma_c\times(0,T).
\end{equation}

\noindent
Furthermore, we introduce the function $\mathcal{I}_\eps: \R \to \R$ defined by
 \begin{equation}
\label{mathcal-i-eps}
\mathcal{I}_\eps(x):= \int_0^x s \applog'(s) \dd s
\end{equation}
which we will use
in the proof of
the existence of solutions to Problem $(P_\eps)$,
and in particular in the derivation of suitable a priori estimates.
In the two following lemmas, we collect some useful  properties
 of the functions $\ln_\eps$,  $\applog$, and  $\mathcal{I}_\eps$.
\begin{lemma}
\label{l:new-lemma1} The following inequalities hold:
\begin{subequations}
\label{ine-lprimo}
\begin{align}
& \label{ine-lprimo-1} \exists\, \eps_*>0\,:  \ \ \forall\, \eps \in
(0,\eps_*) \ \ \forall\, x>0 \quad \ln_{\eps}'(x) \leq
\frac{2}{x}\,,
\\
& \label{ine-lprimo-2}\forall\, \eps >0 \ \ \forall\, x \in \R \quad
\ln_{\eps}'(x) \geq \frac{1}{|x| + 2 + \eps}\,.
\end{align}
\end{subequations}
As a consequence, the function
$\mathcal{I}_\eps: \R \to \R$  \eqref{mathcal-i-eps}
 satisfies
\begin{subequations}
\label{ine-imu}
\begin{align}
& \label{ine-imu-1} \exists\, \eps_*>0\,:  \ \ \forall\, \eps \in
(0,\eps_*) \ \ \forall\, x \geq 0  \quad  \mathcal{I}_\eps(x) \leq  \frac{\eps} {2} x^2+2x\,;
\\
& \label{ine-imu-2} \exists\, C_1, \, C_2 >0 \, : \ \  \forall\, \eps
>0 \ \ \forall\, x \in \R \quad  \mathcal{I}_\eps(x) \geq \frac{\eps} {2} x^2+ C_1 |x|
 - C_2\,.
 \end{align}
\end{subequations}
\end{lemma}
\begin{proof} Estimates \eqref{ine-lprimo-1} and \eqref{ine-lprimo-2} have been derived in
\cite[Lemma~4.1]{bbr4} (cf.\ also \cite[Lemma 4.2]{bcfg1}).
Inequalities  \eqref{ine-imu-1} and \eqref{ine-imu-2} can be deduced from
\eqref{ine-lprimo-1}--\eqref{ine-lprimo-2}
by arguing
in a completely analogous way as in the proof of \cite[Lemma~4.1]{bbr4}, to which the reader is referred.
\end{proof}
\begin{lemma}
\label{l:new-lemma2}  The function $\applog:\R \to \R$
 satisfies:
\begin{align}
& \label{bi-Lip} 
\eps <\applog'(x) \leq \eps+\frac{2}{\eps}
 \quad \text{for all } x \in \R\,,
\\
& \label{Lip}
\left|\frac{1}{\applog'(x)}-\frac{1}{\applog'(y)}\right|\leq  |x-y| \quad \text{for all } x,y \in \R \,.
\end{align}
\end{lemma}
\begin{proof}
The left-hand side  inequality in \eqref{bi-Lip}
directly follows from the definition \eqref{def-applog} of $\applog$ and from estimate \eqref{ine-lprimo-2}
for $\ln_{\eps}'$. In order to prove the right-hand side estimate, it is sufficient to show that
\begin{equation}
\label{applog-lip}
 |\applog(x)-\applog(y)| \leq  \left( \eps+\frac{2}{\eps}\right)|x-y|
\ \ \text{for all } \, x,y \in \R\,.
\end{equation}
This can be checked by recalling that
the Yosida-regularization $\ln_\eps$ is defined by
\begin{equation}
\label{def-ln-eps}
\ln_\eps(x):=\frac 1 \eps(x-\rho_\eps(x))
\end{equation}
where $\rho_\eps=(\mathrm{Id}{+}\eps \ln)^{-1}:\R \to  (0,+\infty)$ is the $\eps$-resolvent of $\ln$.
Plugging \eqref{def-ln-eps} into the definition \eqref{def-applog} of $\applog$ and using the well-known fact that
$\rho_\eps$ is a contraction, we immediately deduce \eqref{applog-lip}.

Observe that, by  the first of \eqref{bi-Lip} the function $x\mapsto \frac1{\applog'(x)}$ is well-defined
on $\R$.
In order to show  that it is itself a  contraction, we use the  formula
\begin{equation}
\label{very-useful-formula}
\ln_\eps'(x)= \frac1{\rho_\eps(x)+\eps} \qquad \text{for all } x \in \R
\end{equation}
(cf.\ \cite[Lemma 4.2]{bcfg1}).
Therefore,
for every $x,\,y \in \R$
\[
\begin{aligned}
\left|\frac1{\applog'(x)}-\frac1{\applog'(y)}\right|  & = \left|\frac{\rho_\eps(x)+\eps}{\eps^2+\eps\rho_\eps(x)+1}-
\frac{\rho_\eps(y)+\eps}{\eps^2+\eps\rho_\eps(y)+1}\right|
\\
&
=
\frac{|\rho_\eps(x){-}\rho_\eps(y)|}{(\eps^2+\eps\rho_\eps(x)+1)( \eps^2+\eps\rho_\eps(y)+1)}
\leq |\rho_\eps(x){-}\rho_\eps(y)|
\end{aligned}
\]
and \eqref{Lip} ensues, taking into account that $\rho_\eps$ is a contraction.
\end{proof}


\medskip
Furthermore,
we will supplement Problem $(P_\eps)$ by the approximate  initial data $(\vartheta_0^\eps,\vartheta_s^{0,\eps},
\uu_0,\chi_0)$, where the family
 $(\vartheta_0^\eps, \vartheta_s^{0,\eps})_\eps$ fulfills
 \begin{subequations}
 \label{dati-teta-better}
\begin{equation}
\label{add-appr-data} (\vartheta_0^\eps)_\eps \subset H, \qquad
(\vartheta_s^{0,\eps})_\eps \subset H_{\Gamma_c},
\end{equation}
and it approximates the data $ (\teta_0,\teta_s^0)$  from \eqref{cond-teta-zero}--\eqref{cond-teta-esse-zero}
in the sense that
\begin{align}
&\label{convdatiteta}
(\vartheta_0^\eps, \vartheta_s^{0,\eps}) \to (\teta_0,\teta_s^0) \text{ in } L^1(\Omega)\times L^1(\Gamma_c)
\text{ as $\eps\down0$\,,}
\\
&\label{bounddatiteta-bis}
 \|\ln_\eps(\vartheta_0^\eps)\|_H\leq\|\ln(\teta_0)\|_H\,,\qquad \|\ln_\eps(\vartheta_s^{0,\eps})\|_{\Hc}\leq \|\ln(\teta_s^0)\|_{\Hc} \quad \text{for all} \, \eps>0\,,
 \\
&
\label{bounddatiteta-ter}
\exists\, \bar{S}>0\,: \ \ \forall\, \eps>0 \quad
\eps (\|\vartheta_0^\eps\|_H+ \|\vartheta_s^{0,\eps}\|_{\Hc})\leq \bar{S}\,,
\\
&
\label{bounddatiteta-1}
\exists\, \bar{S_1}>0\,: \ \ \forall\, \eps>0 \quad
\int_\Omega \mathcal{I}_\eps(\teta_0^\eps(x))\dd x \leq \bar{S_1}(1+\|\teta_0\|_{L^1(\Omega)})\,,
\\
&
\label{bounddatiteta-2}
\exists\, \bar{S_2}>0\,: \ \ \forall\, \eps>0 \quad
\quad \int_\gc \mathcal{I}_\eps(\teta_s^{0,\eps})\dd x \leq \bar{S_2}(1+\|\teta_s^0\|_{L^1(\gc)}).
\end{align}
\end{subequations}
Indeed,  \eqref{add-appr-data} reflects
the enhanced regularity
\eqref{reg-teta-app} and \eqref{reg-teta-s-app} required of solutions $(\teta,\teta_s)$ to Problem $(P_\eps)$.
In what follows, we give an example of construction of a sequence $(\vartheta_0^\eps, \vartheta_s^{0,\eps})_\eps$
fulfilling properties \eqref{dati-teta-better}.
\begin{example}
 \label{ex-dati}
\upshape
We will carry out the construction of the
sequence $(\vartheta_0^\eps)_\eps$ only, the argument for $(\vartheta_s^{0,\eps}
)_\eps$ being completely analogous. For all $\eps>0$  and a.e. in $\Omega$, let us define
\begin{equation}
\label{costr-appr-data} \vartheta_0^\eps:=\min\{\vartheta_0, \eps^{-\alpha}\}\qquad
\quad \text{for some } \alpha>0\,.
\end{equation}
Observe that $\vartheta_0^\eps>0$ a.e. in $\Omega$ (being $\teta_0 >0$
a.e. in $\Omega$ thanks to the second of~\eqref{cond-teta-zero}), $\vartheta_0^\eps\in L^\infty(\Omega)$,  and moreover
\begin{equation}
\label{appr-data_1} \|\vartheta_0^\eps\|_{L^1(\Omega)}\leq\|\vartheta_0\|_{L^1(\Omega)}\,.
\end{equation}
Hence,
\eqref{convdatiteta} is an immediate consequence of the dominated convergence theorem.
Next, noting that
$\|\vartheta_0^\eps\|^2_{H}\leq |\Omega|\, \eps^{-2\alpha}$ and choosing
$\alpha=\frac12$, we can deduce
\begin{equation}
\label{bounddatiteta}
\exists\, \bar{S}>0\,: \ \ \forall\, \eps>0 \quad \varepsilon^{1/2}\|\teta_0^\eps\|_H\leq \bar{S}\
\end{equation}
and also \eqref{bounddatiteta-ter} follows.
Moreover, \eqref{ine-imu-1},   \eqref{appr-data_1}, and  \eqref{bounddatiteta}
give \eqref{bounddatiteta-1}.
Finally, relying on the definition \eqref{costr-appr-data} and on  well-known properties of the Yosida regularization $\ln_\eps$, we find
\begin{equation}
\label{appr-data_2} |\ln_\eps(\vartheta_0^\eps)|\leq |\ln_\eps(\vartheta_0)|\leq |\ln(\vartheta_0)|
 \quad \text{a.e. in } \Omega \,
\end{equation}
whence \eqref{bounddatiteta-bis}.
\end{example}
All in all, the approximation of Problem \ref{prob:rev}
reads:
\begin{problem}[$P_\eps$]
 \label{prob:rev-app}
 \upshape
Let us a consider a quadruple of initial data
 $(\vartheta_0^\eps, \vartheta_s^{0,\eps} , \uu_0, \chi_0)$
  satisfying \eqref{cond-uu-zero}--\eqref{cond-chi-zero} and \eqref{dati-teta-better}. Find
 a
 sextuple
$(\vartheta,  \vartheta_s,   \uu,\chi,\mmu, \xi)$, fulfilling
\begin{align}
&
\label{reg-teta-app}
\teta \in  L^2 (0,T;\V) \cap \mathrm{C}^0 ([0,T];H),
\\
&
\label{reg-log-teta-app}
\applog(\teta) \in L^2 (0,T;\V) \cap \mathrm{C}^0 ([0,T];H) \cap H^1 (0,T;\V'),
\\
&
\label{reg-teta-s-app}
 \teta_s \in  L^2 (0,T;\Vc) \cap  \mathrm{C}^0 ([0,T];\Hc),
\\
&
\label{reg-log-teta-s-app}
\applog(\teta_s) \in L^2 (0,T;\Vc) \cap \mathrm{C}^0 ([0,T];\Hc) \cap H^1 (0,T;\Vc'),
\end{align}
and such that $(\uu,\chi,\mmu,\xi)$ comply with
\eqref{reguI}--\eqref{regchiI},  \eqref{mureg-enhanced},  \eqref{xireg},
 satisfy  the initial
 conditions
 \begin{align}
& \label{iniw-better} \vartheta(0)=\vartheta_0^\eps \quad \aein \ \Omega\,,
\\
& \label{iniz-better} \vartheta_s(0)=\vartheta_s^{0,\eps} \quad \aein \
\Gamma_c\,,
\end{align}
 as well as  \eqref{iniu}--\eqref{inichi},
  and the equations
 \begin{align}
&
 \label{teta-weak-app}
\begin{aligned}
 &  \pairing{}{\V}{\partial_t\applog(\teta)}{ v}
   -\int_{\Omega} \dive(\partial_t\mathbf{u}) \, v  \dd x
+\int_{\Omega}   \nabla \vartheta \, \nabla v  \dd x + \int_{\Gamma_c}
k(\chi)  (\vartheta-\vartheta_s) v   \dd x
\\ & \qquad +\int_{\Gamma_c}
\fc'(\teta-\teta_s) |\Reg (\phi_\eps'(\uun){\bf n})| |\dotut| v  \dd x
 = \pairing{}{V}{h}{v} \quad
\forall\, v \in V \ \hbox{ a.e. in }\, (0,T)\,,
\end{aligned}
\\
& \label{teta-s-weak-app}
\begin{aligned}
 &  \pairing{}{\Vc}{\partial_t\applog(\teta_s)}{ v}  -\int_{\Gamma_c}
\partial_t \lambda(\chi) \, v   \dd x     +\int_{\Gamma_c} \nabla \vartheta_s  \, \nabla
v \dd x
\\ &
\  = \int_{\Gamma_c} k(\chi) (\vartheta-\vartheta_s) v\dd x
 +\int_{\Gamma_c} \fc'(\teta-\teta_s) |\Reg (\phi_\eps'(\uun){\bf n})| |\dotut| v \dd x
  \quad
\forall\, v \in \Vc  \ \hbox{ a.e. in }\, (0,T)\,,
\end{aligned}
\\
 &
 \label{eqIa-app}
\begin{aligned}
 &  b(\dotu,\vv)  +a(\uu,\vv)+ \int_{\Omega} \vartheta \dive (\vv)\dd x
 +\int_{\Gamma_c} \chi \uu \cdot \vv \dd x\\
& \quad + \int_\gc \phi_\eps'(\uun){\bf n}\cdot  \vv\dd x
   +
\int_{{{\Gamma_c}}}\fc(\teta-\vartheta_s)   {\mmu}\cdot {\vv}\dd x =
\pairing{}{\bsW}{\mathbf{F}}{\vv}
 \quad \text{for all } \vv\in \bsW \ \hbox{ a.e. in }\, (0,T)\,,
  \end{aligned}
 \\
 &
 \label{to-be-quoted-also}
 \mmu = |\Reg(\phi_\eps'(\uun){\bf n})|  \zz   \text{ with }  \zz\in {\bf d}(\dotu)\hbox { a.e. in }\,
{\Gamma_c}\times (0,T),
\end{align}
as well as relations \eqref{eqIIa}--\eqref{inclvincolo}.
\end{problem}
\noindent
Observe that, in \eqref{teta-weak-app}, \eqref{teta-s-weak-app}, and \eqref{to-be-quoted-also}
the term $\Reg(\phi_\eps'(\uun){\bf n})$ needs to be understood as
$\Reg(J(\phi_\eps'(\uun){\bf n}))$, where $J$ denotes the
embedding operator from   $L^2(0,T;L^4(\Gamma_c;\R^3))$ to $ L^2(0,T;\bsY_{{\Gamma_c}}')$.

In the next two sections, we address the existence of solutions to Problem $(P_\eps)$
for $\eps>0$ fixed. That is why, for notational convenience we will not specify the
dependence of such solutions on   the parameter  $\eps$.
\subsection{The approximate problem: local existence}
\label{s:3.2}
 The main result of this section is the forthcoming
Proposition
\ref{prop:loc-exist-eps}, stating the
 existence of a local-in-time solution to Problem $(P_\eps)$. We prove it
 by means of a Schauder fixed point argument, which relies on intermediate well-posedness results
for the single equations in Problem  $(P_\eps)$.
\paragraph{Fixed point setup.}
In view of Hypothesis (III), 
 we may choose $\delta \in (0,1)$
such that
\begin{equation}\label{hyp-r-nuova}
\Reg:L^2(0,T;\bsY'_{\Gamma_c})\rightarrow L^\infty(0,T;L^{\frac 2{1-\delta}}(\Gamma_c;\R^3))
\text{ is weakly-strongly continuous}
\end{equation}
(and therefore bounded).
For a fixed $\tau>0$ and a fixed constant $M>0$, we consider the set
\begin{equation}
\label{fixed-point-set}
\begin{aligned}
\! \!\!\!\! \!\!
 \!{\cal Y}_\tau=\{ & (\teta,
\teta_s,\uu,\chi)\in L^2(0,\tau;H^{1-\delta}(\Omega))\times L^2(0,\tau;H^{1-\delta}(\Gamma_c))
\times L^2(0,\tau;H^{1-\delta}(\Omega;\R^3))
\times X_\tau\,: \\
&  \|\teta\|_{L^2(0,\tau;H^{1-\delta}(\Omega))}+\|\teta_s\|_{L^2(0,\tau;H^{1-\delta}(\Gamma_c))}
+ \|\uu\|_{L^2(0,\tau;H^{1-\delta}(\Omega;\R^3))}+ \|\chi\|_{L^2(0,\tau;\Hc)}
\leq M,
\end{aligned}
\end{equation}
 with the topology of $L^2(0,\tau;H^{1-\delta}(\Omega))\times L^2(0,\tau;H^{1-\delta}(\Gamma_c))
\times L^2(0,\tau;H^{1-\delta}(\Omega;\R^3))
\times L^2(0,\tau;\Hc)$, where
$$
X_\tau:=\{\chi\in L^2(0,\tau;\Hc) \, :  \quad \chi\in \mathrm{dom}(\widehat{\beta})\hbox{ a.e.\ on }\Gamma_c
\times (0,\tau)\}.
$$
We are going to construct an operator $\mathcal{T}$ mapping
${\cal Y}_{\widehat{T}}$ into itself for a suitable time
$0\leq \widehat{T} \leq T$, in such a
way that any fixed point of $\mathcal{T}$ yields a solution to Problem $(P_\eps)$
on the interval $(0,\widehat{T})$. In the proof of Proposition \ref{prop:loc-exist-eps},
we will then proceed to show  that $\mathcal{T}:{\cal Y}_{\widehat{T}}\to
{\cal Y}_{\widehat{T}}$ does admit a fixed point.
\begin{notation}
\label{not-4.2}
\upshape
 In the following lines,
we will denote by $S_i$, $i=1,2,3$, a positive constant depending on the problem data,
on  $\bar{S_1}$ and $\bar{S_2}$  in \eqref{bounddatiteta-1}-\eqref{bounddatiteta-2}, on
$M>0$ in \eqref{fixed-point-set}, but \emph{independent} of $\eps>0$, and by
  $S_4(\eps)$ a constant depending on the above quantities and on $\eps>0$ as well.
Furthermore,  with the symbols $\pi_i (A)$, $\pi_{i,j}(A), \ldots,$ we will denote the projection of
a set $A$ on its $i$-, or $(i,j)$-component, $\ldots$.
\end{notation}

\paragraph{Step $1$:}
As a first step in the construction of $\mathcal{T}$, we fix
$(\widehat\teta,\widehat{\teta}_s,\widehat{\uu},\widehat\chi) \in \mathcal{Y}_\tau$ and prove a well-posedness result
for
 (the Cauchy problem for)  the PDE system
(\ref{eqIa-app}--\ref{to-be-quoted-also}),
with  $(\widehat\teta,\widehat\teta_s,\widehat{\chi})$ in place of
$(\teta,\teta_s,\chi)$, and
\[
\fc(\widehat{\teta}-\widehat{\vartheta}_s) |\Reg(\phi_\eps'(\whuun){\bf n})|  \ \text{ replacing } \  \fc(\teta-\vartheta_s) |\Reg(\phi_\eps'(\uun){\bf n})|.
\]
\begin{lemma}\label{lemmaS1}
 Assume \eqref{assumpt-domain}, Hypotheses (I)--(III), and suppose that
 $\mathbf{f}, \, \mathbf{g}, \, \uu_0$ comply with \eqref{hypo-f}--\eqref{hypo-g}, and
\eqref{cond-uu-zero}.

Then, there exists a constant $S_1>0$ such that
for all $(\widehat\teta,\widehat{\teta}_s,\widehat{\uu},\widehat\chi) \in \mathcal{Y}_\tau$
 there exists a
unique  pair $(\uu,\mmu) \in H^1(0,\tau;{\bf W}) \times  L^\infty
(0,\tau;L^{2+\nu}(\gc;\R^3))  $ fulfilling the initial condition \eqref{iniu} and
\begin{equation}
\label{eqS1}
\begin{aligned}
 &  b(\dotu,\vv)  +a(\uu,\vv)+ \int_{\Omega} \widehat\teta \dive (\vv)\dd x
 +\int_{\Gamma_c} \widehat\chi \uu \cdot \vv \dd x\\
& + \int_\gc \phi_\eps'(\uun){\bf n}\cdot  \vv\dd x
   +
\int_{{{\Gamma_c}}} \fc(\widehat{\teta}-\widehat{\vartheta}_s)  \mmu \cdot {\vv}\dd x =
\pairing{}{\bsW}{\mathbf{F}}{\vv}
 \quad \text{for all } \vv\in \bsW \ \hbox{ a.e. in }\, (0,T)\,,
 \\
 &  \mmu = |\Reg(\phi_\eps'(\whuun){\bf n})|
  {\zz}  \ \text{ with } \ \zz\in {\bf d}(\dotu) \hbox { a.e. in }\,
{\Gamma_c}\times (0,T)\,,
 \end{aligned}
 \end{equation}
and the estimate
\begin{equation}\label{boundS1}
\|\uu\|_{H^1(0,\tau;{\bf W})}+ \| \zz \|_{L^\infty (\gc \times (0,\tau))} + \| \mmu \|_{L^\infty
(0,\tau;L^{2+\nu}(\gc;\R^3)) } \leq S_1.
\end{equation}
\end{lemma}
\begin{proof}
Observe  that \eqref{eqS1} has the very same structure of the momentum equation
in the isothermal model for adhesive contact with friction tackled in \cite{bbr5}, for which the existence
of solutions was proved by passing to the limit in a time-discretization scheme.
The arguments from \cite{bbr5} can be easily adapted to the present setting, also  taking into account
that the term
\begin{equation}
\label{new-operator}
\widehat{\mathcal{R}} :  = \fc(\widehat{\teta}-\widehat{\vartheta}_s) |\Reg(\phi_\eps'(\whuun){\bf n})| \text{ is in  }
 L^2(0,\tau; L^{4/3}(\Gamma_c)).
\end{equation}
Indeed,
\eqref{new-operator}
can be checked by observing that
$\fc(\widehat\teta-\widehat\teta_s) \in L^2(0,\tau;L^{4/({1+2\delta})}(\Gamma_c))$
by trace theorems and the Lipschitz continuity of $\fc$,
 and combining this with
\eqref{hyp-r-nuova}.  It follows from \eqref{new-operator}
that  the term $\int_{{{\Gamma_c}}} \widehat{\mathcal{R}} {\zz}\cdot {\vv} \dd x $ in
 the momentum equation
\eqref{eqS1}, with  a selection
 $\zz \in L^\infty (\gc \times (0,\tau)) $ from ${\bf d}(\dotu)$
  as in the
second of   \eqref{eqS1},
is well-defined for every  $\vv \in \mathbf{W}$.  Observe that the $L^\infty
(0,\tau;L^{2+\nu}(\gc;\R^3))$-regularity of $\mmu$ derives from \eqref{hyp-r-1}.

To prove uniqueness, we proceed as in
\cite[Sec.\ 5]{bbr5}: given two solution pairs   $(\uu_1,\mmu_1)$ and $(\uu_2,\mmu_2)$ to the Cauchy problem for \eqref{eqS1},
we test the equation
fulfilled by  $\tilde{\uu}:= \uu_1-\uu_2$
 and $\tilde{\mmu}: = \mmu_1 -\mmu_2$, $\mmu_i =  |\Reg(\phi_\eps'(\whuun){\bf n})|
  {\zz_i}  $ for some
  $ \zz_i\in {\bf d}( \partial_t \uu_i)$, $i=1,2$,
 with $\vv:= \partial_t \tilde{\uu}$ and integrate in time. With
    straightforward calculations, setting $\tilde{\zz}= \zz_1 -\zz_2$ and using the place-holder
    $\widehat{\mathcal{R}} $ from
     \eqref{new-operator},
 we obtain
\[
\begin{aligned}
&
\int_0^t  b(\partial_t \tilde{\uu},\partial_t \tilde{\uu}) \dd s
+\frac12 a(\tilde{\uu}(t), \tilde{\uu}(t) )+\int_0^t
\int_{{{\Gamma_c}}} \widehat{\mathcal{R}}\,  \tilde{\zz}\cdot {\partial_t \tilde{\uu}}\dd x\dd s
\\ &
\leq  \left|\int_0^t\int_{\Gamma_c} \widehat\chi \tilde{\uu} \cdot \partial_t \tilde{\uu} \dd x \dd s\right|
 + \left|\int_0^t\int_\gc (\phi_\eps'(\uuni 1){-}\phi_\eps'(\uuni 2))
 {\bf n}\cdot\partial_t \tilde{\uu}   \dd x\dd s \right|.
\end{aligned}
\]
Now, the third term on the left-hand side is positive by monotonicity of $\mathbf{d}$  and positivity of
$\fc$ (cf.\ \eqref{hyp-fc}), whereas the second integral
on the right-hand side can be estimated  relying on the Lipschitz continuity of $\phi_\eps'$.
All in all, the desired contraction estimate
for $\tilde\uu$
 ensues from
applying the Gronwall lemma, with calculations analogous to those in \cite[Sec.\ 5]{bbr5}. Clearly,
from this
 we also deduce that there
exists a unique $\mmu$ complying with \eqref{eqS1}.

Finally, to prove estimate \eqref{boundS1}
we test \eqref{eqS1} by  $\vv=\partial_t\mathbf{u}$ and integrate on
$(0,t)$ with $t  \in (0,\tau]$.
 We obtain
 the following estimate
 \begin{equation}
 \label{est-local-u}
\begin{aligned}
& C_b\int_0^t \|\partial_t\mathbf{u}\|^2_{\bsW} \dd s +\frac {C_a} 2\|\uu(t)\|^2_\bsW+\int_0^t\int_{{{\Gamma_c}}}\fc(\widehat\teta-\widehat
\teta_s)
{\mmu}\cdot {\partial_t\mathbf{u}}\dd x \dd s +\int_{\Gamma_c}\phi_\eps(\uun(t))\dd x\\
&\leq c\left(\| \uu_0 \|_{\bsW}^2+  \int_\gc \phi_\eps(\uun(0)) \dd x  +\int_0^t( \|\widehat
\teta\|_{H}+\|\bsF\|_{\bsW'})\|\partial_t\mathbf{u}\|_\bsW \dd s
+\int_0^t \|  \widehat\chi \|_{\Hc} \|\uu\|_\bsW\|\partial_t\mathbf{u}\|_{\bsW}\dd s\right)\\
&\leq c\left(1+M^2+\int_0^t \|  \widehat\chi \|_{\Hc}^2 \|\uu\|^2_\bsW  \dd s\right)+\frac {C_b} 2\|\partial_t\mathbf{u}\|^2_\bsW,
\end{aligned}
\end{equation}
 where for the left-hand side we have used the chain-rule identities
 \begin{equation}
 \label{chain-rules-local}
 \begin{aligned}
 &
 \int_0^t a(\uu(s),\partial_t \uu(s))\dd s =\frac12 a(\uu(t),\uu(t))-\frac12 a(\uu_0,\uu_0)
 \\
 &
 \int_0^t \int_\gc \phi_\eps'(\uun){\bf n}\cdot  \partial_t \uu \dd x \dd s=
  \int_\gc \phi_\eps(\uun(t)) \dd x - \int_\gc \phi_\eps(\uun(0)) \dd x,
  \end{aligned}
 \end{equation}
 (cf.\ \cite[Lemma~3.3]{brezis73},
\cite[Lemma~4.1]{colli92})
the coercivity properties \eqref{korn_a} and \eqref{korn_b} of the forms $a$ and $b$, and ultimately
 that  $\int_\gc \phi_\eps(\uun(0)) \dd x \leq \bvarphi(\uu_0)<\infty$ by \eqref{cond-uu-zero}.
By the monotonicity of $\mathbf{d}$ and the positivity  of  $\fc$ (cf.\ \eqref{hyp-fc}), we  also infer that
\begin{equation}
\label{to-be-quoted-later}
\int_0^t
 \int_{{{\Gamma_c}}}\fc(\widehat\teta-\widehat\teta_s)   {\mmu}\cdot {\partial_t\mathbf{u}}\dd x \dd s \geq 0.
 \end{equation}
For the estimates on the right-hand side of \eqref{est-local-u}, we rely on
Young's inequality and  trace theorems. Therefore,  \eqref{boundS1} ensues from
\eqref{est-local-u} by the
 Gronwall lemma. Clearly, the estimate for $\zz$ follows from the fact that $|\zz| \leq 1$ a.e.\
 on $\gc \times (0,\tau)$,  whence the estimate for $\mmu$, in view of \eqref{hyp-r-1}.
 \medskip
\end{proof}

\noindent
Thanks to Lemma
 \ref{lemmaS1} we may define an operator
 \begin{equation}
 \label{ope-1}
 \mathcal{T}_1: \mathcal{Y}_\tau \to \mathcal{U}_\tau:= \{\uu \in H^1 (0,\tau;\bsW)\, : \ \|\uu\|_{H^1(0,\tau;{\bf W})}\leq S_1 \}
 \end{equation}
 mapping every quadruple $(\widehat\teta,\widehat{\teta}_s,\widehat{\uu},\widehat\chi) \in \mathcal{Y}_\tau$
 into the unique solution $\uu$ (together with $\mmu  \in   |\Reg(\phi_\eps'(\whuun){\bf n})| {\bf d}(\dotu)$) of the Cauchy problem  for \eqref{eqS1}.

\paragraph{Step $2$:}
As a second step, we solve (the Cauchy problem for)
 \eqref{eqIIa}--\eqref{inclvincolo}, with   $\widehat\teta_s \in \pi_2(\mathcal{Y}_\tau) $ and $\uu$ from Lemma
 \ref{lemmaS1} on the right-hand side of
   \eqref{eqIIa}.
   \begin{lemma}
   \label{lemmaS1/2}
    Assume   \eqref{assumpt-domain},  Hypotheses (IV)--(V), and suppose that
 $ \chi_0$ complies with
\eqref{cond-chi-zero}.

Then, there exists a constant $S_2>0$ such that
for all $(\widehat{\teta}_s,\uu) \in \pi_2(\mathcal{Y}_\tau) \times \mathcal{U}_\tau$
 there exists a
unique pair $(\chi,\xi)  \in (L^2 (0,\tau; H^2(\gc)) \cap L^\infty (0,\tau; \Vc) \cap  H^1(0,\tau;\hc))
\times L^2 (0,\tau; \hc)$ fulfilling the initial condition \eqref{inichi}, the relations
\begin{equation}
\label{eqS1/2}
\begin{aligned}
 &\partial_t{\chi}+A\chi+\xi+\sigma'(\chi) =-\lambda'(\chi)
\widehat{\vartheta}_s- \frac12 |\uu|^2 \quad\hbox{a.e. in } {\Gamma_c}\times
(0,T)\,,\\
&\xi\in \beta(\chi)\hbox { a.e. in }\, {\Gamma_c} \times (0,T),
 \end{aligned}
 \end{equation}
and the estimate
\begin{equation}\label{boundS1/2}
\|\chi\|_{L^2 (0,\tau; H^2(\gc)) \cap L^\infty (0,\tau; \Vc) \cap  H^1(0,\tau;\hc)} + \| \xi\|_{L^2(0,\tau;\hc)} \leq S_2.
\end{equation}
   \end{lemma}
   \begin{proof}
   The well-posedness of the Cauchy problem for \eqref{eqS1/2}
 follows from standard results in the theory of
parabolic equations with maximal monotone operators,
 after observing that, by Sobolev embeddings and trace theorems,  $\widehat{\teta}_s$ and
  $1/2 |\uu|^2$  are respectively  estimated in
   $L^2 (0,\tau;L^{2/\delta}(\Gamma_c))$ (with $\delta \in (0,1)$ as in \eqref{hyp-r-nuova}),
   and in
   $L^2(0,\tau;\Hc)$.

   In order to obtain estimate \eqref{boundS1/2}, we test
   \eqref{eqS1/2} by $\partial_t\chi$ and integrate in time.
   Exploiting the chain rule for $\widehat{\beta}$  from \cite[Lemma~3.3]{brezis73} 
    \begin{equation}
 \label{chain-rules-local-bis}
 \begin{aligned}
 &
 \int_0^t
\int_{\Gamma_c} \xi \partial_t\chi  \dd x \dd s=
\int_{\Gamma_c}\widehat{\beta}(\chi(t)) \dd x-
\int_{\Gamma_c}\widehat{\beta} (\chi_0) \dd x \,,
\end{aligned}
 \end{equation}
   we obtain
   the estimate
   \begin{equation}
   \label{est-chi-s12}
   \begin{aligned}
   &
   \int_0^t \int_\gc|\partial_t \chi|^2 \dd x \dd s
   + \frac12 \int_\gc |\nabla \chi(t)|^2 \dd x +
   \int_{\Gamma_c}\widehat{\beta}(\chi(t)) \dd x
   \\
   &
   \leq \frac12 \int_\gc |\nabla \chi_0|^2 \dd x
   +\int_\gc \widehat{\beta}(\chi_0) \dd x
   \\ & \qquad
    + \left| \int_0^t\int_\gc \sigma'(\chi)\partial_t \chi \dd x  \dd s \right|
   + \left| \int_0^t \int_\gc \lambda'(\chi) \widehat{\teta}_s \partial_t \chi \dd x \dd s \right|
   + \left| \int_0^t \int_\gc |\uu|^2 \partial_t \chi \dd x \dd s \right|
   \\ & \leq
   C(\|\chi_0 \|_{\Vc}^2 +1) +I_1 +I_2 +I_3.
   \end{aligned}
   \end{equation}
   Now, the last inequality ensues from \eqref{cond-chi-zero}, and we estimate the integral terms $I_i$,
   $i=1,2,3$, as follows:
   \begin{equation}
   \label{useful-estimates}
   \begin{aligned}
   &
   \begin{aligned}
I_1 & \leq \frac14 \int_0^t \int_\gc |\partial_t \chi|^2 \dd x \dd s
+  C \int_0^t \int_\gc (|\chi|^2+1) \dd x
\dd s  + C'
\\ &
\leq \frac14 \int_0^t \int_\gc |\partial_t \chi|^2\dd x \dd s +C  \int_0^t \| \partial_t \chi\|_{L^2(0,s;\hc)}^2 \dd s + C \|
\chi_0\|_{\hc}^2 + C'
\end{aligned}
\\
&
\begin{aligned}
I_2  & \leq  c \int_0^t (\| \chi \|_{L^{2/(1-\delta)}(\gc)}+1) \| \widehat{\teta}_s \|_{L^{2/\delta}(\gc)}
\| \partial_t \chi \|_{\Hc}  \dd s  \\ & \leq \frac14 \int_0^t \| \partial_t \chi \|_{\Hc}^2\dd s +
C \int_0^t  \| \widehat{\teta}_s \|_{L^{2/\delta}(\gc)}^2 (\| \chi \|_{\Vc}^2+1) \dd s,
\end{aligned}
\\
&
I_3 \leq \frac18 \int_0^t \| \partial_t \chi \|_{\Hc}^2\dd s  + c \int_0^t \| \uu \|_{\bfW}^4 \dd s,
   \end{aligned}
   \end{equation}
   where for $I_1$ we have used that $\sigma'$ is Lipschitz continuous
   by virtue of \eqref{hyp-sig},
   whereas the estimate for $I_2$ follows
   from the Lipschitz continuity of $\lambda'$ and from  observing that
   $H^1(\gc) $ embeds continuously in $L^{2/(1-\delta)}(\gc)$ for any $\delta \in (0,1)$, and for $I_3$ we have
   exploited the trace result \eqref{cont-embe}.
As for the left-hand side of \eqref{est-chi-s12}, it remains to observe that,
by convexity,
\begin{equation}
\label{plugged-later}
\exists\, c', \, C'>0 \ \ \text{such that for every $t \in (0,\tau]$}\, :  \ \
\int_{\Gamma_c}\widehat{\beta}(\chi(t)) \dd x \geq -c' \int_\gc |\chi(t)|\dd x-C'.
\end{equation}
Plugging \eqref{useful-estimates} and \eqref{plugged-later} into \eqref{est-chi-s12} and using the Gronwall Lemma
we conclude an estimate for $\chi$ in $L^\infty (0,\tau;\Vc)\cap H^1(0,\tau;\hc)$.
A comparison argument in \eqref{eqS1/2}
(cf.\ also the forthcoming \emph{Seventh a priori estimate} in
Sec.\ \ref{ss:3.3}),
then yields an
estimate in $L^2(0,\tau;\hc)$ for $A\chi + \xi$,
hence for $A\chi$ and $\xi$
 separately in  $L^2(0,\tau;\hc)$ by monotonicity of $\beta$. Therefore,
 by elliptic regularity results
 we get
 the desired bound for
$\chi$
 in  $L^2 (0,\tau; H^2(\gc))$, which concludes the proof of \eqref{boundS1/2}.
 \medskip
   \end{proof}

\noindent
It follows from  Lemma
 \ref{lemmaS1/2} that we may define an operator
 \begin{equation}
 \label{ope-2}
 \begin{aligned}
 \mathcal{T}_2: \pi_2(\mathcal{Y}_\tau) \times \mathcal{U}_\tau  \to \mathcal{X}_\tau:=
 \{ & \chi \in
    L^2 (0,\tau; H^2(\gc)) \cap L^\infty (0,\tau; H^1(\gc)) \cap  H^1(0,\tau;\hc) \cap X_\tau\, :
    \\
    &
    \ \|\chi\|_{L^2 (0,\tau; H^2(\gc))} + \|\chi\|_{L^\infty (0,\tau; \Vc)} +\|\chi\|_{H^1(0,\tau;\hc)}\leq S_2
   \}
    \end{aligned}
 \end{equation}
 mapping  $(\widehat{\teta}_s,\uu) \in \pi_2(\mathcal{Y}_\tau) \times \mathcal{U}_\tau$
 into the unique solution $\chi$ of the Cauchy problem  for \eqref{eqS1/2} (together with $\xi \in  \beta(\chi)$).

\paragraph{Step $3$:}
Eventually, we
solve the Cauchy problem for the system (\ref{teta-weak-app}, \ref{teta-s-weak-app}) with fixed
  $(\widehat
\teta,\widehat\teta_s) \in \pi_{1,2}(\mathcal{Y}_\tau)$ and
 $(\uu,\chi)$ from Lemmas \ref{lemmaS1} and \ref{lemmaS1/2}, respectively. In particular,
 we set
 \begin{equation}
 \label{not-mathcalF}
 \widehat{\mathcal{F}}:= k(\chi)  (\widehat\teta-\widehat\teta_s)+\fc'(\widehat\teta-\widehat\teta_s) |\Reg (\phi_\eps'(\uun){\bf n})| |\dotut|
 \end{equation}
 and plug it into the boundary integral on the left-hand side of
 \eqref{teta-weak-app} and on the right-hand side of \eqref{teta-s-weak-app}.
 Observe that, due to \eqref{fixed-point-set}, \eqref{boundS1}, \eqref{boundS1/2},
 and to the Lipschitz continuity of $\fc$ and $k$,
 there holds
 \begin{equation}
 \label{est-math-cal-F}
 \widehat{\mathcal{F}} \in L^2 (0,\tau;L^{4/3+s}(\gc)) \quad \text{for some } s=s(\delta)>0.
 \end{equation}
We mention in advance that, relying the very fact that the boundary term
$\widehat{\mathcal{F}}$
 in \eqref{eqapptheta1} below does not depend on the unknown $\teta$, we will be able to prove
 uniqueness of solutions for (the Cauchy problem for) \eqref{eqapptheta1}.

\begin{lemma}\label{lemmaS2}
Assume \eqref{assumpt-domain},  Hypotheses (I),  (II), (III), and (V), suppose that $h$ complies with \eqref{hypo-h}, and
let $(\teta_0^\eps,\teta_s^{0,\eps})$ fulfill
\eqref{dati-teta-better}.

Then, there exist positive  constants $S_3$ and $S_4 (\eps)$ such that
for all $(\widehat{\teta},\widehat{\teta}_s,\uu,\chi) \in \pi_{1,2}(\mathcal{Y}_\tau) \times \mathcal{U}_\tau
\times  \mathcal{X}_\tau$
 there exists a
unique couple of functions $(
\teta,\teta_s)$ complying with \eqref{reg-teta-app}--\eqref{reg-log-teta-s-app}, fulfilling the
initial conditions \eqref{iniw-better}--\eqref{iniz-better},  the equations a.e.\ in $(0,T)$
\begin{align}
&
\label{eqapptheta1}
\begin{aligned}
 &  \pairing{}{\V}{\partial_t\applog(\teta)}{ v}
   -\int_{\Omega} \dive(\partial_t\mathbf{u}) \, v  \dd x
+\int_{\Omega}   \nabla \vartheta \, \nabla v  \dd x
+ \int_{\Gamma_c}
\widehat{\mathcal{F}}  v  \dd x
 = \pairing{}{V}{h}{v} \quad
\forall\, v \in V\,,
\end{aligned}
\\
&
\label{eqappthetas1}
\begin{aligned}
   \pairing{}{\Vc}{\partial_t \applog(\teta_s)}{ v }  -\int_{\Gamma_c}
\partial_t \lambda(\chi) \, v   \dd x     +\int_{\Gamma_c} \nabla \vartheta_s  \, \nabla
v \dd x
 = \int_{\Gamma_c} \widehat{\mathcal{F}} v \dd x
  \quad
\forall\, v \in \Vc\,,
\end{aligned}
\end{align}
and the estimates
\begin{align}
&
\label{stimaS2}
 \|\teta\|_{L^2(0,\tau;V) \cap L^\infty(0,\tau;L^1(\Omega))}+\|\teta_s\|_{ L^2(0,\tau;V_{\Gamma_c}) \cap L^\infty(0,\tau;L^1({\Gamma_c}))}\leq S_3,
 \\
 &
 \label{stimaS2bis}
 \|\partial_t \applog(\teta)\|_{L^2(0,\tau;\V')}+ \|\partial_t  \applog(\teta_s)\|_{L^2(0,\tau;\Vc')}\leq S_3,
 \\
 &
 \label{stimaS3bis}
 \|\teta\|_{L^\infty(0,\tau;H)}+\|\teta_s\|_{L^\infty(0,\tau;\hc)} \leq S_4(\eps).
 \end{align}
\end{lemma}
\begin{proof}
Observe that system (\ref{eqapptheta1}, \ref{eqappthetas1}) is decoupled,
hence we will tackle equations \eqref{eqapptheta1} and
\eqref{eqappthetas1}
separately.

Also taking into account
\eqref{est-math-cal-F}, the well-posedness for  the Cauchy problem for the
doubly nonlinear equation \eqref{eqapptheta1} follows from standard results, cf.\
\cite[Thm.\ 1]{dibenetto-showalter} (see also \cite[Lemma 3.5]{bbr3}):
 in particular, uniqueness
for \eqref{eqapptheta1} is trivial, since the terms $\dive(\partial_t \uu)$ and $\widehat{\mathcal{F}}$
are fixed.
In order to conclude estimates
\eqref{stimaS2}--\eqref{stimaS3bis} for $\teta$, we test \eqref{eqapptheta1}
by $\teta$ and integrate on $(0,t)$ with $t \in (0,\tau]$.
Recalling the definition
\eqref{mathcal-i-eps} of $\mathcal{I}_\eps$,
we exploit the formal identity (cf.\ Remark \ref{rem:formal})
\begin{equation}
\label{formal-identity-1}
\pairing{}{\V}{\partial_t \applog(\teta)}{\teta}= \int_\Omega \applog'(\teta)\partial_t \teta \teta \dd x
=\frac{\dd}{\dd t} \int_\Omega \mathcal{I}_\eps(\teta) \dd x
\end{equation}
and thus infer
\begin{equation}
\label{to-stimaS2}
\begin{aligned}
&
\frac\eps 2 \int_\Omega |\teta(t)|^2\dd x +C_1 \int_\Omega |\teta(t)|\dd x -C_2
+\int_0^t \int_\Omega |\nabla\teta|^2 \dd x \dd s
\\ &
\leq \int_\Omega \mathcal{I}_\eps(\teta(t)) \dd x
+\int_0^t \int_\Omega |\nabla\teta|^2 \dd x \dd s
\\ &
\leq \int_\Omega \mathcal{I}_\eps(\teta_0^\eps) \dd x
+ \int_0^t \int_\Omega \dive(\partial_t \uu) (\teta-m(\teta))\dd x \dd s
- \int_0^t \int_{\gc}\widehat{\mathcal{F}}(\teta-m(\teta))\dd x \dd s
\\ & \quad + \int_0^t \pairing{}{\V}{h}{\teta-m(\teta)}\dd s
+ \int_0^t \int_\Omega \dive(\partial_t \uu)  m(\teta)\dd x \dd s\\
& \quad
- \int_0^t \int_{\gc}\widehat{\mathcal{F}}m(\teta)\dd x \dd s
+\int_0^t \int_\Omega h m(\teta) \dd x \dd s
\\ &
\leq \bar{S_1}(1+\|\teta_0\|_{L^1(\Omega)})+ I_1+I_2+I_3+I_4+I_5+I_6,
\end{aligned}
\end{equation}
where the first inequality is due to \eqref{ine-imu-2}, and the estimate for $\int_\Omega \mathcal{I}_\eps(\teta_0^\eps) \dd x$  follows from  \eqref{bounddatiteta-1}.
As for the terms $I_i$, $i=1,\ldots,6$, by
the Sobolev embeddings and trace results \eqref{cont-embe}, joint with
Poincar\'{e}'s inequality \eqref{equiv-norm}, we
easily have
\[
\begin{aligned}
&
\begin{aligned}
I_1+I_2 +I_3 & \leq
 \int_0^t \left(  \| \partial_t \uu\|_{\bsW} \| \teta{-}m(\teta)\|_H  + \| \widehat{\mathcal{F}}\|_{L^{4/3}(\gc)}
  \| \teta{-}m(\teta)\|_{L^4(\gc)}
  + \| h\|_{V'}  \| \teta{-}m(\teta)\|_V\right)
   \dd s
   \\
   &
  \leq
  \frac14 \int_0^t\| \nabla \teta \|_{H}^2 \dd s +  \int_0^t \left( \| \partial_t \uu\|_{\bsW}^2 {+}
\| \widehat{\mathcal{F}}\|_{L^{4/3}(\gc)}^2 {+}   \| h\|_{V'}^2  \right) \dd s,
\end{aligned}
\\
&
I_4+I_5 +I_6\leq C  \int_0^t \left( \| \partial_t \uu\|_{\bsW} {+}
\| \widehat{\mathcal{F}}\|_{L^{4/3}(\gc)}{+} \| h \|_{L^2(\Omega)}  \right) \| \teta\|_{L^1(\Omega)} \dd s.
\end{aligned}
\]
We plug the above estimates into \eqref{to-stimaS2} and use \eqref{hypo-h}, estimate \eqref{boundS1},
and \eqref{est-math-cal-F}.  Relying on  the Gronwall lemma, we conclude
estimates \eqref{stimaS2} and \eqref{stimaS3bis} for $\teta$. Estimate \eqref{stimaS2bis} for
$\applog(\teta)$ follows from
 a comparison in
\eqref{eqapptheta1}.

Since the calculations related to the analysis of equation \eqref{eqappthetas1} are completely analogous,
we choose to omit them.
\medskip
\end{proof}

\noindent
 Thanks to Lemma
 \ref{lemmaS2},  we may define an operator
 \begin{equation}
 \label{ope-3}
 \begin{aligned}
  & \mathcal{T}_3: \pi_{1,2}(\mathcal{Y}_\tau) \times \mathcal{U}_\tau  \times \mathcal{X}_\tau \to
  \\ &
  \begin{aligned}
 \mathcal{W}_\tau:=
 \{  (\teta,\teta_s) \in
     & (L^2 (0,\tau; V) \cap L^\infty (0,\tau;H)) \times (L^2 (0,\tau; \Vc) \cap L^\infty (0,\tau;\hc)) :
    \\
    &
    \ \|\teta\|_{L^2 (0,\tau; V) \cap L^\infty (0,\tau;L^1(\Omega))} + \|\teta_s\|_{L^2 (0,\tau; \Vc) \cap L^\infty (0,\tau;L^1(\gc))} \leq S_3,
    \\
    & \ \|\teta\|_{L^\infty (0,\tau;H)} + \|\teta_s\|_{L^\infty (0,\tau; \hc)} \leq S_4(\eps)
   \}
   \end{aligned}
    \end{aligned}
 \end{equation}
 mapping $(\widehat\teta,\widehat{\teta}_s,\uu,\chi) \in \pi_{1,2}(\mathcal{Y}_\tau) \times \mathcal{U}_\tau  \times  \mathcal{X}_\tau$
 into the unique solution $(\teta,\teta_s)$ of the Cauchy problem  for system (\ref{eqapptheta1}, \ref{eqappthetas1}).

We are now in the position to prove the existence of local-in-time solutions to
Problem $(P_\eps)$, defined on some interval $[0,\widehat T]$ with $0<\widehat T \leq T$.
Note that $\widehat T$ in fact does not depend on the parameter $\eps>0$.
\begin{proposition}[Local existence for Problem $(P_\eps)$]
\label{prop:loc-exist-eps}
Assume  \eqref{assumpt-domain},
 Hypotheses (I)--(V), and conditions
\eqref{hypo-h}--\eqref{hypo-g} on the data $h$, $\mathbf{f}$,  $\mathbf{g}$,
\eqref{cond-uu-zero}--\eqref{cond-chi-zero} on
$\uu_0,\,\chi_0$, and \eqref{dati-teta-better} on
 $\teta_0^\eps,$
$\teta_s^{0,\eps}$.

Then,  there exists $\widehat{T}\in (0,T]$
 such that for every $\eps>0$  Problem   $(P_\eps)$
admits a solution $(\teta,\teta_s,\uu,\chi,\mmu,\xi)$ on the interval $(0,\widehat{T})$.
\end{proposition}
\begin{proof}
Let the operator
$\mathcal{T}: \mathcal{Y}_\tau \to \mathcal{W}_\tau \times \mathcal{U}_\tau \times \mathcal{X}_\tau$
 be defined by
\begin{equation}
\label{def-T-ope} \mathcal{T}(\widehat{\teta}, \widehat{\teta}_s,
\widehat{\uu},\widehat{\chi} ) := (\teta, \teta_s,\uu,\chi) \ \ \text{with } \ \
\begin{cases}
\uu:= \mathcal{T}_1(\widehat{\teta}, \widehat{\teta}_s,
\widehat{\uu},\widehat{\chi} ),
\\
\chi:= \mathcal{T}_2 (\widehat{\teta}_s,\uu),
\\
(\teta,\teta_s):= \mathcal{T}_3 ( \widehat\teta,\widehat{\teta}_s, \uu,\chi).
\end{cases}
\end{equation}
In what follows,  we will show that
 there exists $\widehat{T} \in (0,T]$ such that for every $\eps >0$
\begin{gather}
\label{itself} \text{
 $\mathcal{T}$ maps
 $\mathcal{Y}_{\widehat{T}}$ into itself,}
 \\
 \label{compact-conti}
 \begin{gathered}
\mathcal{T} : \mathcal{Y}_{\widehat{T}} \to
\mathcal{Y}_{\widehat{T}}  \ \ \text{ is compact and
 continuous w.r.t. the topology of } \\ \text{ $L^2(0,\tau;H^{1-\delta}(\Omega))\times L^2(0,\tau;H^{1-\delta}(\Gamma_c))
\times L^2(0,\tau;H^{1-\delta}(\Omega;\R^3))
\times L^2(0,\tau;\Hc)$.}
\end{gathered}
\end{gather}

\noindent \textbf{Ad \eqref{itself}.}
In order to show \eqref{itself},
let $(\widehat\teta,\widehat{\teta}_s,\widehat\uu, \widehat\chi)\in \mathcal{Y}_\tau$
be fixed, and let $(\teta,\teta_s,\uu,\chi):=\mathcal{T}(\widehat\teta,\widehat{\teta}_s,\widehat\uu\,\widehat\chi)$.
 We use the interpolation inequality
 \begin{equation}
\label{interpolation-brezis}
 \|\teta(t)\|_{H^{1-\delta}(\Omega)}\leq c \|\teta(t)\|^{1-\delta}_{H^1(\Omega)}\|\teta(t)\|^\delta_{L^2(\Omega)} \qquad \foraa\, t \in (0,\tau)
\end{equation}
(cf.\ e.g.\ \cite[Cor. 3.2]{brezis-mironescu}).
Now, a further interpolation
between the spaces $L^2 (0,\tau;V)$ and $L^\infty(0,\tau;L^1(\Omega))$ and estimate \eqref{stimaS2} also yield the bound
$\| \teta\|_{L^{10/3}(0,\tau;L^2(\Omega))} \leq \bar{C} S_3$
for some interpolation constant $\bar{C}$.
Integrating \eqref{interpolation-brezis} in time and using
H\"older's inequality we therefore have
\begin{equation}
\label{inte-teta}
\begin{aligned}
\|\teta\|_{L^2(0,t;H^{1-\delta}(\Omega))}^2 & \leq c
 \int_0^t  \|\teta(s)\|^{2(1-\delta)}_{H^1 (\Omega)} \|\teta(s) \|^{2\delta}_{L^2 (\Omega)}\dd s
 \\ & \leq c  \| \teta \|_{L^{2}(0,\tau;H^1(\Omega))}^{2(1-\delta)} t^{(2 \delta)/5}
   \| \teta \|_{L^{10/3}(0,\tau;L^2(\Omega))}^{2\delta}
\leq  C   S_3^{2} t^{(2 \delta)/5}\,.
\end{aligned}
\end{equation}
 We use \eqref{interpolation-brezis}
for $\teta_s$ and $\uu$ to perform calculations analogous to \eqref{inte-teta}, whereas for $\chi$ we trivially
have
$\|\chi\|_{L^2(0,t;\hc)}^2 \leq t \|\chi\|_{L^\infty(0,t;\hc)}^2 \leq t S_2^2$. Combining all of these estimates,
we conclude that there exists  a sufficiently small $\widehat{T}>0$ for which \eqref{itself} holds.

\noindent
\textbf{Ad \eqref{compact-conti}: compactness.}
Exploiting estimates \eqref{boundS1}, \eqref{boundS1/2}, and
the compactness results~\cite[Thm.~4, Cor.~5]{Simon87}, it is immediate to check compactness
of the operator $\mathcal{T}$ as far as the $(\uu,\chi)$-component is concerned.
As for the  $(\teta,\teta_s)$-component, from  \eqref{stimaS2}--\eqref{stimaS3bis} and the Lipschitz continuity of
$\applog$
we deduce an estimate
(with a constant depending on $\eps$)
for $\applog(\teta)$ in $L^2(0,\widehat{T};V)\cap L^\infty (0,\widehat{T}; H) \cap  H^1(0,\widehat{T};V')$
(for  $\applog(\teta_s)$ in
 $L^2(0,\widehat{T};\Vc) \cap L^\infty (0,\widehat{T}; \Hc) \cap H^1(0,\widehat T;\Vc')$, resp.), whence compactness
for $\applog(\teta)$ in $L^2(0,\widehat{T} ;H^{1-\delta}(\Omega))$ (for $\applog(\teta_s)$ in $L^2(0,\widehat{T} ;H^{1-\delta}(\gc))$, resp.), hence for $\teta= \applog^{-1}(\applog(\teta))$ in the same space
(and analogously  for $\teta_s$). Observe that the latter argument relies on
the bi-Lipschitz continuity \eqref{bi-Lip} of $\applog$.

\noindent
\textbf{Ad \eqref{compact-conti}: continuity.}
In order to prove that $\mathcal T $ \eqref{def-T-ope} is continuous, we will  check
that
the operators $\mathcal{T}_i$, $i=1, 2, 3$ defined by \eqref{ope-1}, \eqref{ope-2},
 and \eqref{ope-3} are continuous w.r.t.\ to suitable topologies.

First of all,
 we fix a sequence $\{(\widehat{\teta}_n,
\widehat{\teta}_{s,n}, \widehat{\uu}_n,\widehat{\chi}_n ) \}_n \subset
\mathcal{Y}_{\widehat{T}}$  converging to a
$(\widehat{\teta}_\infty,  \widehat{\teta}_{s,\infty},\widehat{\uu}_\infty,
\widehat{\chi}_\infty )  \in  \mathcal{Y}_{\widehat{T}}$, with
\begin{equation}
\label{co:1}
\begin{array}{lllll}
&  \widehat{\teta}_n \to \widehat{\teta}_\infty
 & \text{in $L^2 (0,\widehat{T}; H^{1-\delta}(\Omega))$}, \qquad &  \widehat{\teta}_{s,n}
 \to \widehat{\teta}_{s,\infty}
 & \text{in $L^2 (0,\widehat{T}; H^{1-\delta}(\gc))$},
\\
&  \widehat{\uu}_{n} \to \widehat{\uu}_{\infty}  &
\text{in $L^{2} (0,\widehat{T}; H^{1-\delta}(\Omega;\R^3))$}, \qquad &  \widehat{\chi}_{n}
 \to \widehat{\chi}_{\infty}
 & \text{in $L^2 (0,\widehat{T}; \hc)$}
\end{array}
\end{equation}
as $n \to \infty$.
We let $\uu_n:= \mathcal{T}_1 (\widehat{\teta}_n,
\widehat{\teta}_{s,n}, \widehat{\uu}_n,\widehat{\chi}_n )$,  and denote by $(
\mmu_n)_n  $ the associated  sequence such that $(\uu_n,\mmu_n) $ fulfill \eqref{eqS1}.
Due
to estimate \eqref{boundS1}, there exist a (not relabeled) subsequence and a pair $(\uu_\infty,\zz_\infty)$ such that
as $n\to \infty$
\begin{equation}
\label{co:1-conseq}
\uu_n \weakto \uu_\infty \quad \text{in $H^1(0,\widehat{T};\bsW)$,} \qquad \mmu_n \weaksto \mmu_\infty \quad
 \text{in $L^\infty
(0,\widehat{T};L^{2+\nu}(\gc;\R^3)) $.}
\end{equation}
Hence, by  well-known compactness results,
 $(\uu_n)_n$ strongly converges to $\uu$ in $\mathrm{C}^0 ([0,\widehat T]; H^{1-\delta} (\Omega; \R^3))$ for all
$\delta \in (0,1]$.
Now,
combining convergences \eqref{co:1} and \eqref{co:1-conseq} and arguing in the very same way as in the
 proof of Thm.\ \ref{mainth:1} (cf.\ the forthcoming Section \ref{s:3.3}), we manage to pass to the limit
 as $n \to \infty$ in \eqref{eqS1}, concluding that the pair $(\uu_\infty,\mmu_\infty)$
 fulfill  equation \eqref{eqS1} with $(\widehat{\teta}_\infty,  \widehat{\teta}_{s,\infty},\widehat{\uu}_\infty
\widehat{\chi}_\infty) $.  Therefore, we have that
\begin{equation}
\label{ope-t1}
\uu_\infty=\mathcal{T}_1
(\widehat{\teta}_\infty,  \widehat{\teta}_{s,\infty},\widehat{\uu}_\infty,
\widehat{\chi}_\infty), \text{ and  convergences \eqref{co:1-conseq} hold for the \emph{whole}  $\{(\uu_n,\mmu_n)\}_n$,}
\end{equation}
the latter fact   by uniqueness of the limit.

Secondly, we
consider the sequence $\chi_n: =\mathcal{T}_2
(\widehat{\teta}_{s,n}, \uu_n)$ with $(\uu_n)_n $ from the previous step, and let
$(\xi_n)_n$ be the associated sequence of selections in $\beta(\chi_n)$,
such that $(\chi_n,\xi_n)$ fulfill \eqref{eqS1/2}.
Thanks to estimate~\eqref{boundS1/2}, we have that  $( \chi_n,\xi_n )_n$ is bounded in
$(L^2(0,\widehat{T};H^2(\gc)) \cap L^\infty(0,\widehat{T};\Vc) \cap H^1(0,\widehat{T};\hc))
\times L^2(0,\widehat{T};\hc) $. Therefore, there exists $(\chi_\infty,\xi_\infty)$ such that, up to a subsequence,
as $n \to \infty$
\begin{equation}
\label{co:2}
\begin{array}{lll}
&
\chi_n\weaksto\chi_\infty & \text{in $L^2(0,\widehat{T};H^2(\gc)) \cap L^\infty(0,\widehat{T};\Vc) \cap H^1(0,\widehat{T};\hc)$,}
\\
& \xi_n \weakto \xi_\infty & \text{in  $L^2(0,\widehat{T};\hc)$,}
\end{array}
\end{equation}
and $(\chi_n)_n$ strongly converges to $\chi_\infty$  in  $L^2(0,T;H^{2-\rho}(\Gamma_c)) \cap \mathrm{C}^0([0,T];H^{1-\delta}(\Gamma_c))$
for all  $\rho \in (0,2]$ and  $\delta \in (0,1]$
by \cite[Thm.\ 4, Cor.\ 5]{Simon87}.
Relying on convergence \eqref{co:1} for $(\widehat{\teta}_{s,n})_n$,
\eqref{co:1-conseq} for $(\uu_n)_n$, and \eqref{co:2}, and arguing as in the passage
to the limit developed in Sec.\ \ref{s:3.3}, it can be shown that the functions $(\chi,\xi)$ solve
\eqref{eqS1/2} with $(\widehat{\teta}_{s,\infty},\uu_\infty)$, i.e.
\begin{equation}
\label{ope-t2}
\chi_\infty=\mathcal{T}_2
(\widehat{\teta}_{s,\infty},\uu_\infty), \text{ and  convergences \eqref{co:2} hold for the \emph{whole}  $(\chi_n,\xi_n)_n$.}
\end{equation}

Thirdly, we let $(\teta_n,\teta_{s,n}):=\mathcal{T}_3 (\widehat{\teta}_{n}, \widehat{\teta}_{s,n},\uu_n,\chi_n)$
with $(\uu_n)_n $ and $(\chi_n)_n$ from the previous steps.
Estimates \eqref{stimaS2}--\eqref{stimaS3bis} imply that there exist
$(\teta_\infty,\teta_{s,\infty})$ such that, along a (not relabeled) subsequence, the following weak convergences
hold as $n \to \infty$
\begin{equation}
\label{cor:3}
\begin{array}{lll}
& \teta_n \weaksto \teta_\infty & \text{in $L^2(0,\widehat T;V)\cap L^\infty (0,\widehat T;H)$,}
\\
& \teta_{s,n} \weaksto \teta_{s,\infty} & \text{in $L^2(0,\widehat T;\Vc)\cap L^\infty (0,\widehat T;\hc)$.}
\end{array}
\end{equation}
Furthermore, taking into account the that $(\applog(\teta_n))_n $
 ($(\applog(\teta_{s,n}))_n$, respectively),
 is bounded in $L^2(0,\widehat T;V)\cap L^\infty (0,\widehat{T};H) \cap  H^1 (0,\widehat T;V')$
 (in $L^2(0,\widehat T;\Vc)  \cap L^\infty (0,\widehat{T};\Hc)  \cap H^1 (0,\widehat T;\Vc')$, resp.)
 and relying on \cite[Thm.\ 4, Cor.\ 5]{Simon87}, we find that
 $\applog(\teta_n) \to \applog(\teta)$ in $L^2(0,\widehat{T}; H^{1-\delta}(\Omega)) \cap
 \mathrm{C}^0 ([0,\widehat T]; H)$, and analogously for $(\applog(\teta_{s,n}))_n$. Therefore,
  thanks to the bi-Lipschitz continuity of $\applog$ we conclude that
  \begin{equation}
  \label{cor:4}
\begin{array}{lll}
& \teta_n \to \teta_{\infty} & \text{in $L^2(0,\widehat T;H^{1-\delta}(\Omega))\cap \mathrm{C}^0 ([0,\widehat T];H)$,}
\\
& \teta_{s,n}  \to  \teta_{s,\infty} & \text{in $L^2(0,\widehat T;H^{1-\delta}(\gc))\cap \mathrm{C}^0 ([0,\widehat T];\hc)$.}
\end{array}
  \end{equation}
  Convergences \eqref{co:1} for
  $(\widehat{\teta}_{n}, \widehat{\teta}_{s,n})_n$,
  \eqref{co:1-conseq} for $(\uu_n)_n$, \eqref{co:2} for $(\chi_n)_n$, and \eqref{cor:3}--\eqref{cor:4},
  combined with the arguments of Sec.\ \ref{s:3.3},
  allow us to pass to the limit as $n \to \infty$ in system (\ref{eqapptheta1}, \ref{eqappthetas1}). Therefore,
  we conclude that
\begin{equation}
\label{ope-t3}
\begin{gathered}
(\teta_{\infty},\teta_{s,\infty})=\mathcal{T}_3
(\widehat{\teta}_\infty,\widehat{\teta}_{s,\infty},\uu_\infty,\chi_\infty) \text{ and }
\\
\text{convergences \eqref{cor:3}--\eqref{cor:4} hold for the \emph{whole}  $(\teta_n,\teta_{s,n})_n$.}
\end{gathered}
\end{equation}
Ultimately,
the continuity of $\mathcal{T}$ ensues from \eqref{ope-t1}, \eqref{ope-t2}, and \eqref{ope-t3}.
\end{proof}
\begin{remark}
\upshape
\label{rmk:toBV}
A key trick to prove compactness and continuity of the operator $\mathcal{T}$ \eqref{def-T-ope}
in the $(\teta,\teta_s)$-component has been:
\begin{compactenum}
\item[-]
 to prove compactness and continuity in
$(\applog(\teta),\applog(\teta_s))$ (exploiting the estimates on the pair $(\partial_t\applog(\teta),\partial_t\applog(\teta_s))$
deduced from  a comparison in the temperature equations \eqref{teta-weak-app} and \eqref{teta-s-weak-app}),
\item[-]
then to  infer
 compactness and continuity in $(\teta,\teta_s)$, relying on  the fact that $\applog$ is bi-Lipschitz,
 cf.\ Lemma \ref{l:new-lemma2}.
 \end{compactenum}

 Obviously we will not be in the position to use such an argument any longer, when taking the
 limit as $\eps \to 0$.
 Indeed, for such a passage to the limit  we will rely on new, $\BV$-type estimates on $(\teta,\teta_s)$ (cf.\ the
 \emph{Fifth} and \emph{Sixth a priori estimate} in the forthcoming Sec.\ \ref{ss:3.3}).
 Exploiting such bounds and a version of the Aubin-Lions theorem for the case of
 time derivatives as measures (see, e.g., \cite[Chap.\ 7, Cor.\ 7.9]{roub-book}), we will
 conclude the desired compactness for $\teta$ and $\teta_s$.
\end{remark}

\subsection{The approximate problem: global existence}
\label{ss:3.3}
We now prove the following
\begin{theorem}[Global existence for Problem $(P_\eps)$]
\label{th:exist-approx-rev}
Assume \eqref{assumpt-domain},
  Hypotheses (I)--(V),  conditions
\eqref{hypo-h}--\eqref{hypo-g} on the data $h$, $\mathbf{f}$,  $\mathbf{g}$,
\eqref{cond-uu-zero}--\eqref{cond-chi-zero} on  $\uu_0$, and $\chi_0$,
and \eqref{dati-teta-better} on the approximate data $\teta_0^\eps$ and $\teta_s^{0,\eps}$.
Then,
\begin{compactenum}
\item[1.]
 for all $\eps>0$ Problem $(P_\eps)$ admits a solution $(\teta,\teta_s,\uu,\chi,\mmu,\xi)$ on
the \emph{whole} interval $(0,T)$;
\item[2.]
there exists
a constant $C>0$ such that for every $\eps>0$ and for \emph{any} global-in-time solution
 $(\teta,\teta_s,\uu,\chi,\mmu,\xi)$
to Problem $(P_\eps)$ the following estimates hold:
\begin{align}
& \label{est-glob-1}
 \eps^{1/2}\|\teta\|_{L^\infty(0,T;H)}+\eps^{1/2}\|\teta_s\|_{L^\infty(0,T;\Hc)}
 \leq C,
 \\&
 \label{est-glob-2}
 \|\teta \|_{L^2(0,T;V)\cap L^\infty(0,T;L^1(\Omega))} +\|\teta_s
\|_{L^2(0,T;\Vc)\cap L^\infty(0,T;L^1(\gc))} \leq C,
\\ &
\label{est-glob-3}
  \|\uu\|_{H^1
(0,T;\bsW)} + \|\chi\|_{L^\infty(0,T;\Vc)\cap H^1(0,T;\Hc)} \leq C,
\\
&
\label{est-glob-4}
 \| \phi_\eps'(\uun){\bf n}  \|_{L^2
(0,T;\bsY_{{\Gamma_c}}')}+\| \mmu \|_{L^\infty (0,T; L^{2+\nu}(\gc;\R^3))}  \leq C \qquad \text{with $\nu>0$ from \eqref{hyp-r-1}},
 \\
 & \label{est-glob-6}
  \|\partial_t \applog(\vartheta)  \|_{L^2 (0,T;V')}+ \|\partial_t \applog(\vartheta_s) \|_{L^2
(0,T;H^1(\gc)')}
\leq C,
\\
& \label{bv-estimate-teta}
 \|\teta \|_{\mathrm{BV}(0,T;W^{1,q}(\Omega)')} \leq C \qquad \text{for any } q>3,
 \\
 &
\label{bv-estimate-teta-s}
 \|\teta_s \|_{\mathrm{BV}(0,T;W^{1,\sigma}(\gc)')} \leq C \qquad \text{for any } \sigma>2,
 \\
 & \label{est-glob-7}
  \| \chi\|_{L^2(0,T;H^2(\gc))}+\|\xi\|_{L^2(0,T;\hc)}\leq C,
  \\
  & \label{est-glob-8}
  \| \applog(\teta)\|_{L^\infty (0,T;H)}+ \|
\applog(\teta_s)\|_{L^\infty (0,T;\Hc)} \leq C\,.
\end{align}
\end{compactenum}
\end{theorem}
In order to prove Thm.\ \ref{th:exist-approx-rev}, in what follows
we establish  a priori estimates
on the  $(\teta,\teta_s,\uu\,\chi)$-component of  \emph{any} given solution
 to Problem ($P_\eps$), \emph{independent} of the time-interval  on which
 such solution is defined.
 Exploiting these   \emph{global-in-time}   estimates and a standard prolongation argument,
 we will conclude that the local solution to Problem ($P_\eps$) from Proposition \ref{prop:loc-exist-eps}
 extends to a \emph{global-in-time} solution. In this way we will obtain the
 first part of the statement.

In fact,
 it will be clear from the calculations below that such global estimates   hold for a constant \emph{independent} of the parameter $\eps>0$,
 whence \eqref{est-glob-1}--\eqref{est-glob-8}, which
 will provide the starting point for the passage to the limit as $\eps \to 0$
 in Sec.\ \ref{s:3.3}.

We mention in advance, the \emph{First, Fifth, Sixth a priori estimates} below
are only formally derived:
in Remark~\ref{rem:formal} later on we will clarify how they can be made fully rigorous.


\begin{notation}
\upshape
We stress that, from now on the symbols
 $c,\, c', \, C,\,C'$, shall denote a generic constant
possibly depending on the problem data but not on $\eps$, which we let
vary, say, in $(0,1)$.
Since the estimates below are not going to depend on the final time,  we will perform all
the related calculations
 directly on the interval $(0,T)$.
\end{notation}
\paragraph {First a priori estimate.} 
We test \eqref{teta-weak-app} by
$\vartheta$, \eqref{teta-s-weak-app} by $\vartheta_s$, \eqref{eqIa-app} by
$\partial_t\mathbf{u}$, and \eqref{eqIIa} by $\partial_t\chi$, add the resulting
relations, and integrate on $(0,t)$, $t \in (0,T]$.
Recalling \eqref{formal-identity-1}, we \emph{formally} have
\begin{equation}
\label{quasiformal_1-bis}
\begin{aligned}
 \int_0^t \pairing{}{V}{\partial_t \applog(\teta)}{\teta} \dd s  & =
 \int_{\Omega} \mathcal{I}_\eps
( \teta (t) )\dd x - \int_{\Omega} \mathcal{I}_\eps
(\tetazeroe)\dd x  \\ & \geq\frac\eps2\|\teta(t)\|^2_H+ C_1 \|\teta (t)\|_{L^1
(\Omega)} -  \bar{S_1}(1+\| \tetazero \|_{L^1 (\Omega)})
\\ & \geq\frac\eps2\|\teta(t)\|^2_H+ C_1 \|\teta (t)\|_{L^1
(\Omega)}  -C\,,
\end{aligned}
\end{equation}
the first inequality due to  \eqref{ine-imu-2} and \eqref{bounddatiteta-1}, and the  last one to the first of \eqref{cond-teta-zero}. In the same way, we have
\begin{equation}
\label{quasiformal_2-bis}
\int_0^t \pairing{}{\Vc}{\partial_t \applog(\teta_s)}{ \teta_s} \dd s
\geq \frac\eps2\|\teta_s(t)\|^2_{H_{\gc}} + C_1 \|\teta_s (t)\|_{L^1
(\gc)} -C\,.
\end{equation}
We take  into account
the cancellation of some terms,
the chain-rule identities \eqref{chain-rules-local} and \eqref{chain-rules-local-bis},
as well as 
\[
 \int_0^t \int_\gc
 \left( \chi \uu \cdot \partial_t\mathbf{u} +\frac12 |\uu|^2 \partial_t\chi \right) \dd x \dd r=
 \frac12 \int_\gc \chi(t) |\uu(t)|^2 \dd x-\frac12 \int_\gc \chi_0 |\uu_0|^2 \dd x.
 \]
 Therefore, with easy calculations we
arrive at
\begin{equation}
\label{calc-1-app}
\begin{aligned}
 &  \frac\eps2\|\teta(t)\|^2_H+ C_1\| \vartheta(t) \|_{L^1 (\Omega)} +\int_0^t
\int_\Omega |\nabla \teta|^2 \dd x \dd r + \frac\eps2\|\teta_s(t)\|^2_{H_{\gc}}+ C_1\| \vartheta_s(t) \|_{L^1
(\gc)} \\ & \quad +\int_0^t \int_\gc |\nabla \teta_s|^2 \dd x \dd r
+ \int_0^t \int_\gc  k(\chi)(\teta-\teta_s)^2 \dd x \dd r\\ & \quad
+
\int_0^t\int_{\Gamma_c}\fc'(\teta-\vartheta_s)(\teta-\teta_s)|\Reg
(\phi_\eps'(\uun){\bf n})| |\dotut| \dd x \dd r
+ \int_0^t b(\partial_t\mathbf{u},\partial_t\mathbf{u}) \dd r \\ & \quad +\frac12 a(\uu(t),\uu(t))
+\frac12
\int_\gc \chi(t) |\uu(t)|^2 \dd x
+\int_0^t \int_{\Gamma_c}
\fc(\teta-\vartheta_s)   {\mmu}\cdot {\partial_t\mathbf{u}} \dd x \dd r
\\ & \quad
+\int_{\Gamma_c}\phi_\eps(\uun(t))\dd x+
\int_0^t \int_\gc |\partial_t\chi|^2  \dd x \dd r +\frac12 \int_\gc |\nabla
\chi(t)|^2 \dd x
\\ &
 \leq C +\frac12 a(\uu_0,\uu_0)+ \frac12 \int_\gc \chi_0 |\uu_0|^2
\dd x +\int_{\Gamma_c}\phi_\eps(\uun(0))\dd x+\frac12 \| \nabla \chi_0 \|_{\Hc}^2+I_1+I_2+I_3+I_4.
\end{aligned}
\end{equation}
Now, observe that, due   to the positivity of $k$ in \eqref{hyp-k},  the seventh integral term on the
left-hand side is positive, and so are the eighth term, thanks to
\eqref{hyp-fc}, the eleventh, since $\chi \in \dom(\widehat
\beta)\subset [0,+\infty)$,  and the twelfth, by \eqref{hyp-fc} and the
fact that $\mmu \cdot \partial_t\mathbf{u} \geq 0$ a.e.\ in $\Gamma_c \times (0,T)$.
As for the right-hand side of \eqref{calc-1-app}, it holds
 $\int_{\Gamma_c}\phi_\eps(\uun(0))\dd x
\leq \bvarphi(\uu_0)<\infty$ by \eqref{cond-uu-zero}. Moreover, we have that
\[
\begin{aligned}
&
 I_1 = \int_0^t \pairing{}{\bsW}{\mathbf{F}}{\partial_t\mathbf{u}} \dd r \leq
 \varrho \int_0^t \| \partial_t\mathbf{u}  \|_{\bsW}^2 \dd r+
 C \|\mathbf{F} \|_{L^2(0,T;\bsW')}^2
 \leq \frac{\varrho}{C_b} \int_0^t  b(\partial_t\mathbf{u},\partial_t\mathbf{u}) \dd r+
 C \|\mathbf{F} \|_{L^2(0,T;\bsW')}^2
\\
&
\begin{aligned}
 I_2 = \int_0^t \pairing{}{V}{h}{\teta} \dd r & = \int_0^t
\pairing{}{V}{h}{\teta-m(\teta)} \dd r+ \int_0^t \int_\Omega h
m(\teta) \dd x \dd r \\ & \leq \frac12 \int_0^t \int_\Omega |\nabla
\teta|^2 \dd x \dd r + C  \int_0^t \| h \|_{H} \|
\teta\|_{L^1(\Omega)} \dd r+ C \int_0^t \|h \|_{V'}^2 \dd r,
\end{aligned}
 \end{aligned}
\]
where the second inequality in the first line  is due to
\eqref{korn_b}, and we choose $\varrho= C_b/2$ in order to absorb
the term $\int_0^t  b(\partial_t\mathbf{u},\partial_t\mathbf{u}) \dd r$ into the corresponding
term on the left-hand side.
 In the estimate for $I_2$,
 the second passage
follows from Poincar\'{e}'s inequality.
Finally, exploiting the convexity of $\widehat\beta$ and the fact that  $\sigma'$ has at most a  linear growth,
(cf.\ also \eqref{useful-estimates} and \eqref{plugged-later}), we have
\[
\begin{aligned}
&
I_3= -  \int_{\Gamma_c}\widehat{\beta}(\chi(t)) \dd x
\leq  C
\int_{\Gamma_c}|\chi(t)| \dd x+C' \leq \frac14\int_0^t
\int_{\Gamma_c}|\partial_t\chi|^2 \dd x \dd s +C \| \chi_0 \|_{L^1(\Omega)} +C'.
\\
&
I_4 = - \int_0^t
\int_{\Gamma_c}\sigma'(\chi)\partial_t \chi \dd x \dd s \leq
\frac18\int_0^t
\int_{\Gamma_c}|\partial_t\chi|^2 \dd x \dd s  +C \int_0^t \| |\partial_t\chi\|_{L^2(0,s;\hc)}^2 \dd s
  +C' \| \chi_0 \|_{\Hc}^2+C'.
\end{aligned}
\]
 We plug  the above calculations
 into the right-hand side of \eqref{calc-1-app}.
Relying on
assumptions \eqref{hypo-h}  for
$h$, \eqref{effegrande} for $\mathbf{F}$,
and on
\eqref{cond-uu-zero}--\eqref{cond-chi-zero} for the data $\uu_0$ and $\chi_0$,
applying the Gronwall Lemma  we immediately deduce estimates
\eqref{est-glob-1}--\eqref{est-glob-3}.
\paragraph{Second a priori estimate.}
It follows from estimate \eqref{est-glob-3}  and the continuous embeddings
\eqref{cont-embe} that the term $\chi \uu $ on the left-hand side of
\eqref{eqIa-app} is bounded in $L^2(0,T;L^{4-\epsilon}(\gc;\R^3))$ for
every $\epsilon \in (0,3]$. Therefore, taking into account the
previously obtained estimates on $\uu$ and $\teta$, and arguing  by
comparison \eqref{eqIa-app}, we also obtain the bound
\begin{equation}
\label{comparison-1-app} \| \fc(\teta-\teta_s)\mmu + \phi_\eps'(\uun){\bf n}  \|_{L^2
(0,T;\bsY_{{\Gamma_c}}')} \leq C.
\end{equation}
Exploiting the fact that  $\mmu$ and $\phi_\eps'(\uun){\bf n} $ are
\emph{orthogonal} and arguing as in \cite[Sec. 4]{bbr5}, from
\eqref{comparison-1-app} we conclude that
\begin{equation}
\label{second-aprio-app}
 \| \fc(\teta-\teta_s)\mmu\|_{L^2
(0,T;\bsY_{{\Gamma_c}}')} + \| \phi_\eps'(\uun){\bf n}  \|_{L^2
(0,T;\bsY_{{\Gamma_c}}')} \leq C,
\end{equation}
whence \eqref{est-glob-4}: the estimate for $\mmu= |\Reg(\phi_\eps'(\uun){\bf n})|  \zz$ follows
 from
 the fact that $ \Reg: L^2 (0,T;\bsY_{{\Gamma_c}}') \to L^\infty
(0,T;L^{2+\nu}(\gc;\R^3)) $   is bounded thanks to
  \eqref{hyp-r-1},  and from the fact that
  \begin{equation}
  \label{trivial-for-z}
  |\zz|\leq 1 \qquad \text{ a.e.\ on
 $\gc\times (0,T)$}
 \end{equation}
by  the definition \eqref{formu-d} of ${\bf d}$.
\paragraph{Third a priori estimate.} We argue by
comparison in  the temperature equation  \eqref{teta-weak-app}.
 It follows from estimates \eqref{est-glob-2}--\eqref{est-glob-3}, the continuous embeddings
\eqref{cont-embe}, and the Lipschitz continuity \eqref{hyp-k} of
$k$,  that
 \begin{equation}
 \label{ingr-2}
\|k(\chi)(\teta-\teta_s)\|_{L^2 (0,T;L^{4-\epsilon}(\gc))} \leq
C \quad \text{for every $\epsilon \in (0,3]$.}
\end{equation}
 Now, taking into
account estimates \eqref{est-glob-3} and \eqref{est-glob-4},  again
\eqref{cont-embe}, the fact that $\Reg: L^2
(0,T;\bsY_{{\Gamma_c}}') \to L^\infty (0,T;L^{2+\nu}({{\Gamma_c}};\R^3)) $
is bounded by \eqref{hyp-r-1}, and the boundedness of $\fc'$ by
\eqref{hyp-fc}, we have at least
 \begin{equation}
 \label{ingr-1}
\| \fc'(\teta-\teta_s) |\Reg (\phi_\eps'(\uun){\bf n})|
|\dotut|\|_{L^2(0,T;L^{4/3}(\gc))} \leq C.
\end{equation}
Therefore,  in view of
\eqref{hypo-h} on $h$, by comparison in \eqref{teta-weak-app} we obtain  estimate \eqref{est-glob-6}
for $\partial_t \applog(\vartheta)$.
\paragraph{Fourth a priori estimate.}
We rely on   \eqref{ingr-2}, \eqref{ingr-1},  and estimate
\eqref{est-glob-3} which, combined with \eqref{hyp-lambda} for $\lambda$, in
particular yields
\begin{equation}
\label{ingr-3} \|\partial_t \lambda(\chi)
\|_{L^2(0,T;L^{3/2}(\gc))}\,.
\end{equation}
Therefore, a comparison in  the temperature equation   \eqref{teta-s-weak-app} yields the second of \eqref{est-glob-6}.
\paragraph{Fifth a priori estimate.} Let us
\emph{formally}
rewrite \eqref{teta-weak} as
\begin{equation}
\label{e:quasi-formal}
\begin{aligned}
\!\!\!\!\!\!\!  \int_{\Omega} \applog'(\teta)\partial_t \teta \cdot
v \dd x & =\int_{\Omega} \dive(\partial_t\mathbf{u}) \, v \dd x  -\int_{\Omega}
\nabla \vartheta \, \nabla v \dd x  - \int_{\Gamma_c} k(\chi)
(\vartheta-\vartheta_s) v \dd x
\\ & \  - \int_{\Gamma_c} \fc'(\teta-\teta_s) |\Reg (\phi_\eps'(\uun){\bf n})| |\dotut|
v \dd x
 + \pairing{}{V}{h}{v} \quad
\forall\, v \in V \ \hbox{ a.e. in }\, (0,T)\,
\end{aligned}
\end{equation}
and choose in \eqref{e:quasi-formal} a test function $v \in V$ of the form
\begin{equation}
\label{choice-test} v=\frac{1}{\applog'(\teta)} w, \qquad \text{with } w \in
W^{1,q}(\Omega) \text{ and } q>3.
\end{equation}
Taking  the contraction property \eqref{Lip}  of $\frac{1}{\applog'}$
into account and considering that $\teta \in \V \subset
L^6(\Omega)$, we have that
$\frac{1}{\applog'(\teta)} \in \V$ and therefore
 $\frac{1}{\applog'(\teta)} w \in \V$.
 Furthermore, it follows from \eqref{Lip} that
  \begin{equation}
 \label{to-quote-lateron-bis}
 \left|\frac{1}{\applog'(\teta)} \right | \leq \frac1{\eps + \ln_\eps'(1)} + |\teta-1| = \frac{\rho_\eps(1) + \eps}{1+ \eps\rho_\eps(1) + \eps^2}
 + |\teta-1| \leq |\teta| +2 \qquad \aein\, \Omega \times (0,T),
\end{equation}
 where we have also used formula \eqref{very-useful-formula} involving the resolvent $\rho_\eps$ of $\ln$, which satisfies
 $\rho_\eps(1)=1$.
 Therefore  again   exploiting \eqref{Lip}  we find
 \begin{equation}
 \label{to-quote-lateron}
 \left\|\frac{1}{\applog'(\teta)} \right \|_H \leq   \|\teta\|_H +c, \qquad \left\|\nabla \left( \frac{1}{\applog'(\teta)} \right) \right \|_H\leq \| \nabla \teta \|_H\,.
 \end{equation}

 Now, we have
\begin{equation}
\label{l2-product} \int_{\Omega} \applog'(\teta) \partial_t \teta \cdot
\left(\frac{1}{\applog'(\teta)} w\right)\dd x = \int_{\Omega} \partial_t \teta w \dd x .
\end{equation}
Moreover, in view of  \eqref{Lip}, \eqref{to-quote-lateron-bis},  \eqref{to-quote-lateron}, the previously obtained  estimates,
as well as
\eqref{hypo-h}  on $h$,  we see that
\[
\begin{aligned}
& \left|\int_{\Omega} \dive(\partial_t\mathbf{u}) \frac{1}{\applog'(\teta)} w  \dd x \right|  \leq \|
\partial_t\mathbf{u}\|_{\bsW} \| ( \| \teta \|_{H} +c) \| w \|_{L^\infty(\Omega)} \doteq
f_1 \in L^2(0,T),
\\
&   \left| \int_{\Omega} \nabla \vartheta \, \nabla \left(\frac{1}{\applog'(\teta)} w\right) \dd x
\right| \leq \| \nabla \teta\|_{H}^2 \| w \|_{L^\infty(\Omega)} +( \|
\teta\|_{L^6(\Omega)}+c)  \| \nabla \teta\|_{H} \| \nabla w
\|_{L^q(\Omega)} \doteq f_2 \in L^1(0,T),
\\
& \left| \int_{\Gamma_c} k(\chi) (\vartheta-\vartheta_s) \frac{1}{\applog'(\teta)} w \dd
x \right| \leq \|k(\chi)(\teta-\teta_s)\|_{\Hc}  ( \| \teta
\|_{L^4(\gc)} +c) \| w \|_{L^4 (\gc)} \doteq f_3 \in L^2(0,T),
\\ &
\begin{aligned} &
\left|\int_{\Gamma_c} \fc'(\teta-\teta_s) |\Reg (\phi_\eps'(\uun){\bf n})| |\dotut|
\frac{1}{\applog'(\teta)}  w \dd x  \right|  \\ & \quad \leq \| \fc'(\teta-\teta_s) |\Reg (\phi_\eps'(\uun){\bf n})||
\dotut|\|_{L^{4/3}(\gc)}  (\| \teta\|_{L^4(\gc)}+1) \| w
\|_{L^\infty(\gc)} \doteq f_4 \in L^1(0,T),
\end{aligned}
\\
 &
 \begin{aligned}
  \left \langle h, {\frac{1}{\applog'(\teta)} w} \right\rangle_\V  & \leq \| h \|_{V'}\left \|  \frac{1}{\applog'(\teta)}  w \right \|_{\V}
  \\ &
   \leq
 \| h \|_{V'} \left( ( \| \teta\|_V +c) \| w \|_{L^\infty(\Omega)} + (\|
 \teta \|_{L^6(\Omega)} +c)\| \nabla w \|_{L^3(\Omega)} \right) \doteq f_5 \in
 L^1(0,T),
 \end{aligned}
\end{aligned}
\]
where we have also  used the continuous embeddings \eqref{cont-embe},
 $W^{1,q}(\Omega) \subset L^\infty (\Omega) $, $V\subset
L^6(\Omega)$, as well as the trace result $W^{1,q}(\Omega) \subset
L^\infty (\gc) $. All in all, we conclude that
\begin{equation}
\label{e:stima-pulita}
\begin{aligned} &
 \exists\, f \in L^1(0,T) \  \ \foraa\, t\in
(0,T) \  \ \forall\, w \in W^{1,q}(\Omega) \, :
\\ &
 \left|\int_{\Omega}
\partial_t \teta(t) w \dd x \right|=
\left|\pairing{}{W^{1,q}(\Omega)}{\partial_t \teta (t)}{w} \right|
\leq f(t) \|w\|_{W^{1,q}(\Omega)}.\end{aligned}
\end{equation}
 Hence, we have
 \begin{equation}
 \label{partial-t-1}
 \| \partial_t \teta\|_{L^1 (0,T;W^{1,q}(\Omega)')} \leq
 C,
 \end{equation} yielding  \eqref{bv-estimate-teta}.
\paragraph{Sixth a priori estimate.}
 We proceed in an
analogous way with \eqref{teta-s-weak} and test it by
\begin{equation}
\label{choice-test-s} v=\frac{1}{\applog'(\teta_s)} w, \qquad \text{with } w \in
W^{1,\sigma}(\gc) \text{ and } \sigma>2.
\end{equation}
The analogues of  estimates \eqref{to-quote-lateron-bis} and \eqref{to-quote-lateron} hold, therefore
we obtain
\begin{equation}
\label{e:formal-s}
\begin{aligned}
    \int_{\gc}\partial_t \teta_s  w \dd x &  = \int_{\Gamma_c}
\partial_t \lambda(\chi) \frac{1}{\applog'(\teta_s)} w \dd x  -\int_{\Gamma_c} \nabla \vartheta_s  \, \nabla
\left(\frac{1}{\applog'(\teta_s)} w\right)\dd x  \\ & \quad + \int_{\Gamma_c} k(\chi) (\vartheta-\vartheta_s)
\frac{1}{\applog'(\teta_s)} w \dd x
 + \int_{\gc}  \fc'(\teta-\teta_s) |\Reg (\phi_\eps'(\uun){\bf n})| |\dotut|
 \frac{1}{\applog'(\teta_s)} w\dd x
 \\ &
 \doteq I_5+I_6+I_7+I_8.
\end{aligned}
\end{equation}
Using  that $\frac1{\applog'}$ and   $\lambda'$ are Lipschitz and relying on the continuous
embedding $W^{1,\sigma}(\gc)  \subset L^\infty(\gc)$ we have
\[
\begin{aligned}
| I_5|  & \leq  \| \lambda'(\chi)\|_{L^4(\gc)}\| \partial_t\chi\|_{L^2(\gc)}
(\|\teta_s\|_{L^4(\gc)} +c) \|w\|_{L^\infty(\gc)}
\\
& \leq C(1+ \|\chi\|_{L^4(\gc)}) \| \partial_t\chi\|_{L^2(\gc)}
(\|\teta_s\|_{L^4(\gc)}+c) \|w\|_{L^\infty(\gc)} \doteq f_6 \in L^1(0,T)
\end{aligned}
\]
in view of  estimates \eqref{est-glob-2}--\eqref{est-glob-3}. Analogously, we have
\[
\begin{aligned}
&  |I_6| \leq \| \nabla \teta_s\|_{\Hc}^2 \|w\|_{L^\infty(\gc)} +
(\|\teta_s\|_{L^{\rho}(\gc)}+c) \| \nabla \teta_s\|_{\Hc} \|\nabla w
\|_{L^\sigma(\gc)}  \doteq f_7 \in L^1(0,T)
 \\
& |I_7| \leq  (1+\| \chi\|_{\Hc}) \|
\vartheta-\vartheta_s\|_{L^4(\gc)}  ( \| \vartheta_s\|_{L^4(\gc)}+c)
\|w\|_{L^\infty(\gc)}  \doteq f_8 \in L^1(0,T),
\\
& |I_8| \leq \| \fc'(\teta-\teta_s) |\Reg (\phi_\eps'(\uun){\bf n})|
|\dotut|\|_{L^{4/3}(\gc)} (  \| \teta_s \|_{L^4(\gc)} +c) \|
w\|_{L^\infty(\gc)} \doteq f_9 \in L^1(0,T)
\end{aligned}
\]
where in the estimate of $I_6$ we choose $\rho$ in such a way
that $1/\rho+ 1/2+1/\sigma=1$, exploiting \eqref{cont-embe},
whereas to deal with $I_7$ we have  used    the fact  that $k$
is Lipschitz, and finally for $I_8$ we have relied on
\eqref{ingr-1}. All in all, we conclude that
\begin{equation}
\label{e:stima-pulita-s}
\begin{aligned} &
 \exists\, f \in L^1(0,T) \  \ \foraa\, t\in
(0,T) \  \ \forall\, w \in W^{1,\sigma}(\gc) \, :
\\ &
 \left|\int_{\gc}
\partial_t \teta_s(t) w \dd x \right|=
\left|\pairing{}{W^{1,\sigma}(\gc)}{\partial_t \teta_s (t)}{w}
\right| \leq f(t) \|w\|_{W^{1,\sigma}(\gc)}.\end{aligned}
\end{equation}
 Hence, we have
 \begin{equation}
 \label{partial-t-2}
 \| \partial_t \teta_s\|_{L^1 (0,T;W^{1,\sigma}(\gc)')} \leq
 C,
 \end{equation}
  whence \eqref{bv-estimate-teta-s}.
\paragraph{Seventh a priori estimate.} We test \eqref{eqIIa} by
$A\chi+\xi$  and integrate in time. Taking into account the chain-rule
identity \eqref{chain-rules-local-bis}, we obtain
\[
\begin{aligned}
& \frac12\int_\gc |\nabla\chi(t)|^2 \dd x +\int_\gc
\widehat{\beta}(\chi(t)) \dd x
 + \int_0^t \int_\gc
|A\chi+\xi|^2 \dd x \dd r
\\ &
 = \frac12\int_\gc |\nabla\chi_0|^2 \dd x
+\int_\gc \widehat{\beta}(\chi_0) \dd x +I_{9}+I_{10}+I_{11},
\end{aligned}
\]
and estimate
\[
\begin{aligned}
& \begin{aligned}
 I_{9}= -\int_0^t \int_\gc \sigma'(\chi) (A\chi+\xi) \dd x \dd r
 & \leq C \int_0^t (1+\| \chi \|_{\Hc}) \| A\chi+\xi\|_{\Hc} \dd r \\ & \leq
\frac14\int_0^t \| A\chi+\xi\|_{\Hc}^2 \dd r + C \int_0^t \| \chi
\|_{\Hc}^2 \dd r +C',
\end{aligned}
\\
&
\begin{aligned}
I_{10} = -\int_0^t \int_\gc \lambda'(\chi) \teta_s  (A\chi+\xi) \dd
x \dd r  & \leq C  \int_0^t (1+\| \chi \|_{L^\rho(\gc)}) \|
\teta_s \|_{L^\nu (\gc)}  \| A\chi+\xi\|_{\Hc}\dd r
\\ &
\leq C(1+ \| \chi\|_{L^\infty (0,T;\Vc)})^2 \int_0^t \| \teta_s
\|_{\Vc}^2\dd r +\frac14 \int_0^t \| A\chi+\xi\|_{\Hc}^2\dd r
\end{aligned}
\\
&  I_{11}=-\frac12 \int_0^t \int_\gc |\uu|^2 (A\chi+\xi) \dd x \dd r
\leq \frac18 \int_0^t \| A\chi+\xi\|_{\Hc}^2\dd r + C \int_0^t \|
\uu\|_{\bsW}^4 \dd r,
\end{aligned}
\]
where for $I_{9}$ we have used the Lipschitz continuity of
$\sigma'$, for $I_{10}$ chosen $\rho$ and $\nu$ in such a way
that $1/\rho+1/\nu+1/2=1$ and then used the continuous embedding
\eqref{cont-embe}, and analogously for $I_{11}$. Applying the
Gronwall Lemma and taking into account estimates \eqref{est-glob-2}--\eqref{est-glob-3},
we then conclude that
 $
\int_0^t \| A\chi+\xi\|_{\Hc}^2 \dd r \leq C. $ A monotonicity
argument  yields
\begin{equation}
\label{8th-est-app} \| A\chi\|_{L^2(0,T;\Hc)}+\|
\xi\|_{L^2(0,T;\Hc)}\leq C,
\end{equation}
whence
\eqref{est-glob-7} by
 elliptic regularity.
\paragraph{Eighth estimate.} We test \eqref{teta-weak-app} by
$\applog(\teta) \in V$ and \eqref{teta-s-weak-app} by $\applog(\teta_s)\in \Vc$, add the
resulting relations and integrate in time.
Observe that, by \eqref{bi-Lip} we have
 \[
\int_0^t
\int_\Omega \nabla \applog(\teta) \nabla \teta \dd x \dd s=
\int_0^t
\int_\Omega \applog'(\teta)|\nabla\teta|^2 \dd x \dd s \geq \eps  \int_0^t
\int_\Omega |\nabla\teta|^2 \dd x \dd s,
\]
 and analogously for the term involving
$\nabla \applog(\teta_s)$.
Therefore,
 we obtain
\begin{equation}
\label{eigth-est}
\begin{aligned}
 & \frac12 \int_{\Omega} |\applog(\teta(t))|^2 \dd x + \eps\int_0^t \int_\Omega
|\nabla \teta|^2 \dd x \dd r + \frac12 \int_{\gc}
|\applog(\teta_s(t))|^2 \dd x + \eps\int_0^t \int_\gc |\nabla
\teta_s|^2 \dd x \dd r
\\ & \quad
+ \int_0^t \int_\gc k(\chi)(\teta-\teta_s)(\applog(\teta)-\applog(\teta_s))
\dd x \dd r
\\ & \quad +\int_0^t \int_{\gc}  \fc'(\teta-\teta_s) |\Reg (\phi_\eps'(\uun){\bf n})|
|\dotut| (\applog(\teta)-\applog(\teta_s)) \dd x \dd r
\\ &
 \leq \frac12
\int_{\Omega} |\applog(\tetazeroe)|^2 \dd x+ \frac12 \int_{\gc}
|\applog(\tetazerose)|^2 \dd x +I_{12}+I_{13} +I_{14}.
 %
\end{aligned}
\end{equation}
Now,  due to the monotonicity of $\applog$ the fourth term
on the left-hand side is non-negative, and so is the fifth one, as
\[
\fc'(y-z) (\applog(y)-\applog(z))=\fc'(y-z) (y-z)\frac{\applog(y)-\applog(z)}{y-z}
\geq 0 \quad \text{for all }y \neq z
\]
also in view of \eqref{hyp-fc}. On the other hand, by the very definition
\eqref{def-applog}
 of
$\applog$,
there holds
\begin{equation}
\|\applog(\tetazeroe)\|_H^2 \leq 2 \eps^2 \|\tetazeroe\|_H^2 +2  \| \ln_\eps(\tetazeroe)\|_H^2 \leq C
\end{equation}
thanks to  \eqref{bounddatiteta-bis}-\eqref{bounddatiteta-ter},
and we have an analogous bound for  $\|\applog(\tetazerose)\|_{\hc}^2$. Moreover,
 we estimate
\[
\begin{aligned}
& I_{12} = -\int_0^t \int_\Omega
 \mathrm{div}(\partial_t\mathbf{u}) \applog(\teta) \dd x \dd r \leq \int_0^t \|
 \partial_t\mathbf{u} \|_{\bsW}^2 \dd s
 + \frac14 \int_0^t  \|\applog(\teta)\|_H^2 \dd r
 \\
 &
 I_{13}  = -\int_0^t \int_\Omega
 h \applog(\teta) \dd x \dd r \leq \int_0^t \| h \|_{H} \|\applog(\teta)\|_H
 \dd r
 \\
 &
 I_{14} =-\int_0^t \int_\gc \lambda'(\chi) \partial_t\chi   \applog(\teta_s) \dd
 x  \dd s\leq C \int_0^t (1+ \| \chi\|_{L^\infty(\gc)}) \| \partial_t\chi
 \|_{\Hc} \| \applog(\teta_s) \|_{\Hc}  \dd r.
\end{aligned}
\]
We plug the above estimates into the r.h.s.\ of \eqref{eigth-est}, and use the previously
proved bounds \eqref{est-glob-3}, \eqref{est-glob-7} (yielding  a bound for $\chi$
in $L^2(0,T;L^\infty(\gc))$), and \eqref{hypo-h}  for $h$.
Applying a generalized version of the Gronwall Lemma (see,  e.g., \cite{Baiocchi67}), we conclude
\begin{equation}
\label{seventh-aprio-app} \| \applog(\teta)\|_{L^\infty (0,T;H)}+ \|
\applog(\teta_s)\|_{L^\infty (0,T;\Hc)} \leq C
\end{equation}
whence, in view of \eqref{est-glob-1},
\begin{equation}
\label{seventh-aprio-app1} \| \ln_\eps(\teta)\|_{L^\infty (0,T;H)}+ \|
\ln_\eps(\teta_s)\|_{L^\infty (0,T;\Hc)} \leq C.
\end{equation}

\begin{remark}
\label{rem:formal}
\upshape
As previously mentioned,  the \emph{First, Fifth, Sixth estimates} should be performed on a further approximate version of Problem
($P_\eps)$. In fact,  identities \eqref{quasiformal_1-bis},
 \eqref{quasiformal_2-bis}, \eqref{e:quasi-formal}, and \eqref{e:formal-s}
 are just formal since the $\dt\applog(\teta)$  only belongs to
$L^2(0,T;V')$ (analogously,  $\dt\applog(\teta_s)$  only  belongs  to
$L^2(0,T;(H^1(\gc))')$. These calculations can be  rigorously justified in a framework where equations
\eqref{teta-weak-app} and
\eqref{teta-s-weak-app} are further regularized by   adding    viscosity contributions modulated by a second
parameter $\nu>0$, that is
 \begin{align}
&
 \label{teta-weak-app-visc}
\begin{aligned}
 &  \int_\Omega \partial_t \applog(\teta) v \dd x
   -\int_{\Omega} \dive(\partial_t\mathbf{u}) \, v  \dd x
+\int_{\Omega}   \nabla \vartheta \, \nabla v  \dd x + \nu
\int_\Omega \nabla (\partial_t\teta) \nabla v \dd x
\\ & + \int_{\Gamma_c}
k(\chi)  (\vartheta-\vartheta_s) v   \dd x+\int_{\Gamma_c}
\fc'(\teta-\teta_s) |\Reg (\phi_\eps'(\uun){\bf n})| |\dotut| v  \dd x
 = \pairing{}{V}{h}{v} \quad
\forall\, v \in V \ \hbox{ a.e. in }\, (0,T)\,,
\end{aligned}
\\
& \label{teta-s-weak-app-visc}
\begin{aligned}
 & \int_\gc \partial_t \applog(\teta_s)v \dd x  -\int_{\Gamma_c}
\partial_t \lambda(\chi) \, v   \dd x     +\int_{\Gamma_c} \nabla \vartheta_s  \, \nabla
v \dd x + \nu\int_{\Gamma_c} \nabla (\partial_t\vartheta_s)  \,
\nabla v \dd x
\\ &
 = \int_{\Gamma_c} k(\chi) (\vartheta-\vartheta_s) v\dd x
 +\int_{\Gamma_c} \fc'(\teta-\teta_s) |\Reg (\phi_\eps'(\uun){\bf n})| |\dotut| v \dd x
  \quad
\forall\, v \in \Vc  \ \hbox{ a.e. in }\, (0,T)\,,
\end{aligned}
\end{align}

The presence of   these additional viscosity   terms
in \eqref{teta-weak-app-visc} and \eqref{teta-s-weak-app-visc}
implies that the solution to the PDE  system  of Problem $(P_\eps)$,
with \eqref{teta-weak-app-visc} in place of
\eqref{teta-weak-app} and \eqref{teta-s-weak-app-visc} in place of \eqref{teta-s-weak-app},
  and supplemented by
 natural initial conditions, satisfies in addition
\begin{align}
\label{app-reg-teta} \applog(\vartheta), \teta \in H^1 (0,T; V), \qquad  \applog(\vartheta_s), \teta_s
\in  H^1 (0,T; \Vc)
\end{align}
and hence the formal identities \eqref{quasiformal_1-bis},
 \eqref{quasiformal_2-bis}, \eqref{e:quasi-formal}, and \eqref{e:formal-s} can be rigorously revised.
Such an approximation of Problem \ref{prob:rev} was considered in ~\cite{bbr3}, (see also~\cite{bbr4}),
to which we refer the reader. Let us just mention here that,
for technical reasons (see~\cite[Remark~3.2]{bbr3}) the viscosity parameter
$\nu$ has to be kept
 distinct from Yosida parameter $\eps$
for the logarithm. Hence it is necessary to derive global a priori estimates independent of $\epsi$ and/or $\nu$, and
then perform the passage to the limit procedure in two steps, first as $\nu \down 0$
and subsequently as $\eps \down 0$.

  To avoid overburdening the paper, we have preferred to
 omit this further vanishing viscosity regularization, at the price of
   developing the calculations for the \emph{First, Fifth}, and  \emph{Sixth estimates} only on a formal level.
\end{remark}


\section{Proof of Theorem \ref{mainth:1}}

\label{s:3.3}
In this section, we detail the passage to the limit
in the approximate Problem $(P_\eps)$ as $\varepsilon$ tends to $0$ and we achieve the proof of
 Theorem \ref{mainth:1} showing that the approximate solutions converge (up to a subsequence) to a solution of Problem \ref{prob:rev}.
Hereafter we make explicit the dependence of the approximate solutions on
 the parameter
 $\eps$ and use the place-holder
\[
\eeta_\eps : =\phi_\eps'(\uune){\bf n}\,.
\]
We split the proof in some steps.

\paragraph{Compactness.}
Combining estimates
\eqref{est-glob-1}--\eqref{est-glob-6}, \eqref{est-glob-7}--\eqref{est-glob-8}, and \eqref{trivial-for-z}
with the Ascoli-Arzel\`{a}
theorem, the well-known \cite[Thm.~4, Cor.~5]{Simon87}, and
standard weak and   weak$^*$-compactness  results, we find that there exists an
nine-uple $(\teta,w,\teta_s,w_s,\uu,\chi, \eeta, \zz, \xi)$
such that, along a suitable (not relabeled) subsequence,
 the following convergences hold
\begin{align}
\label{convue} &
\begin{aligned}
&
\uu_\epsi\weakto\uu\quad\text{in }H^1(0,T;\bfw),\\
&\uu_\epsi\to\uu\quad\text{in }\mathrm{C}^0([0,T];H^{1-\delta}(\Omega;\R^3)) \quad
\text{for all $\delta \in (0,1]$,}
\end{aligned}\\
& \label{convchie}
\begin{aligned} &\chie\weaksto\chi\quad\text{in
} L^2(0,T;H^2(\Gamma_c)) \cap L^\infty(0,T;\Vc)\cap
H^1(0,T;\Hc),
\\
&\chie\rightarrow\chi\quad\text{in
}  L^2(0,T;H^{2-\rho}(\Gamma_c)) \cap \mathrm{C}^0([0,T];H^{1-\delta}(\Gamma_c))
 \quad \text{for all $\rho \in (0,2]$ and  $\delta \in (0,1]$},
\end{aligned}
\\
 \label{convxie}
&\xi_\epsi\weakto\xi\quad\text{in }L^2(0,T;\Hc),\\
\label{conveeta}
&\eeta_\epsi\weakto \eeta \ \hbox{ in }L^2(0,T;\bsY_\gc'),
\\
 &\zz_\epsi \weaksto \zz \ \hbox{ in }L^\infty(\gc\times (0,T);\R^3),
  \\
\label{convreg} & \mmu_\eps =|\Reg(\eeta_\epsi)|  \zz_\epsi \weaksto \mmu \  \hbox{ in
}L^\infty(0,T;L^{2+\nu}(\gc;\R^3))\qquad \text{with $\nu>0$ from (\ref{hyp-r-1}),}
\\
\label{convtetae}
 &\teta_\epsi \weakto\teta\quad\text{in }L^2(0,T;V), \qquad
 \epsi \teta_\epsi
 \to 0\quad\text{in } L^\infty(0,T;H),
 \\
 &
\label{convtetase} \teta_{s,\epsi}\weakto\teta_s\quad\text{in }
L^2(0,T;\Vc),  \qquad \epsi\teta_{s,\epsi}
 \to 0\quad\text{in }L^\infty(0,T;\Hc),
\\
& \label{convelleteta}
\begin{aligned}
& \applog(\vartheta_\epsi)\weaksto w  \quad \text{in }L^\infty (0,T;H)\cap H^1
(0,T;V')\,,
\\
& \applog(\vartheta_\epsi)\to w  \quad \text{in }\mathrm{C}^0 ([0,T];V')\,,
\end{aligned}
\\
& \label{convelletetas}
\begin{aligned}
 & \applog(\vartheta_{s,\epsi})\weaksto w_s \quad \text{in }L^\infty
(0,T;\Hc)\cap H^1 (0,T;\Vc')\,,
\\
& \applog(\vartheta_{s,\epsi}) \to w_s \quad \text{in }\mathrm{C}^0 ([0,T];\Vc')
\end{aligned}
\end{align}
as $\eps \downarrow 0$.
In addition, in view of condition
\eqref{hyp-r-1} on $\Reg$, we have
\begin{equation}\label{erree}
\Reg(\eeta_\epsi)\rightarrow\Reg(\eeta)\quad\hbox{in
}L^\infty(0,T;L^{2+\nu}(\gc;\R^3)) ,
 \ \text{so that } \
\mmu=|\Reg (\eeta)|\zz,
\end{equation}
and \eqref{convreg}   improves to
\begin{equation}\label{erreeze}
|\Reg(\eeta_\epsi)|\zz_\epsi\weak^*|\Reg(\eeta)|\zz\quad\hbox{in
}  L^\infty(0,T;L^{2+\nu}(\gc;\R^3))\,.
\end{equation}
Next, applying a generalized version of the Aubin-Lions theorem for the case of
time derivatives as measures (see, e.g., \cite[Chap.\ 7, Cor.\ 7.9]{roub-book}), from
\eqref{est-glob-2},  and estimates
\eqref{partial-t-1} and \eqref{partial-t-2}
 we deduce that
\begin{align}
& \label{convforteta}
\begin{aligned}
&\teta_\epsi\rightarrow \teta \quad\hbox{in
}L^2(0,T;H^{1-\delta}(\Omega))\quad\hbox{for all } \delta \in (0,1],
\\
&\teta_\epsi\rightarrow \teta \quad\hbox{in
}L^2(0,T;L^{\delta}(\gc))\quad\hbox{for all } \delta \in [1,4),
\end{aligned}
\\
&
\label{convfortetas}
\teta_{s,\epsi}\rightarrow \teta_s \quad\hbox{in
}L^2(0,T;L^{\delta}(\gc))\quad\hbox{for all } \delta \in [1,+\infty).
\end{align}
Moreover, taking into account the Lipschitz continuity and the $\mathrm{C}^1$- regularity of $\fc$ (cf. \eqref{hyp-fc}), from
 \eqref{convforteta}--\eqref{convfortetas}, we have
\begin{align}
\label{convc}
\begin{aligned}
&\fc (\teta_\epsi-\teta_{s,\epsi})\rightarrow\fc (\teta-\teta_{s})\quad\hbox{in
}L^2(0,T;L^{\delta}(\gc))\quad\hbox{for all } \delta \in [1,4)\,,\\
&\fc' (\teta_\epsi-\teta_{s,\epsi})\rightarrow\fc' (\teta-\teta_{s})\quad\hbox{in
}L^q(0,T;L^{q}(\gc))\quad\hbox{for all } q \in  [1,\infty) \,.\\
\end{aligned}
\end{align}
\paragraph{Passage to the limit in \eqref{eqIIa}.}
Now, we consider \eqref{eqIIa}--\eqref{inclvincolo} written for the approximate solutions $(\uu_\eps,\chi_\eps,\teta_{s,\epsi},\xi_\eps)_\eps$.
Taking into account
convergences \eqref{convue}--\eqref{convxie}, \eqref{convfortetas}, and the Lipschitz continuity of $\lambda'$ and $\sigma'$ (cf. \eqref{hyp-lambda}, \eqref{hyp-sig}),
we easily conclude that the  limit quadruple $(\uu,\chi,\teta_{s},\xi)$ satisfies  equation
\eqref{eqIIa}. Combining the weak convergence \eqref{convxie} with the strong one specified in \eqref{convchie},
and taking into account the strong-weak closedness in $L^2 (0,T;\Hc)$ of the graph of (the operator induced by)
$\beta$, we conclude that $\xi \in \beta(\chi)$
a.e. on $\gc \times (0,T)$, i.e. \eqref{inclvincolo} holds.
\paragraph{Passage to the limit in \eqref{eqIa-app}.}
Owing to   convergences  \eqref{convue}--\eqref{convchie}, \eqref{conveeta}, \eqref{convreg}--\eqref{convtetae} and  \eqref{convc},
we can pass to the limit in \eqref{eqIa-app}. We get
\begin{equation}
\label{passed-equation}
b(\partial_t\uu,\vv)+a(\uu,\vv)+ \int_{\Omega} \vartheta \dive (\vv)\dd x
+\int_{\gc}\chi\uu\vv \dd x
+\langle \eeta, \vv \rangle_{\bsY_\gc}+\int_{\gc} \fc(\teta-\vartheta_s)  \mmu\cdot\vv \dd s=\pairing{}{\bsW}{\bf
F}\vv,
\end{equation}
 for all $\vv \in \bsW$.
Now we have to identify $\eeta$ and $\mmu$ as elements of $\partial \bvarphi(\uu)$ and
$|\Reg(\eeta)|  {\bf d}(\dotu) $, respectively,
 i.e.\ to show that
\eqref{incl1} and \eqref{incl1-bis} hold.

First,
we test \eqref{eqIa-app} by $\uu_\eps$.
For every $t \in [0,T]$  we have
\[
\begin{aligned}&
\limsup_{\eps \to 0} \int_0^t \int_\gc \eeta_\eps \cdot \uu_\eps \dd x \dd s
\\
&
  = -  \liminf_{\eps \to 0}
\Big(   b(\uu_\eps(t),\uu_\eps(t)) - b(\uu_0,\uu_0)    + \int_0^t
\big( a(\uu_\eps,\uu_\eps) +
   \int_{\Omega} \vartheta_\eps \dive (\uu_\eps)\dd x\big) \dd s
\\ &  \quad + \int_0^t \big(  \int_{\Gamma_c} \chi_\eps |\uu_\eps |^2 \dd x
 +  \int_{{{\Gamma_c}}}\fc(\teta_\eps-\vartheta_{s,\eps})   {\mmu_\eps}\cdot {\uu_\eps}\dd x\big) \dd s
  - \int_0^t  \pairing{}{\bsW}{\mathbf{F}}{\uu_\eps}
\Big)
\\
& \leq
- \int_0^t \left(   b(\partial_t\uu,\uu)+a(\uu,\uu)+ \int_{\Omega} \vartheta \dive (\uu)\dd x
+\int_{\gc}\chi|\uu|^2 \dd x
+ \int_{\gc} \fc(\teta-\vartheta_s)  \mmu\cdot\uu -  \pairing{}{\bsW}{\bf
F}\uu \right) \dd s
\\ &
=
\int_0^t \int_\gc \eeta\cdot \uu \dd x \dd s
\end{aligned}
\]
where the  $\leq$ follows from
 exploiting  \eqref{convue}, \eqref{convchie}, \eqref{conveeta}, \eqref{convreg}, \eqref{convtetae},
and \eqref{convc},
combined with lower semicontinuity arguments,   and the last equality is due to \eqref{passed-equation}.
We use the above inequality
   and
 to show that  for all $\vv \in \bsY_\gc $ and $t \in [0,T]$ there holds
\[
\begin{aligned}
\int_0^t\langle {\eeta},{\vv-\uu} \rangle_{\bsY_\gc} \dd s \leq  \liminf_{\eps \to 0} \int_0^t
\int_\gc \eeta_\eps \cdot( \vv-\uu_\eps) \dd x \dd s  & \leq   \liminf_{\eps \to 0} \int_0^t
\int_\gc \left( \phi_\eps (v_{\mathrm{N}} ) -  \phi_\eps (\uune)  \right)  \dd x \dd s \\ &
\leq \int_0^t
\int_\gc \left( \phi(v_{\mathrm{N}} ) - \phi(\uun)  \right)  \dd x \dd s
\\ & =
 \int_0^t \left( \bvarphi(\vv)-\bvarphi(\uu) \right) \dd s,
\end{aligned}
\]
where the second inequality follows from the fact that  $\eeta_\eps  =  \phi_\eps'(\uune){\bf n}$,
 the third one from
the Mosco-convergence (see, e.g., \cite{attouch}) of $\phi_\eps$ to $\phi$, and the last one from the definition \eqref{funct-phi}
of $\bvarphi$. All in all, we conclude \eqref{incl1}.

  Let us now show \eqref{incl1-bis}. 
Preliminarily, for every fixed $\teta \in L^2 (0,T;V)$, $\teta_s \in L^2 (0,T;\Vc)$, and
$\eeta \in L^2 (0,T;\bsY_{{\Gamma_c}}')$, we introduce the
functional $\funeta{(\teta,\teta_s,\eeta)}: L^2 (0,T;L^4(\gc;\R^3)) \to
[0,+\infty)$ defined for all $\vv \in L^2 (0,T;L^4(\gc;\R^3))$
by
\[
\begin{aligned}
 \funeta{(\teta,\teta_s,\eeta)}(\vv):  &   = \int_0^T \int_{{\Gamma_c}} \fc(\teta(x,t)-\vartheta_s(x,t))
|\Reg(\eeta)(x,t)| j(\vv(x,t)) \dd x \dd t \nonumber\\
&= \int_0^T
\int_{{\Gamma_c}} \fc(\teta(x,t)-\vartheta_s(x,t))|\Reg(\eeta)(x,t)| | \vvt (x,t)| \dd x \dd t.
\end{aligned}
\]
Clearly, $\funeta{(\teta,\teta_s,\eeta)}$ is a convex and lower semicontinuous
functional on $L^2 (0,T;L^4(\gc;\R^3))$.
 It can be easily verified
 that the subdifferential $\partial\funeta{(\teta,\teta_s,\eeta)} : L^2 (0,T;L^4(\gc;\R^3)) \rightrightarrows L^2 (0,T;L^{4/3}(\gc;\R^3))
 $ of $\funeta{(\teta,\teta_s,\eeta)}$ is given  at every $\vv \in L^2
 (0,T;L^4(\gc;\R^3))$ by
\begin{equation}
\label{repre-funeta} {\bf h} \in \partial\funeta{(\teta,\teta_s,\eeta)}(\vv) \
\Leftrightarrow \quad \begin{cases} {\bf h} \in  L^2
 (0,T;L^{4/3}(\gc;\R^3)),  \\ {\bf h}(x,t) \in \fc(\teta(x,t)-\vartheta_s(x,t))|\Reg(\eeta)(x,t)|
 \mathbf{d}(\vv(x,t))\end{cases}
\end{equation}
for almost all $(x,t) \in {\Gamma_c} \times (0,T)$, where
$\mathbf{d} =\partial j$ is given by \eqref{formu-d}.
We shall
prove that
\begin{equation}
\label{prove} \funeta{(\teta,\teta_s,\eeta)}(\ww) - \funeta{(\teta,\teta_s,\eeta)}(\partial_t\uu)
\geq \int_0^T \int_{\gc} \fc(\teta-\vartheta_s)|\Reg(\eeta) | \zz \cdot(\ww -\dotu) \dd x \dd t
\end{equation}
for all $\ww \in L^2 (0,T;L^4(\gc;\R^3))$.  From \eqref{prove}
we will conclude that
$\fc(\teta-\vartheta_s)|\Reg(\eeta)|\zz\in \partial\funeta{(\teta,\teta_s,\eeta)}(\partial_t\uu)$, hence
the desired
\eqref{incl1-bis} by \eqref{repre-funeta}, the strict positivity \eqref{hyp-fc}
of $\fc$,
 and \eqref{erree}.  In order to show \eqref{prove},
we first observe that
\begin{equation}
\label{ehsicivuole} \limsup_{\eps\to 0} \int_0^T \int_{\gc}
\fc (\teta_\epsi-\teta_{s,\epsi})|\Reg(\eeta_\eps)| \zz_ \eps   \cdot \partial_t \uu_\eps \dd x \dd t
\leq  \int_0^T \int_{\gc} \fc (\teta-\teta_{s})|\Reg(\eeta) | \zz \cdot \partial_t \uu \dd x \dd t,
\end{equation}
which can be  checked   by testing \eqref{eqIa-app} by
$\partial_t\uu_\eps$ and passing to the limit via convergences
\eqref{convue}--\eqref{convchie}, \eqref{conveeta}--\eqref{convreg},  \eqref{erree},
  \eqref{convforteta},  lower semicontinuity arguments, and again the Mosco convergence of $\phi_\eps$.
Therefore, we have
\begin{align}\label{prove1}
\int_0^T \int_{\gc} \fc (\teta-\teta_{s})|\Reg(\eeta) | \zz \cdot(\ww -\dotu)  \dd x \dd t & \leq
\liminf_{\eps \to 0} \int_0^T \int_{\gc}
\fc (\teta_\epsi-\teta_{s,\epsi})|\Reg(\eeta_\eps)| \zz_\eps   \cdot(\ww
-\partial_t\uu_\eps)   \dd x \dd t  \nonumber\\   &  \leq \liminf_{\eps \to 0} \int_0^T \int_{{\Gamma_c}}\fc (\teta_\epsi-\teta_{s,\epsi})
|\Reg(\eeta_\eps)| ( | \wwt |- |(\partial_t \uu_\eps)_{\mathrm{T}} |)  \dd x \dd t   \nonumber\\ & \leq \int_0^T \int_{{\Gamma_c}}
\fc (\teta-\teta_{s})|\Reg(\eeta)| (| \wwt |- |(\partial_t \uu)_{\mathrm{T}} |)  \dd x \dd t
\end{align}
where the first inequality follows from \eqref{ehsicivuole} and convergences
\eqref{erreeze} and \eqref{convc},
 the
second one  from the fact that $|\Reg(\eeta_\eps)|\zz_\eps
 \in |\Reg(\eeta_\eps)| \mathbf{d}(\partial_t\uu_\eps)
 $,
  and the last one from combining the \emph{weak} convergence \eqref{convue}
 with the \emph{strong} convergences  \eqref{erree} and \eqref{convc}.
  Then,
\eqref{prove} ensues.
Furthermore, arguing as in the derivation of \eqref{prove1}, relying on \eqref{convue}, \eqref{erree},  \eqref{convc},  and \eqref{incl1-bis},   and using that, indeed,  $|\Reg(\eeta_\eps)|\zz_\eps \cdot \partial_t\uu_\eps = |\Reg(\eeta_\eps)| |(\partial_t \uu_\eps)_{\mathrm{T}} |$
 a.e.\ in $\gc \times (0,T)$,   we deduce
\begin{equation}
\begin{aligned}
\label{ehsicivuole1}
 & \liminf_{\eps\to 0} \int_0^T \int_{\gc}
\fc (\teta_\epsi-\teta_{s,\epsi})|\Reg(\eeta_\eps)| \zz_ \eps   \cdot \partial_t \uu_\eps  \dd x \dd t \\ &   =
\liminf_{\eps\to 0} \int_0^T \int_{\gc}
\fc (\teta_\epsi-\teta_{s,\epsi})|\Reg(\eeta_\eps)| |(\partial_t \uu_\eps)_{\mathrm{T}} |  \dd x \dd t \\
&
\geq  \int_0^T \int_{\gc} \fc (\teta-\teta_{s})|\Reg(\eeta) | |(\partial_t \uu)_{\mathrm{T}} |  \dd x \dd t  =
\int_0^T \int_{\gc} \fc (\teta-\teta_{s})|\Reg(\eeta) | \zz \cdot \partial_t \uu  \dd x \dd t \,.
\end{aligned}
\end{equation}
Ultimately, from \eqref{ehsicivuole} and \eqref{ehsicivuole1} we conclude
\begin{equation}
\label{ehsicivuole2} \lim_{\eps\to 0} \int_0^T \int_{\gc}
\fc (\teta_\epsi-\teta_{s,\epsi})|\Reg(\eeta_\eps)| \zz_ \eps   \cdot \partial_t \uu_\eps   \dd x \dd t  =
\int_0^T \int_{\gc} \fc (\teta-\teta_{s})|\Reg(\eeta) | \zz \cdot \partial_t \uu   \dd x \dd t \,.
\end{equation}
Now, in addition to \eqref{convue}, we prove the following strong convergence
\begin{equation}
\label{forteu-t}
 \partial_t\uu_\eps\rightarrow
\partial_t\uu\quad\text{in }L^2(0,T;\bsW),
\end{equation}
which is crucial
in order to pass to the limit in  the frictional contribution $\displaystyle \int_{\Gamma_c}
\fc'(\teta_\eps-\teta_{s,\eps}) |\Reg (\phi_\eps'(\uu_{\eps_N}{\bf n})| |(\partial_t \uu_\eps)_T| v  \dd x$
in \eqref{teta-weak-app} and \eqref{teta-s-weak-app}.
To this aim,  we first observe that
\begin{equation}
\label{primo}
\limsup_{\eps \to 0} \int_0^T
b(\partial_t\uu_\eps,\partial_t\uu_\eps)\dd t \leq \int_0^T
b(\partial_t\uu,\partial_t\uu) \dd  t
\end{equation}
arguing in a similar way as in the derivation of \eqref{ehsicivuole}: we test \eqref{eqIa-app} by
$\partial_t\uu_\eps$ and  we pass to the limit  exploiting convergences
\eqref{convue}--\eqref{convchie}, \eqref{conveeta}--\eqref{convreg}, \eqref{convforteta}, \eqref{ehsicivuole2}
and the Mosco convergence of $\phi_\eps$.
 Since the converse inequality for the
$\liminf_{\eps \to 0}$ holds by
 the first of \eqref{convue},  we then conclude that
\[
\lim_{\eps \to 0} \int_0^T
b(\partial_t\uu_\eps,\partial_t\uu_\eps) \dd t =\int_0^T
b(\partial_t\uu,\partial_t\uu) \dd t.
\]
This gives  \eqref{forteu-t}, by the $\bf W$-ellipticity of $b$ (cf.\ \eqref{korn_b}).

\paragraph{Passage to the limit in \eqref{teta-weak-app} and \eqref{teta-s-weak-app}.}
 We pass to the limit in \eqref{teta-weak-app} and \eqref{teta-s-weak-app} relying on the above convergences
   \eqref{convue}--\eqref{convchie}, \eqref{conveeta}, \eqref{convtetae}--\eqref{erree}, \eqref{convforteta}--\eqref{convc},
  \eqref{forteu-t},
  and the following additional convergences for the nonlinear terms
   in \eqref{teta-weak-app} and \eqref{teta-s-weak-app}. Indeed,
conditions \eqref{hyp-k}--\eqref{hyp-lambda}
on $k$ and $\lambda$ and convergences  \eqref{convchie}, \eqref{convforteta}--\eqref{convfortetas}
yield
\begin{equation}
\label{termisto1}
\begin{aligned}
& k(\chi_\eps)(\teta_{\eps}-\teta_{s,\eps})\rightarrow
k(\chi)(\teta-\teta_s)\quad\text{in }L^2(0,T;\Hc),
\end{aligned}
\end{equation}
 as well as
\begin{equation}
\label{termisto2}
\begin{aligned}
&
 \lambda(\chi_\eps) \weakto \lambda (\chi) \quad \text{in } H^1
(0,T; L^{3/2}(\Gamma_c)).
 \end{aligned}
\end{equation}
\noindent
Exploiting all of the above convergences, we get
 \begin{align}
&
 \label{teta-weak-lim}
\begin{aligned}
 &  \!\!\!\!\!\!\!\! \!\!\!\!  \pairing{}{\V}{\partial_t w}{ v}
   -\int_{\Omega} \dive(\partial_t\mathbf{u}) \, v  \dd x
+\int_{\Omega}   \nabla \vartheta \, \nabla v  \dd x
\\ & \! \!\!\!\!  + \int_{\Gamma_c}
k(\chi)  (\vartheta-\vartheta_s) v   \dd x+\int_{\Gamma_c}
\fc'(\teta-\teta_s) |\Reg (\eeta)| |\dotut| v  \dd x
 = \pairing{}{V}{h}{v} \quad
\forall\, v \in V \ \hbox{ a.e. in }\, (0,T)\,,
\end{aligned}
\\
& \label{teta-s-weak-lim}
\begin{aligned}
 &  \!\!\!\!\!\!\!\! \!\!\!\!  \pairing{}{\Vc}{\partial_t w_s}{ v}  -\int_{\Gamma_c}
\partial_t \lambda(\chi) \, v   \dd x     +\int_{\Gamma_c} \nabla \vartheta_s  \, \nabla
v \dd x
\\ &
 \quad = \int_{\Gamma_c} k(\chi) (\vartheta-\vartheta_s) v\dd x
 +\int_{\Gamma_c} \fc'(\teta-\teta_s) |\Reg (\eeta)| |\dotut| v \dd x
  \quad
\forall\, v \in \Vc  \ \hbox{ a.e. in }\, (0,T)\,.
\end{aligned}
\end{align}
It remains to show that
\begin{align}
\label{inclu-elle} & w(x,t)= \ln
(\teta(x,t)) \quad  \hbox{ for a.a. } (x,t) \in \Omega \times (0,T), \\  &\label{inclu-elle-s}
w_s(x,t) = \ln (\teta_s(x,t)) \quad \hbox{ for a.a. } (x,t) \in
\Gamma_c\times (0,T)\,.
\end{align}
We argue just for \eqref{inclu-elle}, the procedure for \eqref{inclu-elle-s} being completely analogous.
First, recalling the definition of $\applog$ (cf. \eqref{def-applog}), we observe that
\eqref{convelleteta} and the second of \eqref{convtetae} give
\begin{equation}\label{convelle}
\ln_\eps(\vartheta_\epsi)\weaksto w  \quad \text{in }L^\infty (0,T;H)\,.
\end{equation}
Thus, by  relying on well-known properties of Yosida regularizations  (cf.\ \cite[Lemma~1.3, p.~42]{barbu76}),
to conclude \eqref{inclu-elle} it is sufficient to check  that
\begin{equation}
\label{e:limsup} \limsup_{\eps \searrow 0} \int_0^T \int_{\Omega}\ln_\eps (\teta_\eps) \teta_\eps \dd  x \dd t
\leq \int_0^T \int_{\Omega} w \teta \dd  x \dd t\,.
\end{equation}
The latter follows combining the weak convergence \eqref{convelle} with the strong convergence \eqref{convforteta}.
This concludes the proof. \fin



\begin{thebibliography}{99}


\bibitem{akrs2002}
A.\  Amassad,   K. L.\
 Kuttler,  M.\ Rochdi, and M.\
Shillor.
\newblock  Quasi-static thermoviscoelastic contact problem with slip dependent
   friction coefficient. \newblock
\emph{Math. Comput. Modelling},
36, 839--854,  2002.


\bibitem{and-kut-shil97}
  K. T. \  Andrews,  K. L.\
 Kuttler, and M.\
Shillor. \newblock  On the dynamic behaviour of a thermoviscoelastic body in frictional
   contact with a rigid obstacle. \newblock
\emph{European J. Appl. Math.}, 8, 417--436, 1997.

\bibitem{Andrews}
K.T. Andrews, M. Shillor, S. Wright, A. Klarbring.
\newblock A dynamic thermoviscoelastic contact problem with friction and
wear.
\newblock {\em Internat. J. Engrg. Sci.}, 14, 1291--1309,
1997.

\bibitem{attouch}
H.~Attouch.
 \newblock {\em Variational
Convergence for Functions and Operators}.
\newblock Pitman, London, 1984.


\bibitem{barbu76}
V.\ Barbu.
\newblock {\em Nonlinear Semigroups and Differential Equations in Banach
  Spaces}.
\newblock Noordhoff, Leyden, 1976.

\bibitem{Baiocchi67} C.~Baiocchi.
\newblock Sulle equazioni differenziali astratte lineari del primo e del secondo ordine negli spazi di {H}ilbert.
\newblock {\em Ann. Mat. Pura Appl. (4)}, 76, 233--304, 1967.



\bibitem{bbr1}
E.~Bonetti, G.~Bonfanti, and R.~Rossi. \newblock Global existence
for a contact problem with adhesion.
\newblock {\em Math. Meth. Appl. Sci.}, 31, 1029--1064, 2008.


\bibitem{bbr2}
E. Bonetti, G.~Bonfanti, and R.~Rossi. \newblock Well-posedness and
long-time behaviour for a model of contact with adhesion.
\newblock {\em  Indiana Univ. Math. J.},  56, 2787--2819, 2007.

\bibitem{bbr3} E. Bonetti, G. Bonfanti, and R. Rossi. \newblock
Thermal effects in adhesive contact: modelling and analysis.
\newblock {\em  Nonlinearity.},  22, 2697--2731, 2009.


\bibitem{bbr4} E. Bonetti, G. Bonfanti, and R. Rossi. \newblock
Long-time behaviour of a thermomechanical model for adhesive
contact.
\newblock {\em
Discrete Contin. Dyn. Syst. Ser. S.},  4,  273--309, 2011.

\bibitem{bbr5} E. Bonetti, G. Bonfanti, and R. Rossi. \newblock
 Analysis of a unilateral contact problem taking into account adhesion and
 friction. \newblock {\em J. Differential Equations},  253,  438--462, 2012.

\bibitem{bbr7} E. Bonetti, G. Bonfanti, and R. Rossi. \newblock
Thermal effects in adhesive contact with friction: the irreversible case.
 \newblock In preparation, 2012.



\bibitem{bcfg1}
\newblock
E. Bonetti, P. Colli, M. Fabrizio and G. Gilardi. \newblock
Global solution to a singular integrodifferential system
related to the entropy balance.  \newblock \emph{Nonlinear Anal.},
66,  1949--1979, 2007.


\bibitem{bcfg2}
\newblock
 E. Bonetti, P. Colli, M. Fabrizio and G. Gilardi.
\newblock
{Modelling and long-time behaviour for phase transitions with
entropy balance and thermal memory conductivity}.
\newblock
\emph{Discrete Contin. Dyn. Syst. Ser. B}, 6, 1001--1026, 2006.


\bibitem{BFR}
E. Bonetti, M. Fr\'emond, and E.~Rocca.
\newblock A new dual approach for a class of phase transitions with
memory: existence and long-time behaviour of solutions.
\newblock  {\em  J. Math. Pures Appl.},
88, 455--481, 2007.


\bibitem{brezis73}
H.~Br\'ezis.
\newblock {\em Op\'erateurs Maximaux Monotones et Semi-groupes de Contractions
  dans les Espaces de Hilbert}.
\newblock Number~5 in North Holland Math. Studies. North-Holland, Amsterdam,
  1973.

\bibitem{brezis-mironescu}
H.\ Br\'ezis,  and P. Mironescu.
\newblock
Gagliardo-Nirenberg, composition and products in fractional Sobolev
    spaces.
    \newblock
{\em J. Evol. Equ.},
1, 387--404, 2001.


\bibitem{colli92}
P. Colli.  \newblock
On some doubly nonlinear evolution equations in Banach spaces.  \newblock
\emph{Japan J. Indust. Appl. Math.}, 9, 181--203, 1992.


\bibitem{dibenetto-showalter}
E.~Di Benedetto and R.~E.~Showalter. \newblock Implicit degenerate
evolution equations and applications. \newblock {\em SIAM J. Math.
Anal.}, 12, 731--751, 1981.

 \bibitem{duvaut}
G.~Duvaut. \newblock \'Equilibre d'un solide \'elastique avec contact unilat\'eral et frottement de Coulomb.
\newblock {\em C. R. Acad. Sci. Paris S\'er. A-B.}, 290, 263--265, 1980.


\bibitem{Eck2002}
C.\ Eck. \newblock Existence of solutions to a thermo-viscoelastic contact problem with
   Coulomb friction. \newblock
\emph{Math. Models Methods Appl. Sci.},
 12, 1491--1511, 2002.


\bibitem{eck05}
C. Eck, J. Jaru{\v{s}}ek, and  M. Krbec.
      \newblock \emph{Unilateral contact problems}. \newblock
Pure and Applied Mathematics (Boca Raton) Vol. 270,
  Chapman \& Hall/CRC, Boca Raton, Florida, 2005.



\bibitem{Eck-Jarusek}
C. Eck and  J. Ja\v{r}usek.
\newblock
 The solvability of a coupled thermoviscoelastic contact problem with small Coulomb friction
 and linearized growth of frictional heat.
  Math. Meth. Appl. Sci., 22, 1221--1234, 1999.



\bibitem{Figue-Trabucho}
I. Figueiredo and  L. Trabucho.
\newblock A class of contact and friction dynamic problems in
thermoelasticity and in thermoviscoelasticity.
\newblock {\em Internat. J. Engrg. Sci.}, 1, 45--66, 1995.



\bibitem{fremond-nedjar}
M. Fr\'emond and B. Nedjar. \newblock Damage, gradient of damage and
priciple of virtual power.
\newblock {\em Internat. J. Solids Structures}, 33, 1083-1103,
1996.


\bibitem{fre}
M.~Fr\'emond.
\newblock {\em Non-smooth Thermomechanics}.
\newblock Springer-Verlag, Berlin, 2002.

\bibitem{gilbarg-trudinger}
D.\  Gilbarg and N.\ Trudinger. \newblock \emph{Elliptic partial differential equations of second order} (Second edition).
\newblock Springer-Verlag, Berlin, 1983.

\bibitem{rochdi-shillor2000}
M.\ Rochdi and M.\ Shillor.
\newblock Existence and uniqueness for a quasistatic frictional bilateral contact
   problem in thermoviscoelasticity.
\newblock \emph{Quart. Appl. Math.},
58,  543--560,  2000.

\bibitem{roub-book}
T.  Roub\'\i\v cek.
\newblock {\em Nonlinear partial differential equations with applications}.
\newblock
International Series of Numerical Mathematics, 153. Birkh\"auser Verlag, Basel, 2005.




\bibitem{Simon87}
J.~Simon.
\newblock Compact sets in the space {$L^p(0,T; B)$}.
\newblock {\em Ann. Mat. Pura Appl. (4)}, 146, 65--96, 1987.



\bibitem{sofonea-han-shillor}
M. Sofonea, W. Han, and M. Shillor. \newblock \emph{Analysis and
approximation of contact problems with adhesion or damage.} \newblock
Pure and Applied Mathematics (Boca Raton), 276. Chapman \& Hall/CRC,
Boca Raton, FL, 2006.


\end{thebibliography}
\end{document}